\documentclass[a4paper]{amsart}
\usepackage[ps2pdf]{hyperref}
\usepackage{amsmath,amssymb}
\usepackage{bbold}
\usepackage{psfrag}
\usepackage{graphicx}
\usepackage{subfigure}
\usepackage{leftidx}
\usepackage[all]{xy}
\usepackage[latin1]{inputenc}
\usepackage[dvipsnames]{xcolor}

\newtheorem{thm}{Theorem}[section]
\newtheorem{cor}[thm]{Corollary}
\newtheorem{lem}[thm]{Lemma}

\theoremstyle{definition}

\newtheorem{exa}[thm]{Example}

\newcommand{\letterf}{z}
\newcommand{\ds}{\displaystyle}

\newcommand{\co}{\colon}
\newcommand{\id}{\mathrm{id}}

\newcommand{\opp}{\mathrm{op}}

\newcommand{\cc}{\mathcal{C}}
\newcommand{\bb}{\mathcal{B}}
\newcommand{\dd}{\mathcal{D}}

\newcommand{\zz}{\mathcal{Z}}

\newcommand{\CC}{\mathbb{C}}
             
                     \newcommand{\ZZ}{\mathbb{Z}}
                     \newcommand{\RR}{\mathbb{R}}

\newcommand{\kk}{\Bbbk}
\newcommand{\kt}{$\Bbbk$\nobreakdash-\hspace{0pt}}

\newcommand{\ti}{\mbox{-}\,}

\newcommand{\R}{\mathbb{R}}
\newcommand{\un}{\mathbb{1}}

\newcommand{\Ob}{\mathrm{Ob}}
\newcommand{\Int}{\mathrm{Int}}
\newcommand{\Reg}{\mathrm{Reg}}

\newcommand{\End}{\mathrm{End}}
\newcommand{\Hom}{\mathrm{Hom}}

\newcommand{\tr}{\mathrm{tr}}

\newcommand{\rank}{\mathrm{rank}}

\newcommand{\inv}{\mathbb{F}}

\newcommand{\lev}{\mathrm{ev}}
\newcommand{\Ev}{\mathrm{Ev}}
\newcommand{\rev}{\widetilde{\mathrm{ev}}}
\newcommand{\lcoev}{\mathrm{coev}}
\newcommand{\rcoev}{\widetilde{\mathrm{coev}}}
\newcommand{\ldual}[1]{#1^{*}}

\newcommand{\dual}[1]{{#1}^*}

\newcommand{\scaledraw}[1]{A}
\newcommand{\scaleraisedraw}[2]{A}
\newcommand{\rsdraw}[3]{\raisebox{-#1\height}{\scalebox{#2}{\includegraphics{#3.eps}}}}
\providecommand{\bysame}{\leavevmode\hbox to3em{\hrulefill}\thinspace}

\begin{document}

\title{On two approaches to 3-dimensional TQFTs}
\author[V. Turaev]{Vladimir Turaev}
\address{%
Vladimir Turaev\newline
\indent Department of Mathematics \newline
\indent Indiana University \newline
\indent Bloomington IN47405 \newline
\indent USA \newline
\indent e-mail: vtouraev@indiana.edu}
\author[A. Virelizier]{Alexis Virelizier}
\address{%
Alexis Virelizier\newline
\indent Department of Mathematics \newline
\indent University Lille 1\newline
\indent 59655 Villeneuve d'Ascq \newline
\indent France \newline
\indent e-mail:  alexis.virelizier@math.univ-lille1.fr}
\date{\today}

\begin{abstract}
We prove that $\vert M\vert_{\cc}=\tau_{\zz(\cc)} (M)$ for  any closed
oriented 3-manifold $M$ and for any spherical  fusion category~$\cc$
of non-zero dimension. Here $\vert M\vert_{\cc}$ is the
Turaev-Viro-Barrett-Westbury state sum invariant of $M$ derived from~
$\cc$, $\zz(\cc)$ is the Drinfeld-Joyal-Street center of $\cc$, and $\tau_{\zz(\cc)}
(M)$ is the Reshetikhin-Turaev surgery invariant of $M$ derived from
   $\zz(\cc)$.
\end{abstract}
\maketitle

\setcounter{tocdepth}{1} \tableofcontents

\section*{Introduction}

Two fundamental constructions of 3-dimensional topological quantum
field theories (TQFTs) are due   to Reshetikhin-Turaev \cite{RT} and
  Turaev-Viro \cite{TV}. The RT-construction is widely viewed as a
mathematical realization of Witten's Chern-Simons TQFT, see
\cite{Wi-}. The TV-construction is closely related to the
Ponzano-Regge state-sum model for   3-dimensional quantum gravity,
see \cite{Ca}. The first connections between these two constructions
were established by Walker \cite{Wa} and Turaev \cite{Tu1}. In 1995,
the first named author conjectured a more general connection between
these constructions. This connection may be formulated as the
identity $\vert M\vert_{\cc}=\tau_{\zz(\cc)} (M)$, see below for a
detailed statement. The   aim of our paper is to prove this
conjecture.

  The Turaev-Viro approach   derives TQFTs
from spherical fusion categories.    We define a fusion category  as a monoidal category with compatible left and right dualities such that all objects are
direct sums of   simple objects and the number of isomorphism
classes of simple objects is finite. The condition of sphericity
says that the left and right dimensions of all objects are equal. A
spherical fusion category has a numerical dimension, which we
suppose to be non-zero throughout the introduction.  The original
TV-construction \cite{TV} applies to
 categories of representations of the quantum group
$U_q(sl_2(\CC))$ at roots of unity. This was extended to
  modular categories (i.e., to spherical fusion categories with braiding) in~\cite{Tu1}. At about the same time,
A.\ Ocneanu  observed  that the use of the braiding may be avoided.
This was formalized by Barrett and Westbury \cite{BW} (see also
\cite{GK}) who derived a topological  invariant $\vert M\vert_{\cc}$
of an arbitrary closed oriented 3-manifold~$M$ from a  spherical
fusion category $\cc$.  The definition of~$\vert
M\vert_{\cc}$ goes by considering a certain state sum on a
triangulation of~$M$ and proving that this sum depends only on~$M$
and not on the choice of triangulation. The key algebraic
ingredients of the state sum are the so-called $6j$-symbols
associated with~$\cc$.

The Reshetikhin-Turaev construction of   3-manifold invariants  uses
as the main algebraic ingredient   a modular  category~$\bb$. This
construction associates with every closed oriented 3-manifold~$M$ a
numerical invariant $\tau_\bb(M)$. Its definition is based on
surgery presentations of~$M$ by links in the 3-sphere.

For every monoidal category~$\cc$,  Joyal and Street \cite{JS} and Drinfeld (unpublished, see Majid~\cite{Maj}) defined a braided
monoidal category $\zz(\cc)$ called the center of~$\cc$.
A fundamental theorem of M\"uger \cite{Mu} says that   the center of a   spherical fusion category~$\cc$ over an algebraically
closed field is   modular. Combining   with the results
mentioned above, we observe that such a
  $\cc$ gives rise to two  3-manifold invariants: $\vert M\vert_{\cc}$ and $\tau_{\zz(\cc)}(M)$.  We prove in
  the present
  paper that these   invariants are equal, i.e.,  for all~$M$,
\begin{equation}\label{eq-main}
\vert
M\vert_{\cc}=\tau_{\zz(\cc)} (M).
\end{equation}
This   was previously
known in several special cases: when~$\cc$   is modular \cite{Tu1, Wa}, when~$\cc$ is the category of bimodules associated
with a subfactor \cite{KSW}, and when $\cc$ is the category of representations of a finite group. For modular $\cc$, the category $\zz(\cc)$ is braided equivalent to the Deligne tensor product $\cc\boxtimes \overline{\cc}$, where $\overline{\cc}$ is the mirror of $\cc$, and therefore Formula \eqref{eq-main} can be rewritten as $$\vert
M\vert_{\cc}=\tau_{\cc \boxtimes \overline{\cc}} (M)=\tau_{\cc} (M) \,\tau_{\overline{\cc}} (M)=\tau_{\cc} (M)\,\tau_{\cc} (-M),$$ where $-M$ is $M$ with opposite orientation. If $\cc$ is a unitary modular category, then $\tau_{\cc} (-M)=\overline{\tau_{\cc} (M)}$ and so $\vert M\vert_{\cc}=|\tau_{\cc} (M)|^2$.

Formula~\eqref{eq-main} relates two
 categorical approaches to 3-manifold invariants through the
categorical center. This relationship sheds new light on both
approaches and shows, in particular, that the  RT-construction is
more general than the state sum construction. For further corollaries of Formula~\eqref{eq-main}, see Section~\ref{sec-TQFT+}.

The proof of Formula~\eqref{eq-main} is based on topological quantum
field theory (TQFT). For a modular category~$\bb$, the  invariants
$\tau_\bb (M)$ extend  to  a 3-dimensional TQFT $\tau_\bb$ derived
from $\bb$, see \cite{RT}. The invariant  $\vert M\vert_{\cc}$ also extends to a 3-dimensional state sum  TQFT $\vert .\vert_\cc$ which we define here in terms
of state sums on    skeletons of 3-manifolds.  It is crucial for the proof of Formula~\eqref{eq-main} that  we allow non-generic skeletons, i.e., skeletons with edges incident to $\geq 4$ regions.  In particular,  we
give two different state sums on any triangulation~$t$ of a closed oriented
3-manifold~$M$: the one in \cite{TV, BW} and a new one. In
the former, the labels are attributed to the edges and the Boltzmann
weights are the $6j$-symbols computed in the tetrahedra; in the
latter, the labels are attributed to the faces and the Boltzmann
weights are computed in the vertices. The existence of two different
state sums   is due to the fact that the triangulation~$t$ gives
rise to two different skeletons of~$M$: the 2-skeleton of the
cellular decomposition of~$M$ dual to~$t$ and the 2-skeleton of~$t$
itself. (It is non-obvious but true that these two
  state sums are equal.)


 Our main theorem claims that for any spherical fusion
 category~$\cc$  over an algebraically
closed field, the   TQFTs $\vert
 .\vert_\cc$ and $\tau_{\zz(\cc)}$
are isomorphic: $$ \vert \Sigma
 \vert_\cc \simeq \tau_{\zz(\cc)}(\Sigma) \quad \text{and} \quad \vert M
 \vert_\cc \simeq \tau_{\zz(\cc)}(M)$$ for any oriented closed surface $\Sigma$ and any oriented 3-cobordism $M$. The proof involves a detailed
study of  transformations of  skeletons of 3-manifolds and the computation of the coend of  $\zz(\cc)$ provided by the theory of Hopf monads in
monoidal categories due to  Brugui{\`e}res and Virelizier
\cite{BV3}. Another important ingredient of the proof is an extension of  $ \vert M
 \vert_\cc $  to a state sum invariant $ \vert M, L
 \vert_\cc $ of a $\zz(\cc)$-colored framed link $L$ in a 3-manifold $M$.

After the appearance of the present paper on arXiv, Balsam and Kirillov \cite{Ba1, Ba2, KB} found a different proof of \eqref{eq-main} based on the concept of an extended TQFT.

The paper is organized as follows. Sections 1--4 deal with
preliminaries on monoidal and fusion categories and the associated
invariants of colored graphs. In Sections 5--9 we construct the TQFT
$\vert
 .\vert_\cc$. In Sections 10 and 11 we recall the necessary definitions from the theory of modular categories and state our main theorems.  A proof of these theorems is given in Sections 12--15. Though we do not specifically use  $6j$-symbols, we include for completeness   an appendix summarizing their   properties.

The authors are indebted to N.~Ivanov and S.~Matveev for   helpful
discussions. The work of V.~Turaev on this paper was partially
supported by the NSF grant  DMS-0904262. A.~Virelizier was partially supported by the ANR grant GESAQ.

Throughout the paper, the symbol $\kk$ denotes a commutative ring.

\section{Pivotal and spherical categories} \label{sect-basic-defi}

We review several classes of monoidal categories needed in the
sequel.

\subsection{Pivotal   categories (\cite{Malt})}
By a \emph{pivotal} (or \emph{sovereign}) category, we mean a strict
monoidal category $ \cc $ with unit object~$\un$ such that to each
object $X\in \Ob(\cc)$ there are associated a \emph{dual
object}~$X^*\in \Ob(\cc)$ and four morphisms
\begin{align*}
& \lev_X \co X^*\otimes X \to\un,  \qquad \lcoev_X\co \un  \to X \otimes X^*,\\
&   \rev_X \co X\otimes X^* \to\un, \qquad   \rcoev_X\co \un  \to X^* \otimes X,
\end{align*}
satisfying the following   conditions:
\begin{enumerate}
  \renewcommand{\labelenumi}{{\rm (\alph{enumi})}}
\item For every $X\in
\Ob(\cc)$, the pair $( \lev_X,\lcoev_X)$ is a \emph{left
duality} for $X$, i.e.,
$$(\id_X \otimes \lev_X)(\lcoev_X \otimes \id_X)=\id_X  \quad  {\rm {and}} \quad (\lev_X \otimes \id_{X^*})(\id_{X^*} \otimes \lcoev_X)=\id_{X^*};$$
\item For every $X\in
\Ob(\cc)$, the pair $( \rev_X,\rcoev_X)$ is a \emph{right
duality} for~$X$, i.e.,
$$  (\rev_X \otimes \id_X)(\id_X \otimes \rcoev_X)=\id_X \quad  {\rm {and}} \quad  (\id_{X^*} \otimes \rev_X)(\rcoev_X \otimes \id_{X^*})=\id_{X^*};
$$
\item For every morphism $f\co X \to Y$ in $\cc$, the  left  dual
$$
f^*= (\lev_Y \otimes  \id_{X^*})(\id_{Y^*}  \otimes f \otimes \id_{X^*})(\id_{Y^*}\otimes \lcoev_X) \colon Y^*\to X^*$$ is equal to the  right  dual
$$
f^*= (\id_{X^*} \otimes \rev_Y)(\id_{X^*} \otimes f \otimes \id_{Y^*})(\rcoev_X \otimes \id_{Y^*}) \colon Y^*\to X^*;$$
\item For all   $X,Y\in
\Ob(\cc)$, the  morphisms $\ldual{X} \otimes \ldual{Y} \to (Y \otimes X)^* $
defined as
$$(\lev_X  \otimes \id_{(Y \otimes X)^*})(\id_{X^*}  \otimes \lev_Y \otimes
\id_{X\otimes (Y \otimes X)^*})(\id_{X^* \otimes Y^*}\otimes \lcoev_{Y \otimes
X})$$   and as $$(\id_{(Y \otimes
X)^*} \otimes \rev_Y)(\id_{(Y \otimes X)^* \otimes Y} \otimes \rev_X
\otimes \id_{Y^*})(\rcoev_{Y \otimes X}\otimes \id_{X^* \otimes
Y^*})  $$ are equal (these morphisms are called   monoidal constraints);
\item  $\lev_\un=\rev_\un \colon \un^* \to \un$ (or, equivalently, $\lcoev_\un=\rcoev_\un\colon \un  \to \un^*$).
\end{enumerate}

Note that the morphisms $\lev_\un$ and  $\lcoev_\un$ (respectively,
$\rev_\un$ and  $\rcoev_\un$) are mutually inverse isomorphisms. The
monoidal constraints   in (d) also are isomorphisms. By abuse of
notation, we will suppress these isomorphisms in the formulas.
 For example, we will write $(f \otimes g)^*=g^* \otimes f^*$ for   morphisms $f,g$ in~$\cc$.

 \subsection{Traces and dimensions}\label{sec-traces} For an endomorphism $f $ of an object $X$ of   a pivotal category $\cc$, one defines the
\emph{left} and \emph{right traces} $\tr_l(f), \tr_r(f) \in
\End_\cc(\un)$ by
$$\tr_l(f)=\lev_X(\id_{\ldual{X}} \otimes f) \rcoev_X  \quad {\text
{and}}\quad
 \tr_r(f)=\rev_X( f \otimes \id_{\ldual{X}}) \lcoev_X  .
$$
Both traces are symmetric: $\tr_l(gh)=\tr_l(hg)$ and $
\tr_r(gh)=\tr_r(hg)$ for any morphisms $g\co X \to Y$ and   $h\co Y
\to X$  in $\cc$. Also $\tr_l(f)=\tr_r(\ldual{f})=\tr_l(f^{**})$ for
any endomorphism $f $ of an object (and similarly  with $l,r$
exchanged). If
\begin{equation}\label{special}
\alpha\otimes \id_X= \id_X\otimes \alpha \quad \text{for all $\alpha\in
\End_{\cc}(\un)$ and $X\in \Ob(\cc)$,}
\end{equation}
then $\tr_l,\tr_r$
are $\otimes$-multiplicative: $ \tr_l(f\otimes g)=\tr_l(f) \,
\tr_l(g)$ and $\tr_r(f\otimes g)=\tr_r(f) \, \tr_r(g) $ for all
endomorphisms $f,g$ of objects~of~$\cc$.

The \emph{left} and \emph{right dimensions} of   $X\in \Ob (\cc)$
are defined by $ \dim_l(X)=\tr_l(\id_X) $ and $
\dim_r(X)=\tr_r(\id_X) $. Clearly,
$\dim_l(X)=\dim_r(\ldual{X})=\dim_l(X^{**})$ (and similarly  with
$l,r$ exchanged). Note that isomorphic objects have the same
dimensions and $\dim_l(\un)=\dim_r(\un)=\id_{\un}$. If~$\cc$ satisfies \eqref{special}, then  left and right dimensions are $\otimes$-multiplicative: $\dim_l (X\otimes Y)= \dim_l (X)
\dim_l (Y)$ and $\dim_r (X\otimes Y)= \dim_r (X) \dim_r (Y)$ for any
$X,Y\in \Ob(\cc)$.

\subsection{Penrose graphical calculus} We represent morphisms in a category $\cc$ by plane   diagrams to be read from the bottom to the top.
The  diagrams are made of   oriented arcs colored by objects of
$\cc$  and of boxes colored by morphisms of~$\cc$.  The arcs connect
the boxes and   have no mutual intersections or self-intersections.
The identity $\id_X$ of $X\in \Ob(\cc)$, a morphism $f\co X \to Y$,
and the composition of two morphisms $f\co X \to Y$ and $g\co Y \to
Z$ are represented as follows:
\begin{center}
\psfrag{X}[Bc][Bc]{\scalebox{.7}{$X$}} \psfrag{Y}[Bc][Bc]{\scalebox{.7}{$Y$}} \psfrag{h}[Bc][Bc]{\scalebox{.8}{$f$}} \psfrag{g}[Bc][Bc]{\scalebox{.8}{$g$}}
\psfrag{Z}[Bc][Bc]{\scalebox{.7}{$Z$}} $\id_X=$ \rsdraw{.45}{.9}{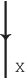}\,,\quad $f=$ \rsdraw{.45}{.9}{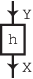} ,\quad \text{and} \quad $gf=$ \rsdraw{.45}{.9}{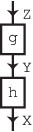}\,.
\end{center}
  If $\cc$ is
monoidal, then the monoidal product of two morphisms $f\co X \to Y$
and $g \co U \to V$ is represented by juxtaposition:
\begin{center}
\psfrag{X}[Bc][Bc]{\scalebox{.7}{$X$}} \psfrag{h}[Bc][Bc]{\scalebox{.8}{$f$}}
\psfrag{Y}[Bc][Bc]{\scalebox{.7}{$Y$}}  $f\otimes g=$ \rsdraw{.45}{.9}{morphism} \psfrag{X}[Bc][Bc]{\scalebox{.8}{$U$}} \psfrag{g}[Bc][Bc]{\scalebox{.8}{$g$}}
\psfrag{Y}[Bc][Bc]{\scalebox{.7}{$V$}} \rsdraw{.45}{.9}{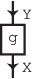}\,.
\end{center}
In a pivotal category, if an arc colored by $X$ is oriented upwards,
then the corresponding object   in the source/target of  morphisms
is $X^*$. For example, $\id_{X^*}$  and a morphism $f\co X^* \otimes
Y \to U \otimes V^* \otimes W$  may be depicted as:
\begin{center}
 $\id_{X^*}=$ \, \psfrag{X}[Bl][Bl]{\scalebox{.7}{$X$}}
\rsdraw{.45}{.9}{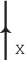} $=$  \,
\psfrag{X}[Bl][Bl]{\scalebox{.7}{$\ldual{X}$}}
\rsdraw{.45}{.9}{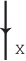}  \quad and \quad
\psfrag{X}[Bc][Bc]{\scalebox{.7}{$X$}}
\psfrag{h}[Bc][Bc]{\scalebox{.8}{$f$}}
\psfrag{Y}[Bc][Bc]{\scalebox{.7}{$Y$}}
\psfrag{U}[Bc][Bc]{\scalebox{.7}{$U$}}
\psfrag{V}[Bc][Bc]{\scalebox{.7}{$V$}}
\psfrag{W}[Bc][Bc]{\scalebox{.7}{$W$}} $f=$
\rsdraw{.45}{.9}{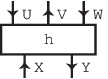} \,.
\end{center}
The duality morphisms   are depicted as follows:
\begin{center}
\psfrag{X}[Bc][Bc]{\scalebox{.7}{$X$}} $\lev_X=$ \rsdraw{.45}{.9}{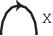}\,,\quad
 $\lcoev_X=$ \rsdraw{.45}{.9}{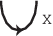}\,,\quad
$\rev_X=$ \rsdraw{.45}{.9}{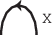}\,,\quad
\psfrag{C}[Bc][Bc]{\scalebox{.7}{$X$}} $\rcoev_X=$
\rsdraw{.45}{.9}{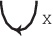}\,.
\end{center}
The dual of a morphism $f\co X \to Y$ and the
  traces of a morphism $g\co X \to X$ can be depicted as
follows:
\begin{center}
\psfrag{X}[Bc][Bc]{\scalebox{.7}{$X$}} \psfrag{h}[Bc][Bc]{\scalebox{.8}{$f$}}
\psfrag{Y}[Bc][Bc]{\scalebox{.7}{$Y$}} \psfrag{g}[Bc][Bc]{\scalebox{.8}{$g$}}
$f^*=$ \rsdraw{.45}{.9}{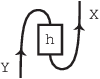}$=$ \rsdraw{.45}{.9}{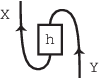}\quad \text{and} \quad
$\tr_l(g)=$ \rsdraw{.45}{.9}{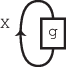}\,,\quad  $\tr_r(g)=$ \rsdraw{.45}{.9}{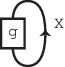}\,.
\end{center}
  If $\cc$ is pivotal, then    the morphisms represented by the diagrams
are invariant under isotopies of the diagrams in the plane keeping
fixed the bottom and   top endpoints.

\subsection{Linear  and spherical categories}\label{sphesphe}
A \emph{monoidal \kt category} is a monoidal category $\cc$ such that its
 hom-sets are (left) \kt modules, the composition and monoidal product of morphisms are \kt bilinear,
  and  $\End_\cc(\un)$ is a free \kt module of rank one.
Then the map $\kk \to \End_\cc(\un), k \mapsto k \, \id_\un$  is a
\kt algebra isomorphism. It is used to identify $\End_\cc(\un)=\kk$.

A pivotal \kt category satisfies \eqref{special}. Therefore the traces $\tr_l,\tr_r$ and the
dimensions $\dim_l, \dim_r$ in such a category are
$\otimes$-multiplicative. Clearly,   $\tr_l,\tr_r$ are  \kt linear.

  A \emph{spherical category} is a pivotal category whose left and
right traces are equal, i.e.,  $\tr_l(g)=\tr_r(g)$ for every
endomorphism $g$ of an object. Then $\tr_l(g)$ and $ \tr_r(g)$ are
denoted $\tr(g)$ and called the \emph{trace of $g$}. Similarly, the
left and right dimensions of an object~$X$ are denoted   $\dim(X)$
and called the \emph{dimension of $X$}.

For spherical categories, the corresponding Penrose graphical
calculus has the following property:   the morphisms represented by
 diagrams are invariant under isotopies of   diagrams in the 2-sphere $S^2=\RR^2\cup
\{\infty\}$, i.e., are preserved under isotopies pushing   arcs of
the diagrams across~$\infty$.  For example, the diagrams above
representing $\tr_l(g)$ and $\tr_r(g)$ are related by such an
isotopy. The condition $\tr_l(g)=\tr_r(g)$ for all $g$ is therefore
necessary (and in fact sufficient) to ensure this property.

\subsection{Remark}\label{rem-pivotal} Our definition of a pivotal category is equivalent to the  standard
definition given in terms of   pivotal structures, see \cite{Malt}.
A \emph{pivotal structure} on a monoidal category $\cc$ with left
duality $\{(X^*,\lev, \lcoev)\}_{X\in \Ob (\cc)}$ is a monoidal
natural isomorphism $\psi=\{\psi_X\co X \to X^{**}\}_{X \in
\Ob(\cc)}$, i.e.,  $\psi_{X \otimes Y}=\psi_X \otimes \psi_Y$ and
$\psi_\un=\id_\un$  (up to the monoidal constraints). A pivotal
category $\cc$ obtains a pivotal structure   by $
 \psi_X=(\rev_X \otimes \id_{X^{**}})(\id_X \otimes \lcoev_{X^*})$
for   $X \in \Ob(\cc)$. Conversely,   a pivotal structure makes
$\cc$   pivotal   by
 $
 \rev_X=\lev_{\ldual{X}}(\psi_X \otimes \id_{\ldual{X}})$ and $
 \rcoev_X=(\id_{\ldual{X}} \otimes \psi^{-1}_X)\lcoev_{\ldual{X}}$.

\section{Multiplicities in pivotal categories}\label{sect-sym-modules}

For $n\geq 1$ objects $X_1, \ldots, X_n$ of a monoidal category
$\cc$, one can consider the \lq\lq set of multiplicities" $
\Hom_\cc(\un,X_1 \otimes \cdots \otimes X_n)$. We show that if $\cc$
is pivotal, then this set essentially depends only on the cyclic
order of the objects $X_1, \ldots, X_n$.

\subsection{Permutation maps}\label{sect-permut} For   objects $X$, $Y$ of  a pivotal
category $\cc$, let $$\sigma_{X,Y}\co \Hom_\cc(\un,X\otimes Y) \to
\Hom_\cc(\un,Y \otimes X)$$  be the map defined as follows: for any
$\alpha \in \Hom_\cc(\un,X\otimes Y)$,
\begin{equation}\label{fla-}
\sigma_{X,Y}(\alpha)=(\lev_X \otimes \id_{Y\otimes X})(\id_{\dual{X}} \otimes \alpha \otimes \id_X)\rcoev_X=
\psfrag{X}[Bc][Bc]{\scalebox{.7}{$X$}} \psfrag{a}[Bc][Bc]{\scalebox{1}{$\alpha$}}
\psfrag{Y}[Bc][Bc]{\scalebox{.7}{$Y$}}
\rsdraw{.45}{.9}{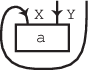}\,.
\end{equation}
It is easy to check that the maps $ \{\sigma_{X,Y}\}_{X,Y\in
\Ob(\cc)}$ are natural  in the sense that
$$(g\otimes f) \sigma_{X,Y}(\alpha)= \sigma_{X',Y'}((f\otimes g)
\alpha)$$ for any morphisms $f\co X\to X'$,  $g\co Y\to Y'$, and $\alpha \co
\un\to X\otimes Y $ in $\cc$. The following lemma shows that the
maps $\{\sigma_{X,Y}\}_{X,Y}$ behave  as permutations.

\begin{lem}\label{lem-Gamma}
For all   $X,Y,Z \in \Ob(\cc)$,
\begin{enumerate}
  \renewcommand{\labelenumi}{{\rm (\alph{enumi})}}
   \item $\sigma_{X,Y}$ is an isomorphism and
  $\sigma_{X,Y}^{-1}=\sigma_{Y,X}$;
  \item $\sigma_{X,\un}=\sigma_{\un,X}=\id_{\Hom_\cc(\un,X)}$;
  \item $\sigma_{X\otimes Y,Z}=\sigma_{Y,Z \otimes X} \,\sigma_{X,Y \otimes Z}$ and $\sigma_{X,Y\otimes Z}=\sigma_{Z \otimes X,Y}\,\sigma_{X\otimes Y,Z}$.
\end{enumerate}
\end{lem}
\begin{proof}
Claims (b) and (c) are   direct consequences of the monoidality of
the duality functor $ (\cc^\opp,\otimes^\opp ) \to (\cc,\otimes )$
defined by $X \mapsto X^*$ and $f \mapsto f^*$. From (b) and~(c), we
obtain  $\sigma_{Y,X} \,\sigma_{X,Y}=\sigma_{X \otimes Y,\un}=
\id_{\Hom_\cc(\un,X \otimes Y)}$, hence (a).
\end{proof}

It is easy to deduce from Claim (a) of this lemma   that
\begin{equation}\label{fla} \sigma_{X,Y}(\alpha)=(\id_{Y\otimes X} \otimes \rev_Y)(\id_Y
\otimes \alpha \otimes
\id_{\ldual{Y}})\lcoev_Y=\psfrag{X}[Bc][Bc]{\scalebox{.7}{$X$}}
\psfrag{a}[Br][Br]{\scalebox{1}{$\alpha$}}
\psfrag{Y}[Br][Br]{\scalebox{.7}{$Y$}} \rsdraw{.45}{.9}{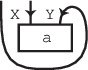}\,
.\end{equation}

\subsection{Symmetrized sets of multiplicities}\label{sect-signedcyclicset}
A \emph{signed object} of a  pivotal category~$\cc$  is a pair
$(X,\varepsilon)$ where $X\in\Ob(\cc)$ and $\varepsilon\in\{+,-\}$.
Given a signed object $(X,\varepsilon)$,    set $X^\varepsilon=X$ if
$\varepsilon=+$ and $X^\varepsilon=X^*$ if $\varepsilon=-$.

A \emph{cyclic $\cc$\ti set} is a triple $(E,c\co E \to \Ob(\cc)
,\varepsilon\co E \to \{+, - \})$, where $E$ is a totally cyclically
ordered finite set. In other words,  a cyclic $\cc$\ti set is a
totally cyclically ordered finite set whose elements are labeled
  by signed objects of $\cc$.  For shortness, we will
 sometimes write $E$ for   $(E,c
,\varepsilon )$.

A cyclic $\cc$\ti set $E=(E,c,\varepsilon)$ determines a set $H(E)$
as follows. For   $e\in E$, set
$$Z_e=c(e_1)^{\varepsilon(e_1)} \otimes \cdots \otimes
c(e_n)^{\varepsilon(e_n)} \in \Ob (\cc)\quad {\text {and}}\quad
H_e=\Hom_\cc (\un,Z_e  ),$$ where $e =e_1<e_2< \cdots <e_n$ are the
elements of $E$ ordered starting from $e$ via the
 given cyclic order   (here $n=\# E$ is the number of elements of~$E$). For $e,f\in E$, we
define a   map $p_{e,f}\co H_e\to H_f$ as follows.   We have $f=e_k$
for   $k\in\{1,\dots,n\}$. Set
$$X= c(e_1)^{\varepsilon(e_1)} \otimes \cdots \otimes
c(e_{k-1})^{\varepsilon(e_{k-1})} \quad {\text {and}} \quad
Y=c(e_k)^{\varepsilon(e_k)} \otimes \cdots \otimes
c(e_n)^{\varepsilon(e_n)}.$$ Clearly, $Z_e=X\otimes Y$ and
$Z_f=Y\otimes X$. The map  $p_{e,f} $ carries a  morphism $\alpha\co
\un \to Z_e   $ to   $\sigma_{X,Y}(\alpha) \co \un \to Z_f$.
Lemma~\ref{lem-Gamma} implies that $(H_e,p_{e,f})_{e,f \in E}$ is a
projective system of sets and bijections. The   projective limit $
H(E )=\underleftarrow{\lim} \, H_e $  depends only on $E$. It is
equipped with a system of  bijections $\tau=\{\tau_e\co H(E ) \to
H_e\}_{e\in E}$, called the \emph{universal cone}. The bijections $
\tau_e\co H(E ) \to H_e
 $   are called   \emph{cone
bijections}.

An isomorphism   of cyclic $\cc$\ti
  sets $\phi\co E \to  E' $
 is  a bijection  preserving the cyclic order and commuting with the maps to
 $\Ob(\cc)$ and $ \{+, - \}$. Such  a  $\phi$
induces    a   bijection
 $H(\phi) \co H(E ) \to H(E' )$ in the obvious way.

%

\subsection{Symmetrized multiplicity modules}\label{sect-dual-signed-set} Let $\cc$ be a pivotal \kt category.
Clearly,  for any cyclic $\cc$\ti set $E$, the set  $H(E )$ is a \kt
module (called the \emph{symmetrized multiplicity module}), and the
cone bijections and the maps $H(\phi)$ as above are \kt
isomorphisms. We now study duality for these
  modules.

  For a   tuple $S=((X_1,\varepsilon_1),\ldots,
(X_n,\varepsilon_n))$ of signed objects of $\cc$ with $n\geq 1$, set
\begin{equation}\label{Evev-}
X_S=X_1^{\varepsilon_1} \otimes \cdots  \otimes X_n^{\varepsilon_n}  .
\end{equation} and
$ S^*=(X_n,-\varepsilon_n),\ldots, (X_1,-\varepsilon_1)$. The tuple
$S$ determines a cyclic $\cc$\ti set $E_S=\{1,2,  \ldots, n\}$,
where $1< 2< \cdots <n<1$ and the label of each $e\in E_S$  is
$(X_e, \varepsilon_e)$.   By definition, the  \kt  module
$H(S)=H(E_S)$ is preserved under cyclic permutations of $S$ and is
isomorphic to $\Hom_\cc(\un,X_S)$ through the   cone isomorphism
$\tau_1$.

 If $S$ and $T$ are two tuples of signed objects of $\cc$,
then  $ST$ is  the tuple obtained  by their concatenation. Clearly,
$X_{ST}=X_S \otimes X_T$,  $(ST)^*=T^* S^*$, and $S^{**}=S$. The
following recursive formulas define   a morphism $ \Ev_S\co X_{S^*}
\otimes X_S \to \un $:
$$\Ev_{(X,+)}=\lev_X, \quad \Ev_{(X,-)}=\rev_X, \quad {\text {and}}\quad \Ev_{ST}=
\Ev_T(\id_{X_{T^*}} \otimes \Ev_S \otimes \id_{X_T})$$ for any
  $S$, $T$. The morphism $ \Ev_S$ induces a \kt bilinear \emph{evaluation form}
\begin{equation}\label{Evev}
\omega_S\co
  \Hom_\cc(\un,X_{S^*}) \otimes \Hom_\cc(\un,X_S) \to  \End_\cc(\un)=\kk
\end{equation}
  by $\omega_S(\alpha \otimes \beta)=\Ev_S(\alpha
\otimes \beta)$ for all $\alpha \in \Hom_\cc(\un,X_{S^*})$ and
$\beta \in \Hom_\cc(\un,X_S)$.   For  example:
\begin{equation*}
\omega_{(X,-),(Y,+)}(\alpha \otimes \beta)=\, \psfrag{X}[Bc][Bc]{\scalebox{.7}{$X$}}
\psfrag{Y}[Bc][Bc]{\scalebox{.7}{$Y$}} \psfrag{a}[Bc][Bc]{\scalebox{.8}{$\alpha$}}
\psfrag{b}[Bc][Bc]{\scalebox{.8}{$\beta$}} \rsdraw{.45}{.9}{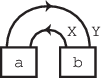}.
\end{equation*}
As an exercise, the reader may verify that $\omega_S(\alpha  \otimes \beta)= \omega_{S^*}(\beta \otimes \alpha)$ for all $\alpha$ and $\beta$.

\begin{lem}\label{lem-omegasigma} If $\cc$ is spherical, then
$\omega_{TS} \bigl(\sigma_{X_{T^*},X_{S^*}} \otimes
\sigma_{X_S,X_T}\bigr)= \omega_{ST}$ for all tuples $S$, $T$ of
signed objects of $\cc$.
\end{lem}
\begin{proof}
Let $\alpha \in \Hom_\cc(\un,X_{(ST)^*})=\Hom_\cc(\un,X_{T^*}\otimes
X_{S^*})$ and $\beta \in \Hom_\cc(\un,X_{ST})=\Hom_\cc(\un,X_S
\otimes X_T)$.  By the pivotality of $\cc$,
\begin{equation*}
 \psfrag{E}[Bc][Bc]{\scalebox{.8}{$\Ev_T$}}
\rsdraw{.45}{.9}{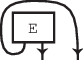}\,=\Ev_{T^*}=\,\psfrag{X}[Br][Br]{\scalebox{.7}{$X_T$}}
\psfrag{Y}[Bc][Bc]{\scalebox{.7}{$X_{T^*}$}}\rsdraw{.45}{.9}{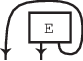}\,.
\end{equation*}
Using these equalities,  Formulas~\eqref{fla-}, \eqref{fla},  and
the sphericity of $\cc$, we obtain
\begin{equation*}
\psfrag{E}[Bc][Bc]{\scalebox{.8}{$\Ev_T$}} \psfrag{R}[Bc][Bc]{\scalebox{.8}{$\Ev_{T^*}$}} \psfrag{F}[Bc][Bc]{\scalebox{.8}{$\Ev_S$}}
\psfrag{a}[Bc][Bc]{\scalebox{.8}{$\alpha$}} \psfrag{u}[Bc][Bc]{\scalebox{.8}{$\beta$}}
\rsdraw{.45}{.9}{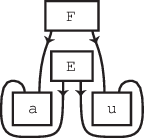} \, = \, \rsdraw{.45}{.9}{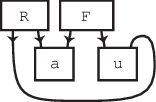} \, = \, \rsdraw{.45}{.9}{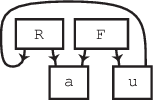} \, = \, \rsdraw{.45}{.9}{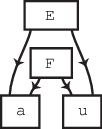} \,.
\end{equation*}
Here the left-hand side represents  $\omega_{TS}
\bigl(\sigma_{X_{T^*},X_{S^*}}(\alpha) \otimes \sigma_{X_S,X_T}
(\beta)\bigr)$ and the right-hand side represents
$\omega_{ST}(\alpha \otimes \beta)$. Hence these two  expressions
are equal.
\end{proof}

By Lemma~\ref{lem-omegasigma},   for spherical $\cc$,  the
form~\eqref{Evev} is compatible with cyclic permutations of $S$ and
induces a well defined pairing $H(S^*) \otimes H(S)\to \kk$ also
denoted~$\omega_S$.

The \emph{dual}    of  a cyclic $\cc$\ti set $ (E,c,\varepsilon)$ is
the cyclic $\cc$\ti set $ (E^\opp,c,-\varepsilon)$, where~$E^\opp$
is~$E$ with opposite cyclic order. Enumerating the elements of $E$
in their cyclic order, we can identify $E$ with $E_S$ for a
tuple~$S$ of signed objects of $\cc$. Then $E^\opp= E_{S^*}$. If
$\cc$ is spherical, then the pairing $\omega_S\co H(S^*) \otimes
H(S)\to \kk$ induces a   pairing $\omega_{E}\co H(E^\opp ) \otimes
H(E )  \to  \kk  $. More generally, a \emph{duality} between cyclic
$\cc$\ti sets~$ E $ and~$ E' $ is an isomorphism  of cyclic $\cc$\ti
sets $\phi\co  E^\opp \to
 E' $. Such $\phi$ induces a \kt isomorphism $H(\phi) \co H(E^\opp ) \to H(E'
 )$ and a  pairing
\begin{equation}\label{bilipai}
 \omega_{E} \circ (H(\phi)^{-1} \otimes \id)\co H(E' ) \otimes H(E ) \to \kk.
\end{equation}

\subsection{Non-degeneracy}\label{sect-non-degen-contraction}
Given two \kt modules $M$ and $N$, one calls a \kt bilinear  form
$\omega \co M \otimes N \to \kk$   \emph{non-degenerate} if there is
a \kt linear map $\Omega \co \kk \to N \otimes M$ such that $(\id_N
\otimes \omega)(\Omega\otimes \id_N)=\id_N$ and $(\omega \otimes
\id_M)(\id_M \otimes \Omega)=\id_M$.  If such~$\Omega$  exists, then
it is unique, and   $\Omega(1)\in N \otimes M$ is   the
\emph{inverse of $\omega$}.
The dual \kt homomorphism $
 \Omega^\star    \co M^\star \otimes N^\star   \to \kk$ is
 called the
\emph{contraction}  induced by $\omega$. Here $M^\star=\Hom_{\kk}
(M, \kk)$.

  A pivotal \kt category $\cc$ is
\emph{non-degenerate} if  for any~$X \in \Ob (\cc)$, the  form
$\omega_X=\omega_{(X,+)} \co \Hom_\cc(\un,X^*) \otimes
\Hom_\cc(\un,X) \to \kk$ (carrying $\alpha\otimes \beta$ to $\lev_X
  (\alpha\otimes \beta)$) is non-degenerate.

\begin{lem}\label{lem-pairing-non-degen}
If $\cc$ is non-degenerate, then   the form~\eqref{Evev} is
non-degenerate for all~$S$ and the form~\eqref{bilipai}  is
non-degenerate for all $E, E', \phi$.
\end{lem}
\begin{proof}
 Let $
 \{\psi_X\co X \to X^{**}\}_{X\in \Ob(\cc)} $ be the  isomorphisms  defined in Remark~\ref{rem-pivotal}.
 For
$S=((X_1,\varepsilon_1),\ldots, (X_n,\varepsilon_n))$, set
$$\psi_S=\psi_{(X_1,\varepsilon_1)} \otimes \cdots \otimes
\psi_{(X_n,\varepsilon_n)}\co X_S \to (X_{S^*})^*,$$ where
$\psi_{(X,+)}=\psi_X$ and $\psi_{(X,-)}=\id_{X^*}$ for any $X\in
\Ob(\cc)$. One verifies that $\Ev_S=\lev_{X_S} (\psi_{S^*} \otimes
\id_{X_S})$ and therefore $\omega_S(\alpha \otimes
\beta)=\omega_{X_S}(\psi_{S^*}(\alpha) \otimes \beta)$.
 Since the form~$\omega_{X_S}$ is non-degenerate
and $\psi_{S^*}$ is an iso\-morphism,  $\omega_S$ is non-degenerate.
\end{proof}

\section{Invariants of  colored graphs}\label{sect-inv-graphs}

In this section, $\cc$ is a pivotal \kt category.

\subsection{Colored graphs in surfaces}\label{sect-graph}  By a \emph{graph}, we mean
a finite graph  without isolated vertices.   Every edge of a graph
connects two (possibly coinciding) vertices called the endpoints of
the edge. We allow multiple edges with the same endpoints.
  A graph~$G$ is
\emph{$\cc$-colored}, if   each edge of~$G$ is oriented and
  endowed with an object  of $\cc$  called the
\emph{color} of the edge.

Let $\Sigma$ be an oriented surface. By a \emph{graph} in~$\Sigma$,
we mean a graph embedded in~$\Sigma$.   A vertex $v$ of a
$\cc$-colored graph~$G$ in~$\Sigma$ determines a cyclic $\cc$\ti set
$ (E_v, c_v, \varepsilon_v)$ as follows:   $E_v$ is the set of
half-edges of~$G$ incident to~$v$ with cyclic order induced by the
opposite orientation of~$\Sigma$; the maps $c_v\co E_v \to \Ob(\cc)$
and $\varepsilon_v \co E_v \to \{+,-\}$ assign to each half-edge
$e\in E_v$  its color and sign (the sign is $+$ if $e$ is oriented
towards $v$ and $-$ otherwise). Set $H_v(G)=H(E_v )$ and
$$H(G)=\otimes_v \, H_v(G),$$ where $v$ runs over all vertices of~$G$
and $\otimes=\otimes_{\kk}$ is the tensor product  over $\kk$. To
stress the role of~$\Sigma$, we shall sometimes write $H_v(G;\Sigma)$ for~$H_v(G)$ and
$H(G;\Sigma)$ for~$H(G)$.

For an $n$-valent vertex $v$ of~$G$, the \kt  module $H_v(G)$ can be
described   as follows. Let $e_1 < e_2 < \cdots < e_n < e_1$ be the
half-edges of~$G$ incident to~$v$ with cyclic order induced by the
opposite orientation of~$\Sigma$.  Then $H_v(G) =
H((X_1,\varepsilon_1), \ldots , (X_n,\varepsilon_n))$, where
$X_r=c_v(e_r)$ and $\varepsilon_r=\varepsilon_v(e_r)$ are the color
and the sign of~$e_r$ for $r=1,\ldots , n$. To simplify notation, we
write $H(X_1\,\varepsilon_1, \ldots , X_n\,\varepsilon_n)$ for this
module. For example, $H(i+, j-,k+)$, $H(j-,k+, i+)$, and
$H(k+,i+,j-)$ all stand for the module associated to the trivalent
vertex
\begin{equation*}
\psfrag{i}[Bc][Bc]{\scalebox{.7}{$i$}} \psfrag{j}[Bc][Bc]{\scalebox{.7}{$j$}} \psfrag{k}[Bc][Bc]{\scalebox{.7}{$k$}}
\psfrag{E}[Bc][Bc]{\scalebox{.7}{$\Sigma$}}  \rsdraw{.45}{.9}{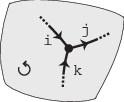}\,.
\end{equation*}

Given two $\cc$-colored graphs $G$ and $G'$ in $\Sigma$, an
\emph{isotopy} of~$G$ into~$G'$  is an isotopy of~$G$ into~$G'$ in
the class of $\cc$-colored graphs in
 $\Sigma$ preserving the   vertices, the edges, and   the orientation and the color
 of the
 edges.
An isotopy $\iota$ of~$G$ into~$G'$ induces an isomorphism of cyclic
$\cc$\ti sets $ E_v \to E_{v'} $, where $v$ is any vertex of~$G$ and
$v'=\iota(v)$. This   induces \kt isomorphisms $H_v(\iota)\co H_v(G)
\to H_{v'}(G')$ and $H(\iota)=\otimes_v \, H_v(\iota)\co H(G) \to
H(G')$.

A \emph{duality}   between two vertices $u$, $v$ of a  $\cc$-colored
graph  $G$  in~$\Sigma$ is a duality  between the cyclic $\cc$\ti
sets $ E_u $ and $ E_{v} $, see Section~\ref{sect-dual-signed-set}.
Such a duality
 induces a pairing $ \omega_{u,v}\co  H_{u}(G) \otimes H_v(G)
\to \kk $ and, when $\cc$ is non-degenerate, a contraction
homomorphism $ \ast_{u,v}\co H_{u}(G)^\star \otimes H_v(G)^\star
\to \kk$.

A $\cc$-colored graph $G \subset \Sigma$
determines  a $\cc$-colored graph   $G^\opp$  in~$-\Sigma$ obtained by
reversing orientation in all edges of~$G$ and in $\Sigma$ while keeping the colors of the edges. The
  cyclic $\cc$-sets determined by a  vertex $v$ of $G$ and $G^\opp$
  are dual. If $\cc$ is non-degenerate, then we can conclude that $$H_v(G^\opp;-\Sigma)= H_v(G;\Sigma)^\star \quad {\text {and}} \quad H (G^\opp;-\Sigma)= H (G;\Sigma)^\star.$$

\subsection{Invariants of   graphs  in  $\R^2$}\label{sect-inv-R2} We   always orient the plane $\R^2$
counterclockwise. Let  $G$ be a $\cc$-colored graph in   $\R^2$. For
each vertex $v$   of~$G$, we choose a half-edge $e_v\in E_v$ and
isotope  $G$ near~$v$ so that   the half-edges incident to $v$   lie
above $v$ with respect to the second coordinate on $\R^2$ and $e_v$
is the leftmost of them. Pick any $\alpha_v \in H_v(G)$ and replace
$v$ by a box colored with $\tau^v_{e_v}(\alpha_v)$ as depicted in
Figure~\ref{fig-alg-inv1}, where $\tau^v$ is the universal cone of
$H_v(G)$.
\begin{figure}[t]
\begin{center}
\psfrag{X}[Bc][Bc]{\scalebox{.7}{$X$}} \psfrag{Y}[Bc][Bc]{\scalebox{.7}{$Y$}} \psfrag{Z}[Bc][Bc]{\scalebox{.7}{$Z$}} \psfrag{T}[Bc][Bc]{\scalebox{.7}{$T$}} \psfrag{h}[Bc][Bc]{\scalebox{.8}{$\tau^v_{e_v}(\alpha_v)$}}
\psfrag{e}[Bc][Bc]{\scalebox{.8}{$e_v$}} \psfrag{v}[Bc][Bc]{\scalebox{.8}{$v$}}\rsdraw{.45}{.9}{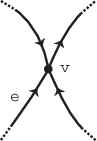} \quad \rsdraw{.45}{.6}{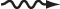} \quad \rsdraw{.45}{.9}{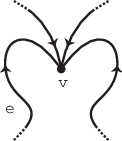} \quad \rsdraw{.45}{.6}{fleche} \quad \rsdraw{.45}{.9}{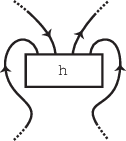} \,.
\end{center}
\caption{}
\label{fig-alg-inv1}
\end{figure}
This   transforms $G$ into   a planar diagram which determines, by
the
  Penrose calculus,  an element of $ \End_\cc(\un)=\kk$ denoted
$\inv_\cc(G)(\otimes_v\alpha_v) $. By linear extension, this
procedure defines a vector $\inv_\cc(G)\in
H(G)^\star=\Hom_\kk(H(G),\kk)$.

\begin{lem}\label{thm-R2}
  The vector  $\inv_\cc (G) \in
H(G)^\star$ is a well-defined isotopy invariant of a $\cc$-colored
graph $G$ in $\R^2$.
 More precisely, for any isotopy $\iota$ between two $\cc$-colored graphs $G$, $G'$ in $\R^2$,
 we have
 $\inv_\cc(G')\, H(\iota)=\inv_\cc(G)$, where $H(\iota)\co H(G) \to H(G')$ is the isomorphism induced by $\iota$.
\end{lem}

\begin{proof} Independence of $\inv_\cc (G)$ of the choice of the half-edges $e_v$ follows from the definition of $H_v(G)$.
Invariance under isotopies follows from the pivotality of $\cc$.
\end{proof}

For example, if $G=S^1\subset \R^2$ is the unit circle  with  one
vertex $v=(1,0)$ and one edge oriented clockwise and colored with
$X\in \Ob (\cc)$, then $E_v$ consists of two elements labeled by
$(X,+)$,
  $(X, -)$ and  $H(G)=H_v(G)\cong \Hom_{\cc} (\un, X^* \otimes X)$. Here
 $\inv_\cc(G)(\alpha)= \lev_X   \alpha \in \End_\cc(\un)=\kk$ for all $\alpha\in   H(G) $.

\subsection{Invariants of graphs  in $S^2$}\label{sect-graph-S2} If $\cc$ is spherical, then the invariant $\inv_\cc$ generalizes to graphs
in the 2-sphere
 $S^2=\R^2 \cup \{\infty\}$  (the orientation of~$S^2$ extends the
counterclockwise orientation in $\R^2$). Let $G$ be a $\cc$-colored
graph in   $S^2$. Pushing, if necessary, $G$ away from $
 \infty $, we obtain   a $\cc$-colored   graph $G_0$ in $
\R^2$. By Section~\ref{sphesphe},  $\inv_\cc(G) =\inv_\cc(G_0)\in
H(G_0)^\star=H(G)^\star$ does not depend on the way we push $G$
in~$\R^2$ and is an isotopy invariant of $G$.

We state  a few
properties of $\inv_\cc$.
\begin{enumerate}
\renewcommand{\labelenumi}{{\rm (\roman{enumi})}}
\item If a $\cc$-colored graph $G' \subset S^2$ is obtained from a
$\cc$-colored graph $G \subset S^2$ through the replacement of the
color $X $ of an edge by an isomorphic object $X'$, then any
isomorphism $ X\to X'$ induces an isomorphism $\phi \co H(G)\to H(G')$
in the obvious way and $\phi^\star (\inv_\cc(G'))=\inv_\cc(G)$. We
call this property of $\inv_\cc$ \emph{naturality}.

\item If   an edge~$e$ of a $\cc$-colored graph $G \subset S^2$ is
colored by $\un$ and   the endpoints of~$e$ are also  endpoints of
 other edges of~$G$, then
  $\inv_\cc(G)=\inv_\cc(G\setminus
 \Int(e))$. Indeed, in the  Penrose calculus, $e$
   can  be deleted without changing the
associated morphism.

\item If a $\cc$-colored graph $G' \subset S^2$ is obtained from a
$\cc$-colored graph $G \subset S^2$ through the replacement of the
color $X $ of an edge $e$  by $X^*$ and the reversion of the
orientation of $e$, then the isomorphism $\psi_X\co X\to X^{**}$ of
Remark~\ref{rem-pivotal} induces an isomorphism $\widehat{\psi} \co
H(G)\to H(G')$   and $(\widehat{\psi})^\star
(\inv_\cc(G'))=\inv_\cc(G)$.

\item We have $H(G\amalg G')=H(G) \otimes H (G')$ and
$\inv_\cc(G\amalg G')=\inv_\cc(G) \otimes \inv_\cc(G')$ for   any
disjoint $\cc$-colored graphs $G, G'\subset S^2$.
\end{enumerate}

\begin{exa}\label{exa-simplegraph}
For an $n$-tuple $S=((X_1,\varepsilon_1),\ldots,
(X_n,\varepsilon_n))$ of signed objects of $\cc$, consider  the
following $\cc$\ti colored graph $\gamma=\gamma_S$ in $S^2$:
\begin{equation*}
\psfrag{v}[Bc][Bc]{\scalebox{.7}{}} \psfrag{u}[Bc][Bc]{\scalebox{.7}{}}
\psfrag{X}[Bl][Bl]{\scalebox{.7}{$X_1$}}  \psfrag{Y}[Bl][Bl]{\scalebox{.7}{$X_2$}}
\psfrag{u}[Bc][Bc]{\scalebox{.8}{$u$}}  \psfrag{v}[Bc][Bc]{\scalebox{.8}{$v$}} \psfrag{Z}[Bl][Bl]{\scalebox{.7}
{$X_{n-1}$}} \psfrag{T}[Bl][Bl]{\scalebox{.7}{$X_n$}} \gamma=\,\rsdraw{.45}{.9}{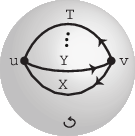}\,.
\end{equation*}
The graph $\gamma$ consists of $n$ edges connecting the   vertices $u$
and $v$, the $r$-th edge being colored with $X_r$ and oriented
towards $v$ if $\varepsilon_r=+$ and towards $u$ otherwise (the
picture corresponds to $\varepsilon_1=-$, $\varepsilon_2=+$,
  $\varepsilon_n=-$). Then
 $H(\gamma)=H({S^*}) \otimes H(S)$
  and $\inv_\cc(\gamma) =\omega_S\co H({S^*}) \otimes H(S) \to \kk$.
 The vertices $u$,  $v$ of $\gamma$ are in
duality  induced by the symmetry with respect to the vertical line
and $\omega_{u,v}=\omega_S=\inv_\cc(\gamma )$.
\end{exa}

\begin{exa}\label{exa-6j} For   any
  $i,j,k,l,m,n \in \Ob(\cc)$, consider the following
$\cc$-colored graph   in~$S^2$:
\begin{equation*}
\Gamma= \psfrag{i}[Bc][Bc]{\scalebox{.7}{$i$}} \psfrag{j}[Bc][Bc]{\scalebox{.7}{$j$}} \psfrag{k}[Bc][Bc]{\scalebox{.7}{$k$}}
\psfrag{l}[Bc][Bc]{\scalebox{.7}{$l$}} \psfrag{m}[Bc][Bc]{\scalebox{.7}{$m$}} \psfrag{n}[Bc][Bc]{\scalebox{.7}{$n$}}
\psfrag{e}[Bc][Bc]{\scalebox{.8}{$e_v$}} \psfrag{v}[Bc][Bc]{\scalebox{.8}{$v$}}
\rsdraw{.45}{.9}{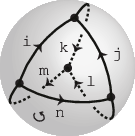}\,.
\end{equation*}
Here $H (\Gamma )$ is   the   tensor product of  the modules $H(m+,
i-,n-)$, $H(j+, i+,k-)$, $H(n+, j-,l-)$, and $H(l+, k+,m-)$. The
vector $ \inv_\cc(\Gamma )\in  H (\Gamma )^\star $ and similar
vectors   associated with other orientations of the edges
of~$\Gamma$   form a family of $2^6=64$ tensors  called
\emph{$6j$-symbols} associated with $i,j,k,l,m,n$. For more on this,
see Appendix.
\end{exa}

\section{Pre-fusion and  fusion categories}

\subsection{Pre-fusion   categories}\label{sect-fusion}
An object $X$ of a monoidal \kt category $\cc$ is
\emph{simple} if $\End_\cc(X)$ is a free \kt module of rank 1.
Equivalently, $X$ is simple if the \kt homomorphism $\kk \to
\End_\cc(X),\, k   \mapsto  k\, \id_X$  is an isomorphism. By the
definition of a monoidal \kt category, the unit object $\un$ is
simple. In the literature,  simple objects in our sense are often called absolutely simple objects or scalar objects.

A \emph{pre-fusion category} (over $\kk$) is a pivotal
\kt category $\cc$ such that
\begin{enumerate}
  \renewcommand{\labelenumi}{{\rm (\alph{enumi})}}
  \item Any finite family of objects of $\cc$ has a   direct sum in~$\cc$;
\item Each object of $\cc$ is a finite direct sum of simple objects;
\item  For any non-isomorphic simple objects $i,j  $ of $\cc$, we have $\Hom_\cc(i,j)=0$.
\end{enumerate}
Conditions (b) and (c) imply that all the $\Hom$ spaces in $\cc$ are
free $\kk$-modules of finite rank. The \emph{multiplicity} of a
simple object $i$ in any  $X\in \Ob (\cc)$ is the integer
\begin{equation*}
N^i_X=\rank_\kk \,\Hom_\cc(X,i)=\rank_\kk \,\Hom_\cc(i,X) \geq 0.
\end{equation*}
This integer depends only on the isomorphism classes of $i$ and $X$.

A  set $I$ of   simple objects of a pre-fusion category $\cc$ is
\emph{representative} if $\un \in I$ and every simple object of
$\cc$ is isomorphic to a unique element of~$I$.     Condition~(b)
above implies that for such~$I$ and any $X\in \Ob (\cc)$, there is a
finite family of   morphisms $(p_\alpha \co X \to i_\alpha, q_\alpha
\co i_\alpha \to X)_{\alpha \in \Lambda }$ in~$\cc$    such that
\begin{equation*}
\id_X=\sum_{\alpha \in \Lambda} q_\alpha p_\alpha, \quad i_\alpha\in I, \quad
\text{and} \quad p_\alpha q_\beta=\delta_{\alpha,\beta} \, \id_{i_\alpha} \quad
\text{for all} \quad \alpha,\beta\in \Lambda,
\end{equation*}
where $\delta_{\alpha,\beta}$ is the Kronecker symbol. Such a family
$( p_\alpha  , q_\alpha )_{\alpha \in \Lambda }$  is called  an
\emph{$I$-partition} of $ X$.   For $i \in I$, set
$\Lambda^i=\Lambda^i_X=\{ \alpha \in \Lambda \, | \, i_\alpha=i\}$.
Then $(p_\alpha \co X \to i)_{\alpha\in \Lambda^i}$ is a basis of
$\Hom_\cc(X,i)$ and $(q_\alpha \co i \to X)_{\alpha\in \Lambda^i}$
is a basis of $\Hom_\cc(i,X)$. Therefore $\# \Lambda^i=N_X^i$, $\#
\Lambda=\sum_{i \in I} N_X^i$, and
 $\dim(X)=\sum_{i\in I} \dim(i) N_X^i$.

\begin{lem}\label{lem-prefusion} Let $\cc$ be a pre-fusion category. Then:
\begin{enumerate}
  \renewcommand{\labelenumi}{{\rm (\alph{enumi})}}
   \item $\cc$ is non-degenerate;
  \item The left and right dimensions of any simple object of $\cc$ are invertible in~$\kk$;
  \item  $\cc$ is spherical if and only if
$\dim_l(i)=\dim_r(i)$ for any simple object $i$ of~$\cc$.
\end{enumerate}
\end{lem}
\begin{proof}
Let   $I$ be a representative set of   simple objects of $\cc$.
  Pick $X \in \Ob (\cc)$ and an  $I$-partition
$(p_\alpha \co X \to i_\alpha, q_\alpha \co i_\alpha \to X)_{\alpha
\in \Lambda }$   of $ X$. For  any endomorphism $f$ of   $X $ and
$\alpha\in \Lambda$, we have $p_\alpha f q_\alpha=\lambda_\alpha
\id_{i_\alpha}$ for some
 $\lambda_\alpha\in \kk$. Then
$$\tr_{l }(f)=\tr_{l } (f\sum_{\alpha \in \Lambda} q_\alpha p_\alpha)=\tr_{l } (\sum_{\alpha \in \Lambda} p_\alpha f q_\alpha )=
 \sum_{\alpha \in \Lambda} \lambda_\alpha
\dim_{l }(i_\alpha)$$ and similarly $\tr_r(f)=\sum_{\alpha \in \Lambda}
\lambda_\alpha \dim_r(i_\alpha)$. This implies Claim (c).

The families $(p^*_\alpha \co \un^*= \un \to X^*)_{\alpha\in
\Lambda^\un}$ and $(q_\alpha \co \un \to X)_{\alpha\in \Lambda^\un}$
are bases of the modules $\Hom_\cc(\un,X^*)$ and $\Hom_\cc(\un,X)$,
respectively, where $\Lambda^\un=\{ \alpha \in \Lambda \, | \,
i_\alpha=\un\}$.
 The pairing $\omega_X$ between these modules defined in
Section~\ref{sect-non-degen-contraction} is non-degenerate because
$\omega_X(p^*_\alpha \otimes q_\beta)=\lev_X(p^*_\alpha \otimes
q_\beta)= p_\alpha q_\beta=\delta_{\alpha,\beta}$ for all
$\alpha,\beta\in \Lambda^\un$. This implies~(a).

Let $i$ be a simple object of   $\cc$. Then $N_{i^* \otimes
i}^\un=1$ and so  there are morphisms $p \co i^* \otimes i \to \un$
and $q \co \un\to i^* \otimes i$ such that  $\Hom_\cc(i^* \otimes
i,\un)=\kk \cdot p$,   $\Hom_\cc(\un, i^* \otimes i)=\kk \cdot q$,
and $pq=\id_{\un}=1\in \kk$. Since $i$ is a simple object, $\lev_i$
is a basis of $\Hom_\cc(i^* \otimes i,\un)$ and $\rcoev_i$ is a
basis of $\Hom_\cc(\un, i^* \otimes i)$. Therefore $\lev_i=\lambda\,
p$ and $\rcoev_i=\mu q$ for some invertible   $\lambda,\mu \in \kk$.
Hence $\dim_l(i)=\lev_i\rcoev_i=\lambda\mu pq=\lambda\mu$ is
invertible. Since $i^*$ is simple, $\dim_r(i)=\dim_l(i^*)$ is
invertible. This proves Claim (b).
\end{proof}

For spherical $\cc$,   we can consider
 the invariant $\inv_\cc$ of  $\cc$-colored
graphs in   $S^2$, see Section~\ref{sect-graph-S2}. The following
lemma provides useful local relations for $\inv_\cc$.

\begin{lem}\label{lem-calc-diag} Let   $\cc$  be a spherical pre-fusion category   and let $I$ be a
representative set   of simple objects of $\cc$.
\begin{enumerate}
  \renewcommand{\labelenumi}{{\rm (\alph{enumi})}}
    \item For any  $i,j \in I$,
  \begin{equation*}
   \psfrag{i}[Bl][Bl]{\scalebox{.7}{$i$}}\psfrag{j}[Bl][Bl]{\scalebox{.7}{$j$}}
\psfrag{u}[Bl][Bl]{\scalebox{.8}{$v$}} \psfrag{v}[Bl][Bl]{\scalebox{.8}{$v'$}}
\inv_\cc \!\begin{pmatrix}
\;\rsdraw{.45}{.9}{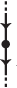}
\end{pmatrix} = \delta_{i,j} \; \inv_\cc \!\begin{pmatrix}
\;\rsdraw{.45}{.9}{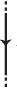}
\end{pmatrix} \, .
\end{equation*}
  \item For any $i,j \in I$,
  \begin{equation*}
   \psfrag{i}[Bl][Bl]{\scalebox{.7}{$i$}}\psfrag{j}[Bl][Bl]{\scalebox{.7}{$j$}}
\psfrag{u}[Bl][Bl]{\scalebox{.8}{$v$}} \psfrag{v}[Bl][Bl]{\scalebox{.8}{$v'$}}
\inv_\cc \!\begin{pmatrix}
\rsdraw{.45}{.9}{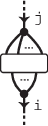}
\end{pmatrix} = \delta_{i,j} (\dim(i))^{-1} \; \inv_\cc \!\begin{pmatrix}
\rsdraw{.45}{.9}{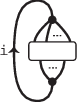}
\end{pmatrix} \otimes \inv_\cc \!\begin{pmatrix}
\;\rsdraw{.45}{.9}{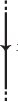}
\end{pmatrix}.
\end{equation*}
  \item \hfill $\displaystyle  \psfrag{i}[Bc][Bc]{\scalebox{.7}{$i$}}
\psfrag{u}[Bl][Bl]{\scalebox{.8}{$u$}}
\psfrag{v}[Bl][Bl]{\scalebox{.8}{$v$}} \inv_\cc
\!\begin{pmatrix} \;\rsdraw{.45}{.9}{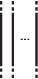}\;
\end{pmatrix} = \sum_{i \in I} \dim(i) \ast_{u,v} \inv_\cc \! \begin{pmatrix}
\;\rsdraw{.45}{.9}{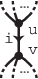}\;
\end{pmatrix}.
$ \hfill \phantom{X} \\[1em]
  \item \hfill $\displaystyle
\psfrag{u}[Bl][Bl]{\scalebox{.8}{$u$}} \psfrag{v}[Bl][Bl]{\scalebox{.8}{$v$}}
\inv_\cc \!\begin{pmatrix}
\;\rsdraw{.45}{.9}{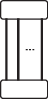}\;
\end{pmatrix} = \ast_{u,v} \, \inv_\cc \!\begin{pmatrix}
\;\rsdraw{.45}{.9}{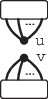}\;
\end{pmatrix}.
$ \hfill \phantom{X}
\end{enumerate}
\end{lem}

In (b) and (d) the empty rectangles stand  for   pieces of
$\cc$-colored graphs   sitting inside the rectangles. The same
$\cc$-colored graphs appear  on both sides of the equalities. The
  relations  (a) and (b) can be applied only when there are other vertices (not shown in
the picture) on the single strand on the right-hand side.  For the
definition of~$\ast_{u,v}$  in
  (c), (d), see  Section~\ref{sect-graph}; here the
duality between the vertices $u$ and $v$ is induced by the symmetry
with respect to a horizontal line. Note that for any choice of
colors of the  strands on the left-hand side of~(c), the sum on the
right-hand side has only a finite number of non-zero terms.

\begin{proof}
Claims (a) and (b) follow  from the equalities $\Hom_\cc(i,j)=0$ for
$i \neq j$ and $f=(\dim(i))^{-1}\, \tr(f) \, \id_i$ for any $f\in
\End_\cc(i)$. Claim (d) follows from (c) since $\Hom_\cc(\un,i)=0$
for $i \neq \un$. Let us prove (c). Let
$S=((X_1,\varepsilon_1),\ldots, (X_n,\varepsilon_n))$ be the  tuple
of signed objects of~$\cc$ determined by the strands on the
left-hand side of~(c). Let $X=X_S\in \Ob (\cc)$ be defined by
Formula~\eqref{Evev-}. Let $( p_\alpha \co X \to i_\alpha, q_\alpha
\co i_\alpha \to X )_{\alpha \in \Lambda }$ be an $I$-partition of~$
X$.    For   $i\in I$, set $\Lambda^i =\{\alpha \in \Lambda \,
\vert\, i_\alpha=i\}$ and denote by $S_i$ the $(n+1)$-tuple of
signed objects $ (i,-)S$. The
 families  $$(e_{\alpha}=(\id_{X^*} \otimes p_{ \alpha})\,
\rcoev_{X})_{\alpha\in \Lambda^i} \quad {\text {and}}\quad
(f_{\alpha}=(\id_{i^*} \otimes q_{ \alpha})\, \rcoev_{i}
)_{\alpha\in \Lambda^i}$$ are bases of the modules $\Hom_\cc(\un,
X_{}^* \otimes i)$ and $\Hom_\cc(\un, i^* \otimes X)=\Hom_\cc(\un,
X_{S_i})$, respectively. Recall the isomorphism $\psi_{S^*}\co
X_{S^*} \to X_S^*=X^*$ defined in the proof of
Lemma~\ref{lem-pairing-non-degen}. Then $(e'_{\alpha} =
(\psi_{S^*}^{-1} \otimes \id_{i })(e_\alpha))_{\alpha\in \Lambda^i}$
is a basis of $\Hom_\cc(\un, X_{S^*} \otimes i)=\Hom_\cc(\un,
X_{S_i^*}  )$.
 Recall from Section~\ref{sect-dual-signed-set} the evaluation form $$\omega=\omega_{S_i}\co \Hom_\cc(\un, X_{S_i^*}  ) \otimes
\Hom_\cc(\un, X_{S_i}  ) \to \kk. $$   A direct
 computation gives
$\omega (e'_{ \alpha} \otimes f_{ \beta})=\delta_{\alpha,\beta}
\dim(i)$ for all $\alpha, \beta\in \Lambda^i$. The  inverse of~$
\omega  $ is the tensor $ (\dim(i))^{-1} \sum_{\alpha \in \Lambda^i}
f_{\alpha}\otimes e'_{ \alpha}$  written shortly as $ \Omega_{i}
\otimes \Omega'_{i}$. We have
\begin{equation*}
\psfrag{i}[Bc][Bc]{\scalebox{.7}{$i$}} \psfrag{X}[Bc][Bc]{\scalebox{.7}{$X_S$}} \psfrag{Y}[Bl][Bl]{\scalebox{.7}{$X_{S^*}$}}
\psfrag{R}[Bc][Bc]{\scalebox{.8}{$\Ev_{S^*}$}} \psfrag{F}[Bc][Bc]{\scalebox{.8}{$\psi^{-1}_{S^*}$}} \psfrag{E}[Bc][Bc]{\scalebox{.8}{$\psi_{S}$}}
\psfrag{a}[Bc][Bc]{\scalebox{.8}{$\Omega'_{i}$}} \psfrag{u}[Bc][Bc]{\scalebox{.8}{$\Omega_{i}$}}
\psfrag{x}[Bc][Bc]{\scalebox{.8}{$p_{ \alpha}$}} \psfrag{o}[Bc][Bc]{\scalebox{.8}{$q_{ \alpha}$}}
\sum_{i \in I} \dim(i)\,\rsdraw{.45}{.9}{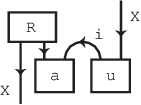} \, = \psfrag{i}[Bc][Bc]{\scalebox{.7}{$i_\alpha$}} \sum_{ \alpha\in \Lambda}\, \rsdraw{.45}{.9}{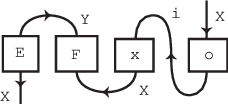} \, = \sum_{ \alpha\in \Lambda}  q_{ \alpha} p_{ \alpha} =\id_{X_S}.
\end{equation*}
This formula and the definition of the contraction maps imply (c).
\end{proof}

\subsection{Fusion   categories}\label{sect-fusion+} By a \emph{fusion
category}, we mean a pre-fusion category~$\cc$ such that   the set
of isomorphism classes of simple objects of~$\cc$   is finite. A
standard example of a fusion category is the   category of finite
rank representations (over~$\kk$) of a finite group  whose order is
relatively prime to  the characteristic of~$\kk$. The category of
representations of an involutory finite dimensional Hopf algebra
over a field of characteristic zero is a fusion category. For more
examples, see~\cite{ENO}.

A representative set $I$ of   simple objects of a fusion category
$\cc$ is finite. The following sum does not depend on the choice of
$I$ and is called the \emph{dimension} of~$\cc$:
\begin{equation*}
\dim (\cc)=\sum_{i \in I} \dim_l( i) \dim_r( i) \in \kk.
\end{equation*}
   By \cite{ENO},
  if   $\kk $ is an algebraically closed field of characteristic zero, then $\dim (\cc)\neq 0$. For spherical~$\cc$, we
have $\dim (\cc)=\sum_{i \in I} (\dim( i))^2$.

\begin{lem}\label{bubbleidentity} Let  $I$ be  a  representative set of   simple objects of a spherical fusion category~$\cc$.
Then for    all $X,Y \in \Ob(\cc)$,
\begin{equation*}
\sum_{k,l\in I} \dim(k)\dim(l) N_{X
\otimes k \otimes Y \otimes l}^\un  =   \dim (X) \dim(Y) \dim(\cc)
\end{equation*}
and
\begin{equation*}
\sum_{j,k,l\in I} \dim(k)\dim(l) N_{j \otimes k\otimes j^* \otimes l}^\un  =   (\dim(\cc))^{2}.
\end{equation*}
\end{lem}
\begin{proof}
We have
\begin{gather*}
\sum_{k,l\in I} \dim(k)\dim(l) N_{X \otimes k\otimes Y \otimes  l}^\un
=\sum_{k\in I} \dim(k) \sum_{l\in I}\dim(l^*) N_{X \otimes
k\otimes Y}^{l^*}\\
=\sum_{k\in I} \dim(k) \sum_{m\in I}\dim(m) N_{X
\otimes k\otimes Y}^{m}   =\sum_{k\in I} \dim(k) \dim(X \otimes k\otimes Y)\\
 =
 \dim(X)  \dim(Y) \sum_{k\in I} (\dim(k))^2=  \dim (X) \dim(Y)  \dim(\cc).
 \end{gather*}
The second formula is a direct consequence of the first one.
\end{proof}

\subsection{The opposite  category}\label{sect-oppos} Each monoidal category~$\cc$ determines the
\emph{opposite} monoidal category $\cc^\opp$   by reversing all the
arrows (and keeping the tensor product and the unit object).
If~$\cc$ is pivotal with the evaluation/coevaluation morphisms
$\{\lev_X, \lcoev_X,  \rev_X   , \rcoev_X\}_{X\in \Ob(\cc)}$,
then~$\cc^\opp$ is pivotal with the evaluation/coevaluation
morphisms $$ \lev_X^\opp=\rcoev_X, \;\;  \lcoev_X^\opp= \rev_X, \;\;
\rev_X^\opp=\lcoev_X, \;\; \rcoev_X^\opp=\lev_X ,$$ where $X\in
\Ob(\cc^\opp)=\Ob(\cc)$. If~$\cc$ is a monoidal
\kt category (respectively, a spherical \kt category, a pre-fusion category, a fusion category), then so
is~$\cc^\opp$.

\section{State sums on triangulated 3-manifolds}\label{sec-ssot3m}

In this section, we derive  a topological invariant  of closed
oriented 3-manifolds from a spherical fusion category.

\subsection{Preliminaries on triangulated   3-manifolds}\label{sect-skeletons-} Let $M$ be a     3-manifold (without boundary). Let~$t$ be a
   triangulation of~$M$.
 The  simplices of~$t$ of dimension 0, 1, 2, 3 are called respectively vertices, edges,   faces, and tetrahedra.
 For a vertex~$x$ of~$t$, the
 union of all simplices of~$t$ incident to~$x$ (i.e., containing~$x$) is a closed 3-ball $B\subset M$.
 The simplices of~$t$ lying
 in~$\partial B $ form a triangulation  of~$\partial B$.
 Let
  $\Gamma  \subset \partial B$ be the 1-skeleton of this triangulation, i.e., the union of    vertices and edges of~$t$ contained in~$\partial B$.
   The pair $(\partial B, \Gamma)$ is called the \emph{link} of~$x$ in~$(M,t)$.

 The link of~$x$   can be visualized  in a small neighborhood of~$x$.
Pick a
 (small) closed 3-ball   $B_x \subset M$  centered at~$x$   such that $B_x$ meets every simplex of~$t$
incident to~$x$ along a    concentric subsimplex with vertex~$x$.
The link of~$x$ can be identified with  the pair $(\partial B_x,
\Gamma_x) $, where~$\Gamma_x $ is the intersection of~$\partial B_x$
with the 2-skeleton~$t^{(2)}$  of~$t$ (formed by the simplices
of~$t$ of dimension $\leq 2$). The vertices (resp.\ edges)
of~$\Gamma_x $ are the intersections of~$\partial B_x$ with the
edges (resp.\ faces) of~$t$ incident to~$x$.

In the sequel, the set of all faces of~$t$ is denoted by $\Reg(t)$.
By an orientation of the 2-skeleton~$t^{(2)}$ of $t$, we mean a
choice of orientation in all faces of~$t$ (we impose no
compatibility conditions on these  orientations of faces).

\subsection{Invariants of  3-manifolds}\label{sec-Io3m}  Let~$\cc$ be  a   spherical fusion
category
over~$\kk$ whose dimension is invertible in~$\kk$. Fix a
 (finite)
representative set $I$ of simple objects of $\cc$. For each closed
oriented   3-manifold $M$, we   define a topological invariant
$|M|_\cc \in \kk$.

Pick a triangulation $t$ of $M$ with oriented 2-skeleton and a map
$c \co \Reg(t) \to I$. For each oriented edge $e$ of $t $,
  we define a \kt module $H_c(e)$
as follows. The orientations of~$e$ and~$M$ determine a positive
direction on a small loop in~$M$ encircling~$e$; this direction
determines a cyclic order on the set $t_e$ of all faces of~$t$
containing~$e$. To each   face   $r\in t_e$ we assign the object
$c(r) \in I$ and a sign equal to $ +$ if the orientations of $r$
and~$e$ are compatible   and   to $ -$ otherwise. (The orientations
of $r$ and~$e$ are compatible  if each pair (a tangent vector
directed outward $r$ at a point of~$e$, a positive tangent vector
of~$e$) is positively oriented in~$r$.)
 In this way,~$t_e$ becomes a cyclic $\cc$\ti set. Set
 $H_c(e)=H (t_e)$. If $e^\opp$ is the same edge  with opposite orientation,
 then   $t_{e^\opp}=(t_e)^\opp$. This induces a   duality between
 the modules
 $H_c(e)$, $H_c(e^\opp)$ and a
 contraction  $\ast_e \co H_c(e^\opp)^\star \otimes H_c(e)^\star \to\kk $, see Section~\ref{sect-non-degen-contraction}.
 Note that the contractions $\ast_e$ and $\ast_{e^\opp}$ are
 equal up to permutation of the tensor factors.

Consider the link $(\partial B_x, \Gamma_x)   $ of a vertex $x $
of~$t$. Every edge $\alpha$ of $\Gamma_x $ lies   in a
face~$r_\alpha$ of $t$ incident to~$x$. We color $\alpha$ with
$c(r_\alpha)\in I$ and endow $\alpha$ with the orientation induced
by that of $r_\alpha \setminus \Int (  B_x)$. In this way, $\Gamma_x
$ becomes a $\cc$-colored graph  in~$\partial B_x$. We identify
$\partial B_x$ with the standard 2-sphere $S^2$ via an
 orientation preserving homeomorphism, where  the orientation of $\partial B_x$ is induced
by that of~$M$ restricted to $M\setminus \Int(B_x)$.  Then
 $H (\Gamma_x)=\otimes_{e\in t_x} H_c(e)$, where
 $t_x$ is the set of all edges of~$t$ incident to $x$ and oriented
away from
  $x$. Section~\ref{sect-graph-S2}
gives    $\inv_\cc (\Gamma_x) \in H (\Gamma_x)^\star$. The tensor
product $\otimes_x \,\inv_\cc (\Gamma_x)$ over all vertices $x$ of
$t$ is a vector in
 $\otimes_e \, H_c(e )^\star$, where~$e$
 runs over all oriented edges of $t$. The contractions $\ast_e$
 apply to different tensor factors, and their tensor product
$\ast_t=\otimes_e \ast_e$   is a
 map  $\otimes_e \, H_c(e )^\star \to \kk$.
Set \begin{equation}\label{eq-simplstatesum}|M|_\cc=(\dim
(\cc))^{-\vert t\vert} \, \sum_{c} \, \left (\prod_{r \in \Reg(t)}
\dim c(r)\right )   \, {\ast}_t ( \otimes_x \,\inv_\cc (\Gamma_x))
\in \kk,\end{equation} where $\vert t\vert$ is the number of
tetrahedra of~$t$ and $c$ runs over all maps $ \Reg(t) \to I$. The
following theorem will be proved in
Section~\ref{sec-skeletonsstatesums+}.

\begin{thm}\label{thm-state-3man-smipl}
$|M|_\cc$ is a topological invariant of $M$ (independent of the
choice   of~$t$  and of the  orientation of~$t^{(2)}$). This
invariant does not depend on the choice of~$I$.
\end{thm}

It is clear from the definitions that $|M \amalg N|_\cc= |M|_\cc \,
|N|_\cc$ for any oriented closed 3-manifolds $M, N$. One can show
that $|-M|_{\cc}= |M|_{\cc^{\opp}}$, where $-M$ is $M$ with opposite
orientation. We prove below that $|S^3|_\cc= (\dim (\cc))^{-1}$ and
$|S^1\times S^2|_\cc=1$.

The  invariant $|M|_\cc$ can be viewed as a state sum (or a
partition function)   on~$t$ as follows. Provide all symmetrized
multiplicity modules in~$\cc$   with distinguished  bases.
 By  states on~$t$, we mean pairs
consisting of a map $c \co \Reg(t) \to I$ and a choice of a basis
vector   $b_c(e)\in H_c(e)$ for every oriented edge $e$ of $t$. Let
$b^\star_c(e)$ be the corresponding vector in the basis of
$H_c(e)^\star$ dual to the distinguished basis of~$H_c(e)$. The
Boltzmann weight associated with such a state is the product of the
factors $\inv_\cc (\Gamma_x) (\otimes_{e\in t_x} b_c(e))$,
$\ast_e(b^\star_c(e^{\opp}), b^\star_c(e))=\ast_{e^{\opp}} (
b^\star_c(e), b^\star_c(e^{\opp}))$, $\dim c(r)$, and $(\dim
(\cc))^{-1}$ contributed respectively by the vertices, edges, faces,
and tetrahedra of~$t$. The invariant $|M|_\cc$ is the sum of these
Boltzmann weights over all states on~$t$. This state sum differs
from the one in \cite{TV, BW}, where the states are labelings
of the edges of~$t$ with elements of~$I$ and the key factor in the
Boltzmann weight is a $6j$-symbol contributed by every tetrahedron.
It is non-obvious but true that these two state sums are equal, see
Section~\ref{sec-skeletonsstatesums}. In particular, our invariant
$|M|_\cc$ is equal to the invariant of~$M$ defined by Barrett and
Westbury \cite{BW}.

\subsection{Remark} We say that a triangulated
 surface~$\Sigma$ is \emph{$I$-colored} if every edge of~$\Sigma$ is oriented and endowed with an element of~$I$.
 In other words,~$\Sigma$ is  $I$-colored  if its
  1-skeleton
$\Sigma^{(1)}$   is a $\cc$-colored graph in~$\Sigma$ with colors of
all edges being in~$I$. For such~$\Sigma$,  we can consider the \kt
module $H_\cc (\Sigma)=H(\Sigma^{(1)}) =\otimes_x \,
H_x(\Sigma^{(1)})$, where~$x$ runs over all vertices of~$\Sigma$.
The state sum~\eqref{eq-simplstatesum}    extends to compact
oriented 3-manifolds~$M$ with $I$-colored triangulated boundary and
gives a vector $|M |_\cc \in H_\cc (\partial M)^\star$. These
vectors can be used to define a 3-dimensional TQFT. We will  discuss
a more general construction in Section~\ref{sec-TQFT}.

\section{State sums on skeletons of 3-manifolds}\label{sec-skeletonsstatesums}

We compute $|M |_\cc$ as  a state sum on any skeleton  of~$M$.  This
generalizes  the state sums of Section~\ref{sec-ssot3m} and of
\cite{TV, BW}.   We begin with topological preliminaries.

\subsection{Stratified 2-polyhedra} By a  \emph{2-polyhedron}, we mean a compact topological space that can be triangulated
using only simplices of dimension $\leq 2$.
For a 2-polyhedron $P$, denote by $\Int(P)$ the subspace of~$P$
consisting of  all points having a neighborhood homeomorphic to
$\RR^2$. Clearly, $\Int(P)$ is an (open) 2-manifold without
boundary. By an \emph{arc} in~$P$, we mean the image of a path
$\alpha\co [0,1]\to P$ which is an embedding except that
  possibly $\alpha(0)=\alpha(1)$. (Thus, arcs may be loops.) The
  points $\alpha(0)$, $\alpha(1)$ are the \emph{endpoints}   and the set $\alpha((0,1)) $ is the \emph{interior} of the
  arc.

   To work with polyhedra, we will use the language of     stratifications  as follows. Consider a   2-polyhedron~$P$  endowed with a
    finite     set of arcs~$E $    such that
\begin{enumerate}
  \renewcommand{\labelenumi}{{\rm (\alph{enumi})}}
    \item different arcs in~$E$ may meet only at their endpoints;
    \item   $ P \setminus   \cup_{e\in E} \, e \subset \Int (P)$ and $ P \setminus   \cup_{e\in E} \, e$ is
dense in~$P$.
\end{enumerate}
The arcs of~$E$ are called  \emph{edges} of~$P$ and their endpoints
    are called \emph{vertices} of~$P$.   The vertices and
edges  of $P$ form a
    graph $P^{(1)}= \cup_{e\in E} \, e  $. Since all vertices of~$P$
    are  endpoints of  the edges,~$P^{(1)}$ has no isolated vertices.
Cutting $P$ along  $P^{(1)}$, we obtain a compact surface
${\widetilde P}$ with interior $P \setminus P^{(1)}$. The
polyhedron~$P$ can be recovered by gluing~${\widetilde P}$
to~$P^{(1)}$ along a   map $ p\co
\partial {\widetilde P} \to P^{(1)} $. Condition (b) ensures the surjectivity of~$p$. We call the pair $(P, E)$  (or, shorter,~$P$)   a \emph{stratified 2-polyhedron}
if
 the set $p^{-1}({\text {the set of   vertices of}} \, P)$ is finite
and each component of the complement of this set in $\partial
{\widetilde P} $ is mapped homeomorphically onto the interior of an
edge of~$P$.

A 2-polyhedron~$P$ can be stratified if and only if $\Int(P)$ is
dense in~$P$. For such a~$P$, the edges of any triangulation form a
stratification. Another example: a closed surface with an empty set
of edges is a stratified 2-polyhedron.

For a   stratified 2-polyhedron~$P$, the connected components
of~${\widetilde P}$ are called \emph{regions} of~$P$. Clearly, the
set $\Reg(P) $ of the regions of~$P$ is finite.  For a vertex   $x$
of~$P$,  a \emph{branch}   of~$P$ at~$x$ is
 a germ  at~$x$ of a region of~$P$ adjacent to~$x$. (Formally, the branches of $P$ at~$x$ can be defined
 as
 paths
 $\gamma\co [0,1]\to P$ such that $\gamma(0)=x$ and $\gamma((0,1])\subset P \setminus
 P^{(1)}$, considered up to
  homotopy in the class of such paths.)  The set of branches of~$P$ at~$x$ is finite and
 non-empty. The branches
of $P$ at $x$ bijectively correspond to  the elements of the set~$
p^{-1} (x)$, where $ p\co
\partial {\widetilde P} \to P^{(1)} $ is the map above.
 Similarly, for an edge   $e$ of~$P$,  a \emph{branch}   of~$P$ at~$e$ is
 a germ  at~$e$ of a region of~$P$ adjacent to~$e$. (A   formal
 definition  proceeds in terms of paths  $\gamma\co [0,1]\to P$ such that $\gamma(0) $ lies in the interior of~$e$
 and $\gamma((0,1])\subset P \setminus P^{(1)}$.)
  The set of
branches of~$P$ at~$e$ is  denoted~$P_e$.   This set   is finite and
non-empty. There is a natural bijection between~$P_e$ and  the set
of connected components of $ p^{-1} ({\text {interior of}}\, e) $.
The number of elements of $P_e$ is the \emph{valence} of~$e$. The
edges of~$P$ of valence 1 and their vertices form a graph called the
\emph{boundary} of~$P$ and denoted~$\partial P$. We say that $P$ is
\emph{orientable} (resp.\@ \emph{oriented}) if all regions of~$ {
P}$ are orientable (resp.\@ oriented).

\subsection{Skeletons of closed  3-manifolds}\label{sect-skeletons}  A \emph{skeleton} of a  closed   3-manifold $M$ is
an oriented    stratified 2-polyhedron $P\subset M$ such that
$\partial P=\emptyset$ and  $M\setminus P$ is a disjoint union of
open 3-balls. An example of a skeleton of~$M$ is provided by the
(oriented) 2-skeleton~$t^{(2)}$ of a  triangulation~$t$ of~$M$,
where the   edges of~$t^{(2)}$ are the   edges of~$t$.

Any  vertex   $x$ of a  skeleton $P \subset M$ has a closed ball
neighborhood $B_x \subset M$ such that $\Gamma_x=P\cap \partial B_x$
is a finite non-empty graph and
  $ P\cap B_x $ is the cone over~$  \Gamma_x $.  The vertices of $\Gamma_x$ are the intersections of $\partial B_x$ with the
half-edges of $P$ incident to~$x$; the edges of $\Gamma_x$ are the
intersections of $\partial B_x$ with the branches of $P$ at $x$. The
condition $\partial P=\emptyset$ implies that every vertex of
$\Gamma_x$ is incident to at least two half-edges of $\Gamma_x$.
The pair
 $(\partial B_x, \Gamma_x)$  is determined by the triple $(M,P,x)$   up to   homeomorphism.
 This pair is   the \emph{link} of~$x$ in~$(M,P)$. If $M$ is
 oriented, then we
   endow~$\partial B_x$ with orientation induced
by that of~$M$ restricted to $M\setminus \Int(B_x)$. Then the link
 $(\partial B_x, \Gamma_x)$ of~$x$  is determined by   $(M,P,x)$   up to orientation preserving  homeomorphism.

\subsection{Computation of~$|M|_\cc$}\label{sec-computat}  Let $\cc$ and  $I$  be
as in Section~\ref{sec-Io3m} and let~$M$ be a closed oriented
3-manifold. We compute $|M|_\cc \in \kk$ as a state sum on any
skeleton~$P$ of~$M$. For a map $c \co \Reg(P) \to I$ and an oriented
edge $e$ of $P $, we define  a \kt module $H_c(e)=H(P_e)$,
where~$P_e$ is the set of branches of~$P$ at~$e$ turned into a
cyclic $\cc$\ti set as~$t_e$ in Section~\ref{sec-Io3m}. Here by the
$c$-color of a branch $b\in P_e$,
 we mean the $c$-color of the region of~$P$ containing  $b$.  If $e^\opp$ is the same edge  with opposite orientation,
 then   $P_{e^\opp}=(P_e)^\opp$. This induces a   duality between
 the modules
 $H_c(e)$, $H_c(e^\opp)$ and a
 contraction  $ \ast_e \co H_c(e)^\star \otimes H_c(e^\opp)^\star \to\kk$.

 \begin{figure}[t]
\begin{center}
\psfrag{i}[Bc][Bc]{\scalebox{.9}{$i$}}
\psfrag{j}[Bc][Bc]{\scalebox{.9}{$j$}}
\psfrag{k}[Bc][Bc]{\scalebox{.9}{$k$}}
\psfrag{l}[Bc][Bc]{\scalebox{.9}{$l$}}
\psfrag{m}[Bc][Bc]{\scalebox{.9}{$m$}}
\psfrag{n}[Bc][Bc]{\scalebox{.9}{$n$}}
\psfrag{q}[Bc][Bc]{\scalebox{.9}{$q$}}
\psfrag{p}[Bc][Bc]{\scalebox{.9}{$p$}}
\psfrag{s}[Bc][Bc]{\scalebox{.9}{$s$}}
\psfrag{t}[Bc][Bc]{\scalebox{.9}{$t$}}
\psfrag{x}[Bc][Bc]{\scalebox{.9}{$x$}}
\psfrag{1}[Bc][Bc]{\scalebox{.7}{$1$}}
\psfrag{2}[Bc][Bc]{\scalebox{.7}{$2$}}
\psfrag{3}[Bc][Bc]{\scalebox{.7}{$3$}}
\psfrag{G}[Br][Br]{\scalebox{1.1}{$\Gamma_x=$}}
\rsdraw{.45}{.9}{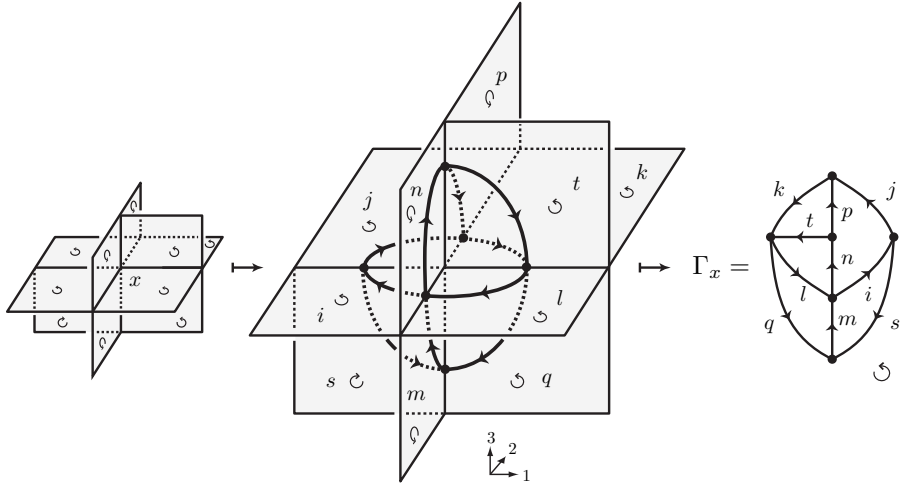}
\end{center}
\caption{The graph $\Gamma_x \subset S^2$ associated with a vertex $x$}
\label{fig-state-sum-algo}
\end{figure}

As in Section~\ref{sec-Io3m},  the link of a   vertex $x \in P$
determines a $\cc$-colored  graph $\Gamma_x$ in $\partial
B_x\simeq S^2$.  An example of $\Gamma_x$ is given in Figure 2; note that
since the orientation of $\partial B_x$ is induced by that of~$M$ restricted to $M\setminus \Int(B_x)$, the identification of $\partial B_x$
with the standard 2-sphere $S^2=\R^2\cup\{\infty\}$ (oriented counterclockwise) requires in this example a mirror reflection.

Section~\ref{sect-graph-S2} yields a tensor
$\inv_\cc (\Gamma_x) \in H_c(\Gamma_x)^*$.  By definition,
 $H_c(\Gamma_x)= \otimes_e\, H_c(e)$, where
 $e$ runs over all edges of $P$ incident to $x$ and oriented away from $x$ (an edge   with both endpoints in $x$ appears in this tensor product twice
 with opposite orientations). The tensor product $\otimes_x \,\inv_\cc (\Gamma_x)$ over all
vertices~$x$ of~$P$ is a vector in
 $\otimes_e \, H_c(e )^\star$, where   $e$
 runs over all oriented edges of $P$. Set
$\ast_P=\otimes_e \, \ast_e\co \otimes_e \, H_c(e )^\star \to \kk$.
The following theorem will be proved in
Section~\ref{sec-skeletonsstatesums+}.

\begin{thm}\label{thm-state-3man}
For any skeleton~$P$ of~$M$,
\begin{equation}\label{eq-simplstatesum+}|M|_\cc=(\dim (\cc))^{-\vert P\vert} \sum_{c} \,\,  \left ( \prod_{r \in \Reg(P)} (\dim c(r))^{\chi(r)} \right ) \,
  {\ast}_P ( \otimes_x \,\inv_\cc (\Gamma_x)) \in \kk,\end{equation}
where   $\vert P\vert$ is the number of components of $M\setminus
P$,  $c$ runs over all maps $ \Reg(P) \to I$, and $\chi(r)$ is the
Euler characteristic of $r$.
\end{thm}

  Formula~\eqref{eq-simplstatesum} is a special case
of~\eqref{eq-simplstatesum+} for   the oriented
2-skeleton~$P=t^{(2)}$ of a triangulation~$t$ of~$M$. When~$P$ is
the oriented 2-skeleton of the cellular subdivision of~$M$ dual
to~$t$ and the orientation of~$P$ is induced by that of~$M$ and a
total order on the set of vertices of~$t$,
Formula~\eqref{eq-simplstatesum+} is equivalent to the
Turaev-Viro-Barrett-Westbury state sum on~$t$.
Theorem~\ref{thm-state-3man} implies   that our definition
of~$|M|_\cc$ is equivalent to the one in   \cite{BW}.

We illustrate Theorem~\ref{thm-state-3man} with two examples. An
oriented 2-sphere embedded in~$S^3$ is a skeleton of~$S^3$
  with void set of edges and vertices and   one region.
Formula~\eqref{eq-simplstatesum+} gives $|S^3|_\cc=(\dim (\cc))^{-2}
\sum_{i\in I}(\dim (i))^2= (\dim (\cc))^{-1}$.    Pick a point $x\in
S^1$ and a circle $\ell\subset S^2$. The set $P=(\{x\} \times S^2 )
\cup (S^1\times \ell)$ is a skeleton of $S^1\times S^2$ with one
edge $\{x\}\times \ell$ viewed as a loop; the orientation  of the
three regions of~$P$ is arbitrary. Formula~\eqref{eq-simplstatesum+} and the second formula of Lemma~\ref{bubbleidentity}  give $|S^1\times S^2|_\cc=1$.

\section{Moves on skeletons and proof of   Theorems~\ref*{thm-state-3man-smipl}
and~\ref*{thm-state-3man}}\label{sec-skeletonsstatesums+}

The     proof of Theorems~\ref{thm-state-3man-smipl}
and~\ref{thm-state-3man} given at the end of the section is based on
a study of transformations of skeletons of a closed 3-manifold~$M$.

\subsection{Moves on skeletons}   The symbols $\# v$, $\# e$, $\# r$ will
denote  the number of vertices,   edges, and    regions
 of a skeleton of~$M$, respectively. We define four moves   $T_1-T_4$ on a
 skeleton $P$   of~$M$ transforming~$P$ into   a new skeleton of~$M$, see Figure~\ref{fig-moves}.
  The \lq\lq phantom edge  move"~$T_1$ keeps~$P$ as a polyhedron and  adds one new edge connecting   distinct vertices of
  $P$ (this edge is an arc in $P$ meeting
    $P^{(1)}$ solely at its endpoints and has the valence~2).
  This
  move  preserves $\# v$, increases~$\# e$ by 1, and either preserves $\# r$ or increases $\# r$ by~1.
The \lq\lq contraction move" $T_2$ collapses into a point an
 edge~$e$ of~$P$ with distinct endpoints.
 This
  move is allowed only when at least one endpoint of~$e$ is the endpoint of some other edge. The move~$T_2$  decreases both   $\# v$  and $\# e$  by 1 and  preserves $\# r$.
The \lq\lq percolation  move" $T_3$  pushes a branch~$b$ of~$P$
through a vertex~$x$ of~$P$. The branch~$b$ is pushed across a small
disk~$D$ lying in
   another branch of~$P$ at~$x$  so that $D\cap P^{(1)}= \partial D\cap
   P^{(1)}=\{x\}$ and both these branches are adjacent to the same component of $M\setminus P$.
   This move   preserves $\# v$ and  increases both $\# e$ and~$\# r$ by~1.
The \lq\lq bubble  move" $T_4$  adds to~$P$ an embedded disk
$D_+\subset M$ such that   $D_+\cap P=\partial D_+ \subset
P\setminus P^{(1)}$,  the circle $\partial D_+$ bounds a disk~$D_-$
in $ P\setminus P^{(1)}$, and the 2-sphere $D_+\cup D_-$ bounds a
ball in $M$ meeting~$P$ precisely
   along~$D_-$.
   A  point of the circle $\partial D_+ $ is chosen as a vertex and the circle itself is viewed
   as an edge of the resulting skeleton.
  This
   move      increases~$\# v$   and~$\# e$ by~1 and~$\# r$ by~2.

   \begin{figure}[h, t]
\begin{center}
\psfrag{T}[Bc][Bc]{\scalebox{.9}{$T_1$}} \psfrag{L}[Bc][Bc]{\scalebox{.9}{$T_2$}}  \rsdraw{.45}{.9}{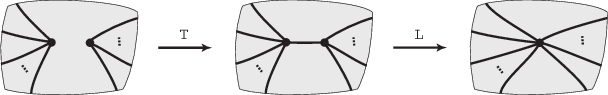} \\
[1.6em]  \psfrag{T}[Bc][Bc]{\scalebox{.9}{$T_3$}}  \rsdraw{.45}{.9}{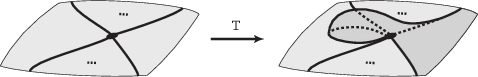} \\
[1.6em]  \psfrag{T}[Bc][Bc]{\scalebox{.9}{$T_4$}}  \rsdraw{.45}{.9}{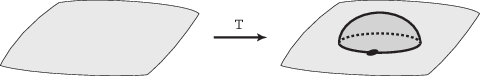}
\end{center}
\caption{Local moves on skeletons}
\label{fig-moves}
\end{figure}

   The orientation of   the skeletons produced by the moves $T_1-T_4$ on~$P$  is induced by the
  orientation of~$P$ except for the  small disk regions created by
  $T_3$,
  $T_4$ whose
  orientation
    is chosen   arbitrarily.

  The moves $T_1-T_4$ have  obvious inverses. The move   $T_1^{-1}$
 deletes a 2-valent edge~$e$ with distinct endpoints; this move is
   allowed only when both endpoints of~$e$ are endpoints of some other edges and the orientations on both sides of~$e$ are compatible.
We call the moves $T_1-T_4$ and their inverses \emph{primary moves}.
In  the sequel, we tacitly assume the right to use ambient
isotopies of   skeletons in~$M$. In other
 words, ambient isotopies are   treated as primary moves.

   \begin{lem}\label{lem-momoves}
Any two skeletons  of   $M$ can be related by  primary moves.
\end{lem}

We will prove this lemma in Section~\ref{proof-lem} using the
definitions
 and results of the next two subsections.

\subsection{Further moves}\label{proof-lem--}  We    introduce several additional moves on   skeletons of~$M$.
We begin with two   versions
  $T'_4$ and $T''_4$ of the bubble move, see Figure~\ref{fig-movesTT}.
   The move~$T'_4$  adds to a skeleton~$P$ an embedded   disk $D_+ $ as in   $T_4$
    with the only difference that the circle
   $\partial D_+ $ meets $P^{(1)} $ in a
  single point $x$   and bounds a disk in   a region~$r$ of~$P$.
  The move  $T'_4$  can be expanded as a product of primary
   moves as follows. First, if~$x$ is not a vertex of~$P$, then    turn~$x$ into a vertex by~$T_2^{-1}$. Then
        glue  a disk to~$r$ by~$T_4$. After that,~$T_1$ adds an edge in~$r$   connecting~$x$ to the
   vertex on this circle, and finally~$T_2$ collapses this edge to a
   point. This gives~$T'_4(P)$.

The move~$T''_4$   transforms~$P$ into~$P\cup D_+$,
   where~$D_+$ is a   disk in a small neighborhood  of
a point of an edge~$e$ of~$P$ such that   $D_+\cap P=\partial D_+$
and the set $
\partial D_+ \cap e$ consists of 2 points. Under this move, $e$ splits
  into 3 subedges $e_1, e_2, e_3$ numerated so that
  $\partial e_2=
  D_+ \cap e$.    We  construct a sequence of primary moves turning $P\cup
D_+$ into~$P$. First,~$T_2$ collapses $e_2$ to a point, then
   $T_3^{-1}$   pushes~$\partial D_+$ to one region. The resulting  skeleton  is
 $T'_4(P)$; by the above it can be transformed into $P$ by primary
   moves. The inverse sequence of   moves
transforms~$P$ into~$P\cup D_+$. Thus, $T''_4$    is a composition
of   primary moves.

\begin{figure}[h, t]
\begin{center}
\psfrag{T}[Bc][Bc]{\scalebox{.9}{$T_4'$}}  \rsdraw{.45}{.9}{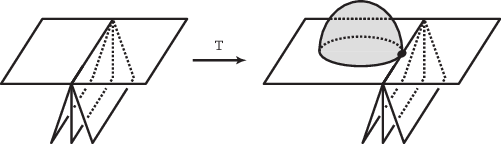} \\[1em]
\psfrag{T}[Bc][Bc]{\scalebox{.9}{$T_4''$}}  \rsdraw{.45}{.9}{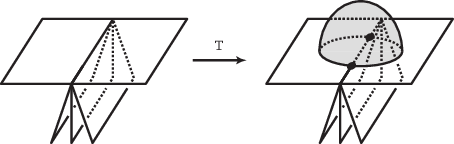}
\end{center}
\caption{The moves $T_4'$ and $T_4^{''}$}
\label{fig-movesTT}
\end{figure}


For any non-negative integers $m, n$ with $m+n\geq 1$, we define a
move   on   skeletons $T^{m,n}$, see Figure~\ref{fig-Tmn-move}.
 This move is sometimes
called a $m+1\to n+1$ move for the numbers of vertices in the
picture before and after the move. The move $T^{m,n}$ preserves
orientation in all \lq\lq big" regions; the orientation in the small
triangular regions created or destroyed by $T^{m,n}$  may be
arbitrary. In the case $n=0$, this move is allowed only  when the
orientations of the top and bottom regions on the left are
compatible. The move inverse to $T^{m,n}$ is $T^{n,m}$. Note  that
$T^{m,n}$ is a composition of primary moves. To see this,
  split $T^{m,n}$ as a product of a move~$T^{(m)}$ and the inverse of $  T^{(n)} $ as shown in
Figure~\ref{fig-Tmn-move}. The move $T^{(m)}$  is obtained by
applying~$m$ times the move~$T_2$ to collapse  the edges in the
$m-1$ small triangles on the left to the vertex~$x$   and then
applying $T_3^{-1}$ to the vertical branch $m-1$ times to remove the
loops at~$x$ resulting from the third edges of these triangles.

\begin{figure}[h,t]
\begin{center}
\psfrag{T}[Bc][Bc]{\scalebox{.9}{$T^{m,n}$}}
\psfrag{E}[cr][cr]{\scalebox{1.5}[3.1]{$\{$}}
\psfrag{L}[cl][cl]{\scalebox{1.5}[4.4]{$\}$}}
\psfrag{A}[cc][cc]{\scalebox{.9}{$T^{(m)}$}}
\psfrag{B}[cc][cc]{\scalebox{.9}{$(T^{(n)})^{-1}$}}
\psfrag{u}[cc][cc]{\scalebox{.9}{$m$}}
\psfrag{v}[cc][cc]{\scalebox{.9}{$n$}}
\rsdraw{.45}{.9}{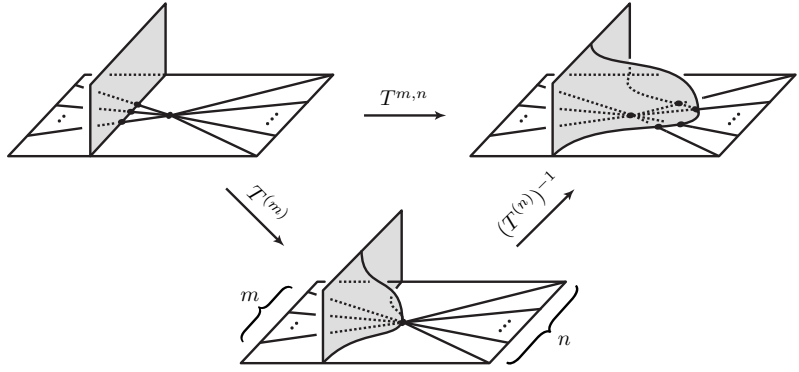}
\end{center}
\caption{The move $T^{m,n}$}
\label{fig-Tmn-move}
\end{figure}

 \subsection{Special skeletons}\label{proof-lem-}  We briefly recall the theory of
special skeletons due to Casler, Matveev, and Piergallini  (see, for
example, Chapter~1 of~\cite{Mat1}).    A \emph{special 2-polyhedron}
$Q$ is a compact 2-polyhedron  such that: all
   components of $\Int(Q)$ are  open  2-disks,   $Q \setminus \Int(Q)$ has no circle
   components, and  every point of $Q$ has a neighborhood
homeomorphic
  to an open subset of
   the set \begin{equation}\label{eq-papillon} \{(x_1, x_2, x_3) \in \RR^3 \, \vert\,  x_3 = 0, {\text {or}}\,  x_1 = 0 \,{\text {and}}\,
   x_3 > 0, \,{\text {or}}\, x_2 =
0 \, {\text {and}}\, x_3 < 0\}.\end{equation}   Then $Q \setminus
\Int(Q)$ is a   graph with only 4-valent vertices. The edges of this
graph yield a canonical stratification of~$Q  $; this  turns $Q$
into a stratified 2-polyhedron. All
   edges of $Q$   have   valence~3 and all vertices of~$Q$ are adjacent to 6 branches. Since all regions of~$Q$ are disks, $Q$
    is orientable.

  A \emph{special skeleton} or shorter an \emph{s-skeleton} of a closed 3-manifold~$M$ is an oriented special   2-polyhedron $Q\subset M$
  such that
   $M\setminus Q$ is a disjoint union of
open 3-balls. Any s-skeleton of~$M$ is a
  skeleton of~$M$ in the sense of Section~\ref{sect-skeletons}. For example, the  2-skeleton of the cellular
subdivision of~$M$ dual to a triangulation
  is a s-skeleton of~$M$ (the regions  are provided  with arbitrary orientation).

  \begin{lem}\label{lem-primmoves}
Any skeleton $P$ of $M$ can be transformed by the
   primary
   moves   into a special skeleton.
\end{lem}

\begin{proof}   The
   transformation proceeds in five steps.

   Step 1.    Adding if necessary
    new
  edges  by $T_1$,
   we can ensure that all regions of~$P$ are disks. This condition
   will be preserved through the rest of the construction.

 Step 2. Let $e$ be an edge of $P$  of valence $2$.   We apply $T''_4$ at a point of $e$.
  Let $e_1, e_2, e_3$ be the subedges of $e$ as in the definition of~$T''_4$.  Delete $e_2$ via $T_1^{-1}$
   and contract both~$e_1$ and~$e_3$ via~$T_2$  (the orientation  of
   the small disks created by $T''_4$ should be chosen   so that~$T_1^{-1}$ can be applied).  The resulting skeleton has one 2-valent edge
  less  and two new trivalent edges. Continuing by induction, we   obtain a
   skeleton, still denoted~$P$, whose all edges have valence $\geq 3$.

   Step 3. If  $P$ has an edge $e$ of valence $n\geq 4$, then
   apply   $T_2^{-1}$ to add  a new vertex~$x$ inside~$e$.
   This   splits $e$ into two  subedges $e_1$ and~$e_2$.
Next   apply~$T_3$
   pushing one of the branches of~$P$ at~$x$ across a small disk $D$
   lying in an adjacent branch and touching  $e$ at $x$. This move creates a trivalent edge
   $ \partial D$   and keeps all the  other edges. Then   apply   $T_2^{-1}$ to insert a new edge~$e_+$ between~$e_1$ and~$e_2$.
   The valence of~$e_+$ is~$n-1$. Finally,  apply~$T_2$ twice to contract~$e_1$ and~$e_2$.
   The resulting skeleton differs from the original one in that the
  edge~$e$ is replaced with a trivalent edge and an
   $(n-1)$-valent edge. Continuing by induction, we   obtain a
   skeleton, still denoted~$P$, whose all edges are trivalent.

   Step 4.  Let   $(\partial B_x,\Gamma_x)$ be the link of a vertex   $x\in P$.  If the graph~$\Gamma_x$ is disconnected, then we modify~$P$
near~$x$ as follows. Pick two edges $\alpha_1, \alpha_2$
of~$\Gamma_x$ lying on
     different components of $\Gamma_x$ and adjacent to the same
    component of $\partial B_x \setminus \Gamma_x$. Let~$b_i$ be the branch  of~$P$ at~$x$
   containing $\alpha_i$ and let $u_i, v_i$ be the endpoints of~$\alpha_i$   for
   $i=1,2$ (possibly, $u_i=v_i$).    The  branches $b_1$, $b_2$ are
    adjacent  to the same
    component~$X$ of~$M \setminus P$.
Apply~$T'_4$ adding to~$P$ a disk~$D_+ $
    so that   $x\in \partial D_+ \subset b_1$ and $\Int (D_+)\subset X$. Then   apply $T_3$ to push~$D_+$ through~$x$ into
    $b_2$. In the new position, the circle $\partial D_+$ is formed by a loop in~$b_1$ and a loop in~$b_2$, both based at~$x$.
   The   graph~$\Gamma_x$ is modified under this transformation as
   follows: one deletes the edges $\alpha_1, \alpha_2$   and
   adds   a quadrilateral (formed by 4   vertices and 4 edges) and
   also~4  additional edges connecting the  vertices of the quadrilateral to   $u_1, v_1, u_2,
   v_2$, respectively. The resulting   graph   has one
   component less than~$\Gamma_x$. Continuing by induction, we   obtain a
   skeleton, still denoted~$P$, such that the link graphs of all vertices of~$P$ are connected.

Step 5.   At the previous steps we transformed the original skeleton
into a skeleton~$P$ such that all regions are disks, all edges are
trivalent,  and  the link graphs of all vertices are   connected.
Since all edges of $P$ are
   trivalent, $P^{(1)}\cap \Int (P)=\emptyset$. By the definition of
   a stratified 2-polyhedron, $P\setminus P^{(1)} \subset \Int(P)$.
   Thus, $\Int(P)=P\setminus P^{(1)}$ is a disjoint union of open disks.

   We call a vertex~$x$  of~$P$ \emph{standard} if there is a homeomorphism of a
neighborhood of~$x$ in~$P$ onto
   the set~\eqref{eq-papillon} carrying~$x$ into the point  $(0,0,0)$. If all vertices of $P$ are standard,
   then all points of~$P$ have neighborhoods
homeomorphic
  to an open subset of~\eqref{eq-papillon}, and
   the graph $P^{(1)}=P\setminus \Int (P)$ has no circle components
   (every component of $P^{(1)}$ contains a vertex of $P$ and all vertices are incident to 4 half-edges).
In this case, $P$ is a s-skeleton.

 If a vertex $x\in P$ is non-standard, then
we \lq\lq blow  up"~$P$ in a neighborhood of~$x$ as follows. Let
$B=B_x\subset M$ be a small (closed) ball neighborhood of~$x$ such
that
  $ P\cap B  $ is the cone over~$  \Gamma_x=P\cap \partial
B  $.     The      connected trivalent  graph $\Gamma_x $ splits the
2-sphere $\partial B $ into 2-disks. Set
$$P'=(P \setminus   \Int (B)) \cup \partial B   . $$ The 2-polyhedron~$P'$
 has the same  edges and vertices as $P$ except
that all edges incident to $x$ are cut near~$x$,
  the vertex~$x$ is deleted, and the
edges and vertices of~$\Gamma_x$ are added. This turns~$P'$ into a
  stratified 2-polyhedron. The set $M\setminus P'$ is a
disjoint union of $M\setminus P$ and an open 3-ball. We provide the
regions of~$P'$ lying in~$P$ with the induced orientation and orient
the regions of~$P'$ lying in~$\partial B $ in an arbitrary way. This
turn~$P'$ into a skeleton of~$M$. Since  all edges of~$P$ are
trivalent, the   graph~$\Gamma_x$
    is trivalent, i.e., every vertex of~$\Gamma_x$ is incident to   three half-edges
of~$\Gamma_x$. Therefore  all new vertices of~$P'$ are standard.
Below we construct a sequence of primary moves~$P'\to P$. The
inverse sequence transforms~$P$ in~$P'$. Blowing up~$P$ at all
non-standard vertices, we transform~$P$ into a s-skeleton.

To construct  a sequence of primary moves   $P'\to P$, pick a
maximal tree in $\Gamma_x \subset (P')^{(1)}$ and collapse it to a
point by several $T_2$-moves. This   gives a new skeleton~$P''$
of~$M$. The  edges of~$P''$ lying in $\partial B$ are trivalent
loops forming a wedge of  $n\geq 1$  circles based at a point $y\in
\partial B$. At least one of these loops, $S$, bounds a
disk $D\subset
\partial B$ in the complement of the other loops. Let $b \subset
\partial B \setminus \Int (D)$ be the region of $P''$ adjacent to
$S$ from the exterior, and let $b'\subset  (M-B)\cup S $ be the
third region of~$P''$ adjacent to~$S$. If $n\geq 2$, then
applying~$T_3^{-1}$ to~$b$ at the vertex~$y$  of~$P''$, we   change
the way in which~$b$ is glued to the rest of~$P''$ so that~$\partial
b$ does not pass along~$S$ anymore. Under this move, the region~$D$
of~$P''$ unites with~$b'$. After the move, the regions of $P''\cap
\partial B$ distinct from $D$
  form a 2-sphere meeting the rest of the skeleton along a wedge
of $n-1$ loops. Continuing by induction we reduce ourselves to the
case $n=1$. In this case,   apply $(T_4')^{-1}$ once. This gives a
skeleton   isotopic to~$P$. \end{proof}

\subsection{Proof of Lemma~\ref{lem-momoves}}\label{proof-lem}
 In view of Lemma~\ref{lem-primmoves}, we need only to prove  that any
  two s-skeletons   of~$M$ can be related by primary
  moves. By \emph{MP-moves}   on s-skeletons
(for Matveev and Piergallini), we mean the moves $T^{2,1}$,
$T^{1,2}=(T^{2,1})^{-1}$, and $ ( T''_4)^{\pm 1} $,   see
Figure~\ref{fig-32moves} for~$T^{2,1}$. All MP-moves transform
s-skeletons into s-skeletons and  are compositions of primary moves.
Therefore it is enough to show that any two s-skeletons of~$M$ are
related by MP-moves.
 Applying if necessary   $T''_4$, we can ensure that    given   s-skeletons have $\geq 2$ vertices, and that their complements in~$M$ consist
   of the same number of open 3-balls. By Theorem~1.2.5
  of~\cite{Mat1}, the  special   2-polyhedra underlying these s-skeletons     can be related   by a
finite sequence of     moves $T^{1,2}$, $T^{2,1}$.    This implies
our claim   up to the choice of orientation of the regions. The
latter indeterminacy can be eliminated because for any region~$r$ of
a s-skeleton~$P$, there is a sequence of
  MP-moves  transforming~$P$ into~$P_-$, where~$P_-$ is~$P$ with
opposite orientation in~$r$  and the same orientation in all other
regions. Indeed,  the branches of $P$ adjacent to the sides of~$r$
form a cylinder neighborhood of~$\partial r$ in $P- \Int (r)$.
Pushing~$r$ in a normal direction so that~$\partial r$ sweeps  half
of this cylinder, we obtain    a solid cylinder $
 r\times [0,1]\subset M$ meeting~$P$ at $r\cup (\partial r\times [0,1])$, where $r= r\times \{0\} $. Set $r'= r\times \{1\} $. Then
 $P'=P\cup r'$ is a s-skeleton of~$M$, where the regions of~$P'$ contained in~$P$ receive the induced orientation and
 the orientation of the region  $ r'$ is opposite to that of~$r$. We can transform~$P'$ into~$P$ by
 MP-moves  pushing $ \partial r'$ inside $r$ and eventually
 eliminating~$r'$. This transformation involves one move
$T^{1,2}$, several moves  $T^{2,1}$, and one move~$(T''_4)^{-1}$ at
the end.
 A similar elimination of~$ r$ from~$P'$ gives~$P_-$. Hence~$P$ and~$P_-$ are related by
MP-moves.

\begin{figure}[h,t]
\begin{center}
\psfrag{T}[Bc][Bc]{\scalebox{.9}{$T^{2,1}$}}  \rsdraw{.45}{.9}{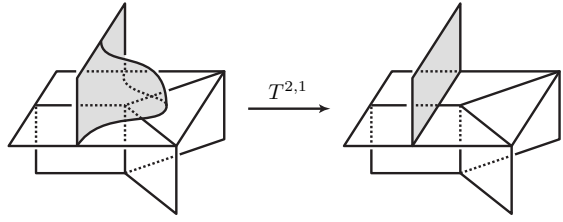}
\end{center}
\caption{The move $T^{2,1}$}
\label{fig-32moves}
\end{figure}

\subsection{Proof of Theorems~\ref{thm-state-3man-smipl} and~\ref{thm-state-3man}} The state sum $|M|_\cc$ does not depend
on the choice of the representative set~$I$   by the naturality of
$\inv_{\cc}$ and of the contraction maps.  By
Lemma~\ref{lem-momoves}, to prove the rest of the theorems, we need
only to prove the invariance of the right-hand side of
Formula~\eqref{eq-simplstatesum+} under the
 moves $T_1-T_4$ on $P$.

 The move $P \stackrel {T_1}{\to} P'$ involves two distinct vertices  of~$P$ and modifies their link graphs by adding a new
 vertex $u$ (respectively, $v$)
  inside an edge.   The colorings of~$P'$ assigning different colors to the
  regions of~$P'$ lying on different sides of the new edge contribute zero
  to the state sum by Lemma~\ref{lem-calc-diag}(a) (there are no such colorings if these regions coincide).
  The colorings of~$P'$ assigning the same color $i\in I$ to
  these
  regions contribute the same as the colorings of~$P$ assigning~$i$ to the region of~$P$  containing the new
 edge:
 \begin{equation}\label{eq-computation}
\psfrag{i}[Bc][Bc]{\scalebox{.9}{$i$}}\psfrag{u}[Bc][Bc]{\scalebox{.9}{$u$}}\psfrag{v}[Bc][Bc]{\scalebox{.9}{$v$}} *_{u,v}\;
\inv_\cc \!\begin{pmatrix}
\;  \rsdraw{.45}{.9}{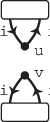}\,
\;\;\end{pmatrix} = \;\;\psfrag{i}[Bc][Bc]{\scalebox{.9}{$i$}}  \inv_\cc \!\begin{pmatrix}
\; \rsdraw{.45}{.9} {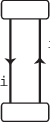} \; \; \end{pmatrix} = \;(\dim(i))^{-1}\;  \inv_\cc \!\begin{pmatrix}
\; \rsdraw{.45}{.9}{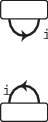}\; \end{pmatrix} \;.
\end{equation}
Here the first equality follows from Lemma~\ref{lem-calc-diag}(d)
and the second equality   can be deduced from this lemma or
proved directly using   that $\Hom_{\cc} (i\otimes i^*,\un)=\kk\,
\lev_i$ and $\Hom_{\cc} (\un, i\otimes i^* )=\kk \,\lcoev_i$.
 The factor $\dim (i)$ on the right-hand side of~\eqref{eq-computation} is compensated by the
change in the Euler
 characteristics of the regions.

 The invariance under $T_2$ follows from    Lemma~\ref{lem-calc-diag}(d). The invariance under $T_3$ follows from
  Lemma~\ref{lem-calc-diag}(c) with two vertical strands on the left-hand side. Here $i\in I$ is the color of the disk
  region created by the move.  The invariance under~$T_4$ follows from     Lemma~\ref{bubbleidentity}. Here $i\in I$ is the color of the big region
  where the move proceeds  and $k,l\in I$ are the colors of the small disks created by
 the move. The   factor  $\dim (i) \dim (\cc) $ is compensated by the change in
the number of components of $M\setminus P$  and in the Euler
 characteristic. We  use the equality $\ast_e
 (\inv_{\cc} (\Gamma_v)) = N^{\un}_{i\otimes k\otimes l}$,  where
 $e$ and $v$ are respectively the   edge and the vertex forming the boundary of the small disks.

\subsection{Remarks}\label{rem-onemoreinv} 1. The invariant $ \vert
 M \vert_\cc$ of a closed connected oriented 3-manifold~$M$ can be generalized to an arbitrary spherical fusion category~$\cc$ (without the assumption that $\dim(\cc)$ is invertible in~$\kk$).  Indeed, the right-hand side of
Formula~\eqref{eq-simplstatesum+}  is the product of  $(\dim
(\cc))^{-\vert P\vert}$ and a certain sum which we denote
$\Sigma_\cc(P)$. The definition of $\Sigma_\cc(P)\in \kk$
applies to an arbitrary spherical fusion category~$\cc$.   Let us call a special skeleton
 $P\subset M$ an \emph{s-spine} if $M\setminus P$ is an open ball and $P$ has at least two vertices.   By
 \cite{Mat1}, $M$  has a s-spine~$P$ and any two s-spines of~$M$ can be related by  the   moves $T^{1,2}$, $T^{2,1}$ in the class of s-spines. The arguments above imply that   $\Sigma_\cc(P)$ is preserved under these moves.
  Therefore   $\vert \vert
 M\vert\vert_\cc= \Sigma_\cc(P)$   is a
 topological invariant of~$M$. If $\dim(\cc)$ is invertible, then $   \vert \vert
 M\vert\vert_\cc= \dim (\cc)\,  \vert
 M \vert_\cc$.

 2. A stratified 2-polyhedron $P'$ is a \emph{subdivision} of a
 stratified 2-polyhedron~$P$ if they have the same underlying 2-polyhedron, all vertices of~$P$ are among the
 vertices of~$P'$, and all edges of~$P$ are unions of edges
 of~$P'$. Then $P'\setminus (P')^{(1)}\subset P\setminus
 P^{(1)}$ and therefore an orientation of~$P$ induces an orientation
 of~$P'$. If $P$ is a skeleton of a closed 3-manifold $M$, then
 any subdivision~$P'$  of $P$ is also a skeleton of $M$. As an exercise,
 the reader may relate $P$ and $P'$ by the moves $T_1^{\pm 1}$,
 $T_2^{\pm 1}$.

 \section{Skeletons in the relative case}\label{sec-scerel}

We study   skeletons of   3-manifolds with boundary and their
transformations.

\subsection{Skeletons of pairs}\label{sec-skes} Let $M$ be a compact 3-manifold (with boundary).    Let~$G$ be an
oriented graph in~$\partial M$ such that all vertices of~$G$ have
valence $\geq 2$. (A graph   is \emph{oriented}, if all its edges
are oriented.) A \emph{skeleton} of the pair $(M,G)$ is an oriented
  stratified 2-polyhedron $P\subset M$ such that
\begin{enumerate}
  \renewcommand{\labelenumi}{{\rm (\roman{enumi})}}
    \item  $P\cap
\partial M= \partial P= G$;

\item every vertex $v$ of $G$  is an endpoint of a unique edge
$d_v$ of $P$ not contained in~$\partial M$; moreover,
$d_v\cap\partial M=\{v\}$  and  $d_v$ is not a loop;

\item every edge $a$ of $G$  is an edge of~$P$  of valence 1; the
only region~$D_a$ of~$P$ adjacent to~$a$ is   a closed 2-disk,
$D_a\cap \partial M= a $,  and the orientation of~$D_a$ is
compatible with that of~$a$ (see Section~\ref{sec-Io3m} for
compatibility of orientations);

\item $M\setminus P$ is a disjoint union of a finite collection of
open 3-balls and a 3-manifold homeomorphic to   $(\partial M
\setminus G) \times [0,1)$ through a homeomorphism extending the
identity map
\begin{equation}\label{ided}\partial (M\setminus P)=\partial M \setminus G=(\partial M
\setminus G) \times \{0\}.\end{equation}
\end{enumerate}

Conditions (i)--(iii) imply that in a neighborhood of~$\partial M$,
  a skeleton of $(M,G)$ is a copy of
  $G\times [0,1]$.
The primary moves
 $T_1^{\pm 1} -  T_4^{\pm 1}$ on   skeletons of closed 3-manifolds extend to skeletons~$P$ of   $(M,G)$ in the obvious way.
 These moves  (and all other moves on skeletons considered below)     keep    $\partial P=G $ and    preserve the skeletons in a
neighborhood of their boundary~$G$. In particular, the move $T_1$
 adds
 an
 edge with  both endpoints in $\Int(M)$, the move~$T_2$ collapses an
 edge contained in $\Int(M)$, etc.   Ambient isotopies of   skeletons in~$M$ keeping the boundary pointwise are  also viewed as primary moves.

\begin{lem}\label{lem-skeletons}
Every pair (a compact  orientable 3-manifold~$M$, an oriented
graph~$G$ in~$\partial M$ such that all vertices of~$G$ have valence
$\geq 2$) has a skeleton. Any two skeletons of $(M,G)$ can be
related by primary moves in~$M$.
\end{lem}

We   prove this lemma in Section~\ref{sec-skes++} using the
 results of Section~\ref{sec-skes+} and  the theory
 of special skeletons which we briefly outline.  Suppose that the graph $G\subset \partial M$ is trivalent, i.e., all its vertices have valency $3$
  (possibly, $G=\emptyset$).
  A skeleton $P$ of   $(M,
G)$ is   \emph{special} or an {\emph {s-skeleton}} if all regions of
$P$ are 2-disks, all edges of~$P$ are trivalent, all vertices of~$P$
are incident to 4 half-edges, and every point of~$P\setminus G$ has
a neighborhood in~$P$ homeomorphic to an open subset of the
set~\eqref{eq-papillon}.   As in the proof of
Lemma~\ref{lem-primmoves}, we can transform  any skeleton~$P$ of
$(M, G)$
  into a s-skeleton  by primary moves applied away
from~$\partial P=G$. By \cite{TV}, Corollary 6.4.C,    any two
s-skeletons of $(M, G)$ can be related by    MP-moves and   lune
moves applied away from $G$. The MP-moves on s-skeletons of $(M, G)$
are defined as in Section~\ref{proof-lem} and the \emph{lune moves}
$\mathcal L^{\pm 1}$ are shown in Figure~\ref{fig-movesee}. These
moves on s-skeletons are allowed here only when they produce
s-skeletons (this is always the case for the MP-moves and ${\mathcal
L}$)
 and preserve the orientation of the regions. In particular, the move~${\mathcal L}^{-1}$ is allowed only
when the orientations of two regions united by this move are
compatible and the union of these regions is a 2-disk. The
orientation of the small disk regions destroyed or created by the
MP-moves  and   ${\mathcal L}^{\pm 1}$
   may be arbitrary. Note   that Corollary 6.4.C of \cite{TV}
   does not handle orientations of the skeletons but remains true in the oriented
   setting:
   the
   indeterminacy in the choice of the orientation of the regions can be resolved   as in
   Section~\ref{proof-lem}.

\begin{figure}[h,t]
\begin{center}
 \rsdraw{.45}{.9}{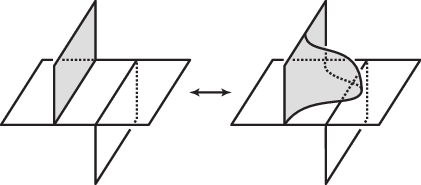}
\end{center}
\caption{The lune moves ${\mathcal L}^{\pm 1}$}
\label{fig-movesee}
\end{figure}

\subsection{Skeletons and frames}\label{sec-skes+} Fix a compact
3-manifold~$M$. We assume   that~$M$ is oriented and  endow
$\partial M$ with the induced orientation.

A \emph{skeleton} of~$M$ is a skeleton of the pair $(M, \emptyset)$
in the sense of Section~\ref{sec-skes}. (For
  closed~$M$,
this notion is the same as in Section~\ref{sect-skeletons}.) A
skeleton $Q $ of~$M$ is a \emph{frame}
  if there is an
embedding $i\co  \partial M\times [0,1]\to M$ extending the
identification  $\partial M \times \{0\}=\partial M$ such that
  $$Q\cap i( \partial M\times [0,1])=i (  \partial M\times
  \{1\}),$$ and the orientation of  the regions of~$Q$ contained in $i (  \partial M\times
  \{1\})$    is induced by that of~$\partial M$ via $i$.
Such an embedding~$i$ is called a \emph{$Q$-collar}. By the
definition of a skeleton, all components of $M\setminus (Q\cup i(
\partial M\times [0,1]))$ are open
  3-balls.

 Not all skeletons of~$M$ are frames. We illustrate this claim with an example.
  Consider  an unknotted torus $T=S^1\times S^1$
  in a closed 3-ball~$B$  and add
to~$T$ two  disks in~$B$ lying on different sides of~$T$ and bounded by the loops $S^1\times \{s\}$ and $\{s\} \times S^1$, where $s\in S^1$.
   For any orientation  of the regions, the   resulting oriented 2-polyhedron is a   skeleton   of~$B$ but not a
   frame.

  For a skeleton $Q\subset \Int(M)$ of~$M$, denote
  by~$\widehat Q$ the union   of~$Q$ with all open ball components of $M\setminus
  Q$. The skeleton~$Q$ is a frame if and only if  $\widehat Q$
   is
   a  3-manifold    with boundary and the orientation of all regions of~$Q$ contained
   in $\partial \widehat Q$ is induced by the orientation of~$M$
   restricted to~$\widehat Q$.
 The   surface
     $Q^+=\partial \widehat Q \subset Q$ is homeomorphic to~$\partial M$  and equal to $i (
\partial M\times
  \{1\})$ for any $Q$-collar~$i$.
A local analysis shows that any edge  of~$Q$ meeting~$Q^+$ either
  lies in~$Q^+$ or meets~$Q^+$ in one or two endpoints.   Every
  vertex of $Q$ lying in $Q^+$ is incident to at least one edge of~$Q$ lying in $Q^+$. Therefore the set  $Q^+\cap Q^{(1)}$
  is a graph whose edges and vertices  are the  edges and vertices  of~$Q$ lying in   $Q^+$. (Recall that $Q^{(1)}$ is the
union of edges of~$Q$.)

 By \emph{primary moves} on a frame~$Q$ of~$M$, we   mean the primary moves on skeletons
 $T_1^{\pm 1} -  T_4^{\pm 1}$    applied inside~$M$ at  vertices of~$Q$ not lying in~$Q^+$ or at  edges of~$Q$ disjoint from
 $Q^+$. Such vertices and edges are surrounded by the ball components of $M\setminus Q$. Therefore these
 moves
 do not modify   $\widehat Q$ and   produce frames of~$M$.

 By \emph{generalized MP-moves} on   frames, briefly \emph{GMP-moves},  we mean the moves
$\{T^{m,n}\}_{m,n}$,  $ (T''_4 )^{\pm 1} $  introduced in
Section~\ref{proof-lem--} and the   \emph{lune  moves} $\mathcal
L^{\pm 1}$ shown in Figure~\ref{fig-movesee}. These moves are
allowed in this context only when they produce frames
 and preserve the orientation of the regions. In particular, the move~${\mathcal L}^{-1}$ is allowed only
when the orientations of two regions united by this move are
compatible. The orientation of the small disk regions destroyed or
created by the moves
   may be arbitrary.

As we know, all GMP-moves on frames except possibly the lune moves
expand as compositions of primary moves on skeletons though the
intermediate skeletons may not be frames. The
 lune moves  also expand  as   compositions of primary
moves on skeletons; we leave it to  the reader as an exercise.

\begin{lem}\label{lem-skeletons1}
 The manifold $M$ has a frame. Any two frames of~$M$
can be related by a finite sequence of primary moves and GMP-moves
in the class of frames.
\end{lem}

\begin{proof}   Fix a  triangulation $t$ of
$M$ and denote~$\partial t$ the induced triangulation of~$\partial
M$. The triangulation~$t$ gives rise to a dual
  cellular decomposition~$t^*$ of~$M$. It  is formed by the cells of the cellular decomposition~$(\partial t)^*$ of $\partial M$ dual to $\partial t$
  and  the
  cells   dual to the  simplices of~$t$ in~$M$. The 2-skeleton $ Q_t=(t^*)^{(2)} \subset M$ of $t^*$ is
a 2-polyhedron
  containing $\partial M$, and $M\setminus Q_t$ is the disjoint union of
  the open 3-cells of~$t^*$. The edges of $t^*$ form a stratification of $Q_t$.   We endow   all   regions of~$Q_t$ lying in $\partial M$
with the orientation induced by that of~$\partial M$. All other
regions of $Q_t$ are oriented in an arbitrary way. Pushing $Q_t$
inside $M$ we obtain a frame of $M$.

The construction of $ Q_t$ can be generalized as follows. Let~$F{}$
be the (trivalent)   graph in~$\partial M$ formed by the vertices
and edges of~$(\partial t)^*$. Any skeleton $P$ of the pair $(M,
F{})$ determines a 2-polyhedron $Q=Q(P)=P\cup
\partial M$. The edges of~$P\subset Q$ form a stratification of $Q$. We endow all  regions of~$Q$ lying in
$\partial M$ with the orientation induced by that of~$\partial M$.
The other regions of $Q$ inherit their orientation from $P$. Pushing
$Q$ inside $M$ we obtain a frame of $M$. The frames of~$M$  obtained
in this way from skeletons   of   $(M, F{})$ are called
\emph{$t$-frames}.

The discussion at the end of Section~\ref{sec-skes} shows that any
two skeletons of   $(M, F{})$ can be related by primary moves in the
class of skeletons of $(M, F{})$.
   Extending these moves to the associated $t$-frames in the
obvious way, we obtain that any two $t$-frames of~$M$ can be related
by primary moves   in the class of $t$-frames.

We show  now that any frame~$Q$ of~$M$ can be transformed into a
$t$-frame by     GMP-moves.    Pick a $Q$-collar $i\co
\partial M\times [0,1]\hookrightarrow M$. Let $V$ and $E$ be  the sets of
vertices and     edges of the graph $F{}$ respectively.
Deforming~$i$, we can ensure that the graph $F_i{}=i(F{}\times
\{1\})\subset Q^+$ is generic in the sense that its vertices
$\{i(v\times \{1\})\}_{v\in V}$ lie in $Q^+\setminus Q^{(1)}$ and
its edges $\{i(a\times \{1\})\}_{a\in E}$ are transversal to the
edges of  the graph $   Q^+\cap Q^{(1)}$.  Set
$$R=Q\cup i(F{}\times  [\frac{1}{2}, 1]) \cup i(\partial M\times \{\frac{1}{2}\}) .$$
We stratify the 2-polyhedron $R\subset \Int(M)$ as follows. The
points of $F_i{}\cap Q^{(1)}$ split the edges of the graphs $F_i{}$
and  $Q^+\cap Q^{(1)}$ into smaller subedges. For  edges of~$R$ we
take
 all these subedges   together with the edges of~$Q$ not lying
 in~$Q^+$
 and the arcs $\{i(v\times [\frac{1}{2}, 1])\}_{v\in V}$  and  $\{i(a\times \{\frac{1}{2}\})\}_{a\in E}$. Then the  vertices  of~$R$ are the vertices
of~$Q$ and  the   points of the sets $F_i{}\cap Q^{(1)}$  and
$\{i(v\times \{\frac{1}{2}, 1\}) \}_{v\in V}$.
 Clearly, $$R^{(1)}= Q^{(1)} \cup i(V\times [\frac{1}{2},1]) \cup i(F{}\times \{\frac{1}{2},1\})
 .$$
  We orient $R$ as follows:
the regions contained in~$Q$ inherit their orientation from~$Q$, the
orientation of the regions   contained in $ i(\partial M\times
\{\frac{1}{2}\})$ is induced by that of~$\partial M$, the regions
$\{i(a\times [\frac{1}{2},1])\}_{a\in E}  $  are oriented in an arbitrary
way.  It is clear that~$R$ is a $t$-frame of~$M$. Each region~$r$
of~$R$ contained in    $R^+=i(\partial M\times \{\frac{1}{2}\})$  is a
2-disk whose boundary can be pushed towards $Q^+$ and further
contracted into a small circle  by    GMP-moves in the class of
frames. At the end, this \lq\lq moving" disk can be eliminated via
the inverse bubble move. Note that the orientation of the regions
of~$R$ yield no obstructions to these moves because all regions
of~$R$ contained in~$Q^+$ are oriented coherently.     In this way,
eliminating consecutively all   regions of~$R$ contained in~$R^+$,
we can transform $R$ into $Q$ by   GMP-moves. The inverse sequence
of moves transforms~$Q$  into the $t$-frame~$R$. Now, the properties
of $t$-frames established above imply the claim of the lemma.
\end{proof}

\subsection{Proof of Lemma~\ref{lem-skeletons}}\label{sec-skes++}  It is enough to consider  the case where~$M$ is connected. The case $\partial M=\emptyset$
having been treated above, we    assume that $\partial M\neq
\emptyset$.
 Fix an orientation of~$ M$.
Denote by $V$ and $E$   the sets of vertices and edges of $G $
respectively.

In analogy with the construction of the polyhedron~$R$ at the end of
the proof of Lemma~\ref{lem-skeletons1}, we construct a skeleton of
$(M,G)$ from any frame $Q\subset M$.  Pick a $Q$-collar $i\co
\partial M\times [0,1]\hookrightarrow M$ such that the graph $G_i =i(G\times
\{1\})\subset Q^+$ is generic, i.e.,  its vertices $\{i(v\times
\{1\})\}_{v\in V}$ lie in $Q^+\setminus Q^{(1)}$ and its edges
$\{i(a\times \{1\})\}_{a\in E}$ are transversal to the edges of the
graph $Q^+\cap Q^{(1)}$. Set
 $$P=P(Q,i)=Q\cup i(G \times  [0, 1])   .$$ We  stratify the 2-polyhedron~$P$ as follows. The   points of $G_i\cap
Q^{(1)}$ split the edges of $G_i$ and the edges of $Q^+\cap Q^{(1)}$
into smaller subedges. In the role of  edges of~$P$ we take
 all these subedges   together with the edges of~$Q$ not lying
 in~$Q^+$
 and the arcs $\{i(v\times [0, 1])\}_{v\in V}$  and  $\{i(a\times \{0\})\}_{a\in E}$. The   vertices  of~$P$ are the vertices
of~$Q$    and  the   points of the sets $G_i\cap Q^{(1)}$ and
$\{i(v\times \{0, 1\}) \}_{v\in V}$.
 Clearly, $$P^{(1)}= Q^{(1)} \cup i(V\times [0,1]) \cup G \cup G_i
 .$$
  We orient $P$ as follows:
the regions contained in~$Q$ inherit their orientation from~$Q$; the
orientation of the regions  $\{i(a\times [0,1])\}_{a\in E}  $ is
induced by that of~$G$. Since all vertices of~$G$ have valency $\geq
2$, we have $\partial P=G$.  Clearly,~$P$ is a skeleton of~$(M,G)$.

A different choice of the $Q$-collar $i$ may lead to a different
skeleton. However, any two $Q$-collars  are isotopic in the class of
$Q$-collars. Making an isotopy $\{i_s\}_{s\in [0,1]}$ between them
transversal to $Q^+\cap Q^{(1)}$, we can ensure that for all but a
finite set of values of~$s$, the graph $G_{i_s}= i_s(G\times
\{1\})\subset Q^+$ is generic, and when~$s$ increases through each
of the exceptional values,  the   skeleton $P(Q,i_s)$  is modified
via $T^{m,n}$ or ${\mathcal L}^{\pm 1}$. The orientation of the
regions do not create obstructions to these moves because all
regions of $P(Q,i_s)$ contained in~$Q^+$ are oriented coherently. We
conclude that up to GMP-moves, the skeleton~$ P(Q,i) $ does not
depend on the choice of~$i$.

 Consider two frames $Q$ and $Q'$ of~$M$.  If~$  Q'$ is obtained
from~$Q$   by a  primary move, then $\widehat {Q'}=\widehat Q$ and
any $Q$-collar $i$ is also a $Q'$-collar. Then the skeleton
$P(Q',i)$ is obtained from $P(Q,i)$ by the same primary move.
Suppose that~$ Q'$ is obtained from~$Q$ by a GMP-move~$T$. This move
proceeds inside a small neighborhood $U\subset Q$ of an embedded arc
in~$Q$. Deforming a $Q$-collar~$i$ in the class of $Q$-collars, we
can ensure that $G_i   \cap U =\emptyset$. Then an appropriate
deformation of $i$ yields a $Q'$-collar~$i'$ such that $G_i
 =G_{i'}$. Clearly,  the skeleton $P(Q',i')$ is obtained from
$P(Q,i)$ by the same move~$T$. Since~$T$ expands  as a product  of
primary moves, the skeletons $P(Q,i)$ and $P(Q',i')$ are related by
primary moves. By Lemma~\ref{lem-skeletons1}, the skeletons of
$(M,G)$ associated with any two frames of~$M$ are related by primary
moves.

To  complete the proof, it is enough to show that any skeleton~$P$
of $(M,G)$ can be transformed by   primary moves in a skeleton of
$(M,G)$ associated with a frame. Recall the edges $\{d_v\}_{v\in V}$
and the regions $\{D_a\}_{a\in E}$ of~$P$ from the definition of a
skeleton.  The transformation of~$P$ proceeds in three steps. At
Step~1 we blow up all vertices of~$P$ lying in $\Int (M)$ (including
the standard vertices)  as  at Step~5 of the proof of
Lemma~\ref{lem-primmoves}. It is explained there that this blowing
up can be achieved by primary moves. We need to tune up a little
this procedure: the regions of the
  resulting skeleton,~$P_1$, lying on the small 2-spheres created by blowing
   up are endowed with orientation induced by that of~$M$ restricted to the small
  3-balls bounded by these spheres (the orientation of all other regions of~$P_1$ is induced by that of~$P$).
The skeleton~$P_1$ has edges of two types: the
  \lq\lq short" edges lying on the small spheres (in the
proof of Lemma~\ref{lem-primmoves}, these are the edges
of~$\Gamma_x$ lying in  $\partial B_x$) and the \lq\lq long" edges
obtained by shortening the original edges of~$P$ at their endpoints.
In particular, the edges $\{d_v\}_{v\in V}$ of~$P$ give rise    to
long edges $\{d'_v\}_{v\in V}$ of~$P_1$. The long edges of $P_1$ are
pairwise disjoint. A  long edge~$e$ distinct from $\{d'_v\}_{v\in
V}$ has an open ball neighborhood $U_e$ in~$M$
  that can be identified with~$\RR^3$ so that
$$P_1\cap U_e= (\RR^2 \times \{0, 1 \}) \cup (Y_n\times [0,1]),$$
where $n\geq 2$ is the valence of $e$,
 $Y_n\subset \RR^2$ is a union of~$n $ rays   with
common origin~$O$, and $e=O\times [0,1]$. Let $D\subset \RR^2$ be a
2-disk centered at~$O$ and meeting $Y_n$ along~$n$ radii.
  We modify~$P_1$ by adding the cylinder $\partial D\times [0,1]\subset U_e$ surrounding~$e$.
   The orientation of the   regions lying on this cylinder is induced by that of~$M$ restricted to $  D\times [0,1]\subset U_e$
    while the orientation of all other regions is induced from that of~$P_1$.
  This
 modification can be   achieved by MP-moves and hence by
 primary  moves. We apply these modifications in disjoint
 neighborhoods of all long
edges of~$P_1$ distinct from $\{d'_v\}_{v\in V}$. The
  resulting skeleton,~$P_2$, has regions of two types: the
  \lq\lq small" or \lq\lq narrow" regions created by the previous  transformations at the vertices and   edges   and the
  \lq\lq wide" regions   obtained by  cutting
the original regions of~$P_2$ near their boundary. In particular,
the regions $\{D_a\}_{a\in E}$ of $P$ give rise   to slightly
smaller  regions $\{D'_a\}_{a\in E}$ of~$P_2$.   The wide regions
of~$P_2$ (as well as their closures) are pairwise disjoint. All
regions of~$P_2$ except $\{D'_a\}_{a\in E}$ lie in $\Int (M)$. For a
wide region~$r$ of~$P_2$ lying in $\Int (M)$, we can use MP-moves to
add to~$P_2$ a new region~$r'$ parallel to~$r$, cf.\ the end of the
proof of Lemma~\ref{proof-lem}. This can be done so that the
orientations of~$r$ and~$ r'$ are induced by that of~$M$ restricted
to the solid cylinder $r\times [0,1]$   between~$r$ and~$r'$. These
modifications are applied to all wide regions of~$P_2$ except
$\{D'_a\}_{a\in E}$ inside their   disjoint neighborhoods in~$M$.
This gives a skeleton,~$P_3$, of $(M,G)$. We claim that~$P_3$ is
associated with a frame of~$M$. To see this, consider the
2-polyhedron $Q\subset P_3\cap \Int(M)$ obtained from~$ P_3$ by
removing  the graph $G=\partial P=\partial P_3$, the interiors of
the edges $\{d'_v\}_{v\in V}$, and the interiors of the regions
$\{D'_a\}_{a\in E}$.  The 2-polyhedron~$Q$ can be stratified so that
the graphs~$Q^{(1)} $ and $Q \cap \cup_a \overline {D'_a}$ have only
double transversal crossings and the union of these graphs is equal
to  $Q \cap P_3^{(1)}$. All regions of~$Q$ are regions (or unions of
regions) of~$P_3$ and inherit orientation from~$P_3$. This turns $Q$
into an oriented stratified 2-polyhedron. Since~$P_3$ is a skeleton
of $(M,G)$, the polyhedron~$Q$ is a skeleton of~$M$. We claim
that~$Q$ is a frame. Indeed, the set $\widehat Q\subset M$ is the
union of a regular neighborhood of~$P$ in~$M$ with all 3-ball
components of $M\setminus P$. This union  is a 3-manifold with
boundary, and the orientation of all regions of~$Q$ contained
   in $\partial \widehat Q$ is induced by the orientation of~$M$
   restricted to~$\widehat Q$.  Now, it
is easy to see from the definitions that there is a $Q$-collar $i\co
\partial M\times [0,1]\hookrightarrow M$ such that
$i(\{v\}\times [0,1])=d'_v$ for all $v\in V$ and  $i(a\times
[0,1])=D'_a$ for all $a\in E$. Then $P=P(Q,i)$ is the skeleton
associated with~$Q$.

\section{The state-sum TQFT}\label{sec-TQFT}

We construct in this section the state sum TQFT associated with any spherical fusion category with invertible dimension.

\subsection{Preliminaries on  TQFTs}\label{Preliminaries on  TQFTs} For convenience of the reader, we outline   a definition of a 3-dimensional  Topological Quantum
Field Theory  (TQFT)  referring for details to  \cite{Tu1}. We first define a  category of 3-dimensional cobordisms $\mathrm{Cob}_3$.
   Objects of $\mathrm{Cob}_3$ are closed oriented surfaces. A morphism   $\Sigma_0 \to \Sigma_1$ in $\mathrm{Cob}_3$ is represented by  a pair $(M,{h})$,  where $M$ is
 a compact oriented 3-manifold and ${h}$ is      an orientation-preserving homeomorphism $  (-\Sigma_0) \sqcup\Sigma_1 \simeq  \partial M$.
Two such pairs $(M, {h}\co  (-\Sigma_0) \sqcup\Sigma_1 \to \partial M )$ and
$(M', {h}' \co  (-\Sigma_0) \sqcup\Sigma_1 \to \partial M')$ represent the same morphism $\Sigma_0\to \Sigma_1$ if there is an orientation-preserving homeomorphism  $F\co  M \to M'$ such that ${h}'=F{h}$.  The identity morphism of a  surface $\Sigma$ is represented by the cylinder $ \Sigma \times [0,1]$ with the product orientation and the tautological identification of the boundary with $(-\Sigma) \sqcup\Sigma$. Composition of morphisms in $\mathrm{Cob}_3$ is defined through gluing of cobordisms:    the composition of morphisms $(M_0, {h}_0) \co \Sigma_0 \to \Sigma_1$ and $(M_1, {h}_1) \co \Sigma_1 \to \Sigma_2$  is    represented by the pair $(M,{h})$, where $M$ is the result of gluing  $M_0$ to $M_1$ along ${h}_1 {h}_0^{-1} \co  {h}_0(\Sigma_1)   \to {h}_1({\Sigma_1})$ and ${h}={h}_0 \vert_{\Sigma_0} \sqcup {h}_1 \vert_{\Sigma_2} \co (-\Sigma_0) \sqcup\Sigma_2 \simeq  \partial M$. The category $\mathrm{Cob}_3$ is   a symmetric monoidal category with tensor    product   given by   disjoint union.
The unit object of $\mathrm{Cob}_3$ is the empty surface $\emptyset$ (which by convention has a unique orientation).

 Denote $\mathrm{vect}_\kk$  the category  whose objects are finitely generated projective $\kk$-modules and whose morphisms are $\kk$-homomorphisms of modules.  We view $\mathrm{vect}_\kk$ as a symmetric monoidal category with standard tensor    product and unit object $\kk$. A {\it 3-dimensional TQFT} is a symmetric monoidal functor $Z\co \mathrm{Cob}_3 \to \mathrm{vect}_\kk$. In particular, $Z(\emptyset)=\kk$,  $Z( \Sigma  \sqcup \Sigma') = Z(
\Sigma ) \otimes Z( \Sigma') $ for any  closed oriented
surfaces $\Sigma, \Sigma'$, and similarly    for  morphisms.


Each compact oriented 3-manifold $M$ determines two  morphisms $\emptyset \to \partial M$ and   $  -\partial M \to  \emptyset$ in $\mathrm{Cob}_3$. The associated   homomorphisms
$Z(\emptyset)=\kk \to Z
(\partial M)$ and $ Z
(-\partial M)\to  Z(\emptyset)=\kk$ are denoted $Z(M, \emptyset,  \partial M)$ and $Z(M, -\partial M, \emptyset)$, respectively. If $\partial M=\emptyset$, then
 $Z(M, \emptyset,  \partial M)=Z(M, -\partial M, \emptyset)\co  \kk \to \kk$ is multiplication by an element of $\kk$ denoted $Z(M)$.

The category     $\mathrm{Cob}_3$ includes as a subcategory  the category of closed oriented surfaces and (isotopy classes of) orientation-preserving homeomorphisms of surfaces.
 Indeed,
 any such  homeomorphism  $f \co \Sigma \to \Sigma'$
determines  a morphism  $  \Sigma \to \Sigma' $ in $\mathrm{Cob}_3$ represented by the pair $(C =\Sigma' \times [0,1], h\co    (-\Sigma) \sqcup\Sigma'  \simeq\partial C)$, where
  $h  (x ) =(f(x), 0) $  for $x\in \Sigma$ and    $h(x') =(x', 1)$ for $x'\in \Sigma'$.    Restricting a   TQFT $Z\co   \mathrm{Cob}_3\to \mathrm{vect}_\kk$ to this  subcategory, we obtain the action of   homeomorphisms  induced  by $Z$.

An {\it isomorphism} of  3-dimensional TQFTs  $Z_1\to Z_2$  is a  natural monoidal isomorphism of   functors.
 Such an isomorphism is
  a system
of  $\kk$-isomorphisms $Z_1( \Sigma)\simeq  Z_2( \Sigma)$,  where $\Sigma$ runs over all closed oriented surfaces.
These  $\kk$-isomorphisms should be multiplicative with respect to
disjoint unions of surfaces  and commute   with the action of
  cobordisms (and in particular, of homeomorphisms).    For $\Sigma=\emptyset$, the isomorphism
$Z_1( \Sigma)\simeq  Z_2( \Sigma)$ should be  the
identity map $\kk\to \kk$. This implies that if two TQFTs $Z_1$, $Z_2$ are isomorphic, then $Z_1(M)= Z_2(M)$ for any closed oriented 3-manifold $M$.

\subsection{Invariants of  $I$-colored graphs}\label{sec-Io3m+}
  Fix up to the end of Section~\ref{sec-TQFT}  a   spherical fusion
category~$\cc$ over~$\kk$ such that $\dim(\cc)$  is invertible
in~$\kk$. Fix a
representative set $I$ of simple objects of $\cc$. We shall derive from  $ \cc$ and $ I$ a 3-dimensional TQFT.

By  an  \emph{$I$-colored graph}  in a
surface,  we mean a  $\cc$-colored graph such that the colors of all
edges belong to~$I$ and all vertices   have valence $\geq 2$. For
any compact  oriented 3-manifold~$M$ and any $I$-colored graph~$G$
in~$\partial M$, we   define   a topological invariant $|M, G| \in
\kk$   as follows. Pick a skeleton~$P\subset M$ of the pair $(M,G)$.
Pick  a map $c \co \Reg(P) \to I$ extending the coloring of $G$ in
the sense   that for every edge~$a$ of~$G$, the value of $c$ on the
region of~$P$ adjacent to~$a$ is the
 $\cc$-color of~$a$.
For every oriented edge $e$ of $P $,  consider the  \kt module
$H_c(e)=H(P_e)$, where~$P_e$ is the set of branches of~$P$ at~$e$
turned into a cyclic $\cc$\ti set as  in Section~\ref{sec-Io3m}. Let
$E_0$ be the set of oriented   edges of~$P$ with both endpoints in
$\Int(M)$, and let~$E_\partial$ be the set of edges of~$P$ with
exactly one endpoint  in $\partial M$ oriented towards this
endpoint. Note that every vertex~$v$ of~$G$ is incident
to  a unique edge $e_v$ belonging to $ E_\partial$ and $ H_c(e_v) = H_v(G^\opp;-\partial M)$, where   the
orientation of~$\partial M$ is induced by that
 of~$M$. Therefore
 $$
 \otimes_{e\in E_\partial}\, H_c(e)^\star=\otimes_v \, H_v(G^\opp;-\partial M)^\star=H(G^\opp;-\partial M)^\star.
 $$
For  $e\in E_0$,  the equality $P_{e^\opp}=(P_e)^\opp$  induces a   duality between
 the modules
 $H_c(e)$, $H_c(e^\opp)$ and a
 contraction  $   H_c(e)^\star \otimes H_c(e^\opp)^\star \to\kk$.  This contraction does not depend on the orientation of~$e$
  up to permutation of the factors. Applying these contractions, we obtain a homomorphism
$$\ast_P \co \otimes_{e\in E_0 \cup E_\partial}\, H_c(e )^\star \longrightarrow  \otimes_{e\in E_\partial}\, H_c(e)^\star=H(G^\opp;-\partial M)^\star .$$
 As in Section~\ref{sect-skeletons},
any  vertex   $x$ of~$P $ lying in $\Int(M)$ determines an oriented
graph~$\Gamma_x$ on $S^2$, and the mapping $c$ turns $\Gamma_x$ into
an $I$-colored graph.    Section~\ref{sect-graph-S2} yields a tensor
$\inv_\cc (\Gamma_x) \in H_c(\Gamma_x)^*$. Here
 $H_c(\Gamma_x)= \otimes_e\, H_c(e)$, where~$e$ runs over all edges of~$P$ incident to~$x$ and oriented away from~$x$. The tensor product $\otimes_x \,\inv_\cc (\Gamma_x)$ over all
vertices~$x$ of~$P$  lying in $\Int(M)$  is a vector in
 $\otimes_{e\in E_0 \cup E_\partial}\, H_c(e )^\star $.

\begin{thm}\label{thm-state-3man+r}
For a  skeleton~$P$ of $(M,G)$, set
\begin{equation*} |M,G|=(\dim (\cc))^{-\vert P\vert} \sum_{c} \,\,  \left ( \prod_{r \in \Reg(P)} (\dim c(r))^{\chi(r)} \right ) \,
  {\ast}_P ( \otimes_x \,\inv_\cc (\Gamma_x)) , \end{equation*}
where   $\vert P\vert$ is the number of components of $M\setminus
P$,  $c$ runs over all maps $ \Reg(P) \to I$ extending the coloring
of $G$, and $\chi$ is the Euler characteristic. Then $|M,G|\in H(G^\opp;-\partial M)^\star$ does
not depend on the choice of~$P$.
\end{thm}

\begin{proof} Since any two skeletons  of $(M,G)$ are related by primary moves, we need only to verify the invariance of
$|M,G|$ under these moves. This invariance is a   local property
verified exactly as in the proof of Theorem~\ref{thm-state-3man}.
\end{proof}

Though there is a canonical isomorphism $H(G^\opp;-\partial M)^\star\simeq H(G;\partial M)$ (see the last remark of Section~\ref{sect-graph}), we view  $|M,G|$ as an element of $ H(G^\opp;-\partial M)^\star$.

Taking  $G=\emptyset$, we obtain a scalar topological invariant
 $|M|_\cc= |M,\emptyset|\in H(\emptyset)^\star=\kk $ of $M$. This generalizes the invariant defined above  for closed $M$.

\subsection{Invariants of  3-cobordisms}\label{sec-Io3m++}
A {\it 3-cobordism} is  a triple $(M, \Sigma_0, \Sigma_1)$, where $M$ is a compact oriented 3-manifold and  $\Sigma_0, \Sigma_1$ are disjoint  closed oriented surfaces contained in  $\partial M$ such that
$\partial M = (-\Sigma_0) \sqcup\Sigma_1$ in the category of oriented manifolds. We call $\Sigma_0$ and $\Sigma_1$ the bottom base and the top base of $M$, respectively.

Consider  a 3-cobordism  $(M, \Sigma_0 ,  \Sigma_1)$ and  an $I$-colored graph $G_i\subset \Sigma_i$ for
$i=0,1 $.
 Theorem~\ref{thm-state-3man+r} yields
  a vector
  $$|M,G_0^\opp \cup G_1 | \in H(G_0 \cup  G_1^\opp, -\partial M)^\star= H(G_0, \Sigma_0)^\star \otimes H( G_1^\opp, -\Sigma_1)^\star.$$
The isomorphism $H( G_1^\opp, -\Sigma_1)^\star\simeq H( G_1, \Sigma_1)$ given by the last remark of Section~\ref{sect-graph} induces an isomorphism
$$
\Upsilon\co H(G_0, \Sigma_0)^\star \otimes H( G_1^\opp, -\Sigma_1)^\star \to \Hom_\kk\bigl (H(G_0, \Sigma_0),H( G_1, \Sigma_1)  \bigr ).
$$
Set $$|M, \Sigma_0, G_0,  \Sigma_1, G_1|= \frac{(\dim (\cc))^{\vert G_1\vert}}{\dim (G_1)} \,  \Upsilon \bigl(|M,G_0^\opp \cup G_1|\bigr)\co H(G_0;\Sigma_0) \to   H(G_1;\Sigma_1)
 ,$$
 where for an $I$-colored graph $G$ in a surface $
 \Sigma$, the symbol  $\vert G \vert$ denotes the number of components of $\Sigma \setminus G$ and $\dim (G)$
 denotes the product of the dimensions of the objects of~$\cc$ associated with the edges
 of~$G$. To compute the homomorphism $|M, \Sigma_0, G_0,  \Sigma_1, G_1|$, let $\Omega\in H( G_1, \Sigma_1) \otimes H( G_1^\opp, -\Sigma_1)$ be the inverse of the canonical pairing $ H( G_1^\opp, -\Sigma_1) \otimes H( G_1,\Sigma_1) \to \kk$. Pick any expansion $\Omega=\sum_\alpha a_\alpha \otimes b_\alpha$, where $a_\alpha \in H( G_1, \Sigma_1)$ and $ b_\alpha \in H( G_1^\opp, -\Sigma_1)$. Then for any $h \in H(G_0,\Sigma_0)$,
 $$
 |M, \Sigma_0, G_0,  \Sigma_1, G_1|(h)= \frac{(\dim (\cc))^{\vert G_1\vert}}{\dim (G_1)} \,\sum_{\alpha} |M,G_0^\opp \cup G_1 |(h \otimes b_\alpha)\,a_\alpha.
 $$

By \emph{skeleton} of
a closed surface~$\Sigma$ we mean an oriented graph $G\subset \Sigma$
such that all vertices of~$G$ have valence $\geq 2$ and all
components of $\Sigma\setminus G$ are open disks.  For example, the
vertices and the edges   of a triangulation~$t$  of~$\Sigma$ (with
an arbitrary orientation of the edges) form a skeleton of~$\Sigma$.
For a graph~$G$,   denote by $\mathrm{col}   (G)$ the set of all
maps from the set of edges of~$G$ to~$I$.

 \begin{lem}\label{lem-skeletons81}
Let $(M_0, \Sigma_0,   \Sigma_1 )$,  $(M_1, \Sigma_1,   \Sigma_2 )$
be two 3-cobordisms and   $ (M,   \Sigma_0,   \Sigma_2 )$  be the
3-cobordism   obtained by gluing $M_0$ and $M_1$ along $\Sigma_1$.
For any $I$-colored graphs  $G_0\subset \Sigma_0$, $G_2\subset
\Sigma_2$ and any skeleton~$G$ of $\Sigma_1$,
$$|M, \Sigma_0, G_0,  \Sigma_2, G_2|=\sum_{c\in \mathrm{col}   (G)}
 \, |M_1, \Sigma_1, (G,c),  \Sigma_2, G_2| \circ |M_0, \Sigma_0, G_0,  \Sigma_1, (G,c) |.$$
\end{lem}

\begin{proof} This  is a direct consequence of the definitions since
 the union of a  skeleton of  $(M_0, G_0^\opp \cup G)$  with a  skeleton of   $(M_1, G^\opp \cup G_2)$ is a skeleton of  $(M, G_0^\opp \cup G_2)$.
\end{proof}

\subsection{The state-sum TQFT}\label{sec-TQFT-}
For a skeleton $G$ of a closed oriented surface $\Sigma$,
set
$$\vert G; \Sigma\vert^\circ =\oplus_{c\in \mathrm{col}   (G)} \,  H((G,c); \Sigma).$$
  Given a 3-cobordism $(M, \Sigma_0 , \Sigma_1)$, we
define for any skeletons $G_0\subset \Sigma_0$ and $G_1\subset
\Sigma_1$ a homomorphism $$|M, \Sigma_0, G_0,  \Sigma_1, G_1|^\circ \co \vert G_0; \Sigma_0\vert^\circ\to \vert G_1; \Sigma_1\vert^\circ$$ by
\begin{equation}\label{eq-func-} |M, \Sigma_0, G_0,  \Sigma_1, G_1|^\circ=\sum_{\substack{c_0 \in \mathrm{col}(G_0)\\ c_1 \in \mathrm{col}(G_1)}} |M, \Sigma_0, (G_0, c_0),  \Sigma_1, (G_1, c_1)| , \end{equation}
where $|M, \Sigma_0, (G_0, c_0),  \Sigma_1, (G_1, c_1)|\co H((G_0,c_0); \Sigma_0) \to H((G_1,c_1); \Sigma_1)$. Lemma~\ref{lem-skeletons81} implies
that for any cobordisms $M_0, M_1, M$ as in this lemma and for any
skeletons $G_i \subset \Sigma_i$ with $i=0, 1,2 $,
\begin{equation}\label{eq-func} |M, \Sigma_0, G_0,  \Sigma_2, G_2|^\circ=|M,
\Sigma_1, G_1,  \Sigma_2, G_2|^\circ \circ  |M, \Sigma_0, G_0,
\Sigma_1, G_1|^\circ.\end{equation} These constructions assign a
finitely generated free module to every closed oriented surface with
distinguished skeleton and   a homomorphism of these modules to
every 3-cobordism whose bases are endowed with skeletons.
   This data   satisfies an
appropriate version of axioms of a TQFT except one:   the
homomorphism associated with the cylinder over a surface, generally
speaking, is not the identity. There is a standard procedure which
transforms such a \lq \lq pseudo-TQFT"   into a genuine
  TQFT and gets rid of the skeletons of surfaces  at the
same time. This procedure is described in detail in a similar
setting in \cite{Tu1}, Section VII.3. The idea is that if $G_0,G_1$
are two skeletons  of a closed oriented surface $\Sigma$, then the
cylinder cobordism $M=\Sigma\times [0,1]$ gives a homomorphism
$$p(G_0,G_1)=|M, \Sigma \times \{0\}, G_0 \times \{0\},  \Sigma \times \{1\}, G_1\times \{1\}|^\circ \colon \vert G_0; \Sigma\vert^\circ \to \vert G_1; \Sigma\vert^\circ\, .$$
Formula~\eqref{eq-func} implies that  $p(G_0,G_2)=p(G_1,G_2)\,
p(G_0,G_1)$ for any skeletons $G_0$, $G_1$, $G_2$ of~$\Sigma$.
Taking $G_0=G_1=G_2$ we obtain that $p(G_0,G_0)$ is a projector onto
a direct summand $\vert G_0; \Sigma\vert $  of $\vert G_0;
\Sigma\vert^\circ $. Moreover,
  $p(G_0,G_1)$ maps
$\vert G_0; \Sigma\vert $ isomorphically onto $\vert G_1;
\Sigma\vert $. The finitely generated projective $\kk$-modules
$\{\vert G; \Sigma\vert\}_G$,
 where $G$ runs over all skeletons of $\Sigma$, and the
homomorphisms $\{p(G_0,G_1)\}_{G_0, G_1}$ form a projective system.
The projective limit of this system is a $\kk$-module,  $\vert
\Sigma\vert $,  independent of the choice of a skeleton of $\Sigma$.
For each skeleton $G$ of~$\Sigma$, we have a  \lq\lq cone
isomorphism" of $\kk$-modules $\vert G; \Sigma\vert \cong \vert
\Sigma\vert$.  By convention, the empty surface $\emptyset$   has a unique (empty) skeleton and $\vert \emptyset \vert =\kk$.

Any 3-cobordism $(M,\Sigma_0, \Sigma_1)$   splits as a product of a
3-cobordism with a cylinder over $\Sigma_1$. Using this splitting
and Formula~\eqref{eq-func}, we obtain that the
homomorphism~\eqref{eq-func-} carries $\vert \Sigma_0\vert \cong
\vert G_0; \Sigma_0\vert \subset \vert G_0; \Sigma_0\vert^{\circ} $
into $\vert \Sigma_1\vert \cong \vert G_1; \Sigma_1\vert \subset
\vert G_1; \Sigma_1\vert^{\circ} $ for any skeletons $G_0, G_1$ of
$\Sigma_0, \Sigma_1$, respectively. This gives a homomorphism  $|M,
\Sigma_0, \Sigma_1 | \colon \vert \Sigma_0\vert\to \vert
\Sigma_1\vert$ independent of the choice of $G_0$, $G_1$.

An   orientation preserving homeomorphism of closed oriented
surfaces $f\colon \Sigma\to \Sigma'$  induces an isomorphism $\vert f\vert
\colon \vert \Sigma\vert  \to \vert \Sigma'\vert$ as follows.  Pick
a skeleton $G$ of $\Sigma$. Then
  $G'=f(G)$ is a   skeleton  of $\Sigma'$,  and
  $\vert f\vert $
   is  the composition of the  isomorphisms
$$\vert \Sigma\vert  \cong   \vert G; \Sigma\vert \cong \vert G'; \Sigma'\vert
   \cong \vert \Sigma'\vert \, .$$
Here the first and the third isomorphisms are the cone isomorphisms
and the middle isomorphism is induced by the homeomorphism  of pairs
$f\co (  \Sigma, G) \to (  \Sigma', G')$. It is easy to check that $\vert f\vert $
does not depend on the choice of $G$.

To accomplish the construction of  the 3-dimensional TQFT $\vert \cdot \vert$, we need only to
associate  with
every  morphism $\varphi \co \Sigma_0   \to  \Sigma_1$  in   $\mathrm{Cob}_3$
the induced homomorphism $\vert  \varphi \vert \co \vert \Sigma_0 \vert \to \vert \Sigma_1 \vert $.
Represent $\varphi$  by a pair $(M, h\co (-\Sigma_0) \sqcup\Sigma_1 \simeq  \partial M)$ as above.  For $i=0,1$ denote by $\Sigma'_i$ the surface $h(\Sigma_i)\subset \partial M$ with orientation induced by the one in $\Sigma_i$.  The 3-cobordism   $(M,  \Sigma'_0,   \Sigma'_1 )$   yields a homomorphism  $|M,
\Sigma'_0, \Sigma'_1 | \colon \vert \Sigma'_0\vert\to \vert
\Sigma'_1\vert$. The homeomorphism
 $h  \co \Sigma_i \to \Sigma'_i$ induces an isomorphism
 $ \vert \Sigma_i\vert\cong  \vert
\Sigma'_i\vert$  for $i=0,1$. Composing these three homomorphisms we obtain the homomorphism
$\vert  \varphi \vert \co \vert \Sigma_0\vert\to \vert
\Sigma_1\vert$.
This homomorphism does not depend on the choice of the representative pair $(M,h)$.
It follows from the definitions and Lemma~\ref{lem-skeletons81} that the
assignment
 $\Sigma\mapsto \vert \Sigma \vert, \,\,\,\, \varphi
\mapsto |\varphi | $ satisfies all the   axioms of a
TQFT.  To stress the dependance of $\cc$, we  shall   denote this TQFT by $\vert \cdot
\vert_\cc$. Considered
  up to isomorphism,  the TQFT  $\vert \cdot
\vert_\cc$ does
not depend on the choice of the representative set $I$ of simple
objects of $\cc$. For any closed oriented 3-manifold~$M$, the invariant $\vert M\vert_\cc \in \kk$ produced by this TQFT   coincides with
the invariant  of Sections~\ref{sec-ssot3m}
and~\ref{sec-skeletonsstatesums}.

\subsection{Computation of $|S^2|_\cc$}\label{Esp-S2}
To illustrate our definitions, we compute the $\kk$-module $\vert S^2\vert_\cc$.  A circle $G\subset \R^2\subset \R^2\cup\{\infty\}=S^2$  oriented counterclockwise and viewed as a graph with one  vertex $x$ and one edge $e$  is a skeleton of $S^2$.  Assigning $i\in I$ to  $e$, we turn $G$ into an $I$-colored graph   $G_i$. By definition,  the $\kk$-module  $\vert S^2\vert_\cc$ is isomorphic to the image of the endomorphism  $p(G,G)=\sum_{i,j \in I} p_i^j$ of $|G; S^2|^\circ=\oplus_{i\in I} H(G_{i})$, where
$$p_i^j= |\Sigma \times [0,1], \Sigma \times \{0\}, G_i  \times \{0\}, \Sigma  \times \{1\}, G_j  \times \{1\}|_\cc  \co H(G_i) \to H(G_j).$$
To compute $p_{i}^j$,
consider the 2-polyhedron $P=(G\times [0,1])\cup (S^2 \times \{\frac{1}{2}\}) $ in $ S^2 \times [0,1]$.   We stratify $P$ by taking as  edges the arcs $x\times [0,\frac{1}{2}]$,  $x\times [\frac{1}{2},1]$, and    $e\times \{t\}$ for $t\in \{0, \frac{1}{2},1\}$. The polyhedron $P$ has 3 vertices $ x\times \{t\}$  with $t\in \{0, \frac{1}{2},1\}$ and 4 disk  regions. We orient the two regions  adjacent to the boundary so that $P$ is a skeleton of the pair $(S^2 \times [0,1],
(G^\opp \times \{0\}) \cup (G \times \{ 1\})  )$ and  endow  the two regions contained in $S^2 \times \{\frac{1}{2}\}$ with orientation induced by that of  $S^2  $.  Clearly, $\vert P\vert =4$. Set $$G_i^j=(G_i^\opp \times \{0\}) \cup (G_j \times \{ 1\})\subset \partial(S^2 \times [0,1]).$$ The maps $  \Reg(P) \to I$ extending the coloring of $G_i^j$ are numerated by the colors   $z, w\in I$ of the   region contained in $S^2 \times \{\frac{1}{2}\}$. The link of the  vertex $ (x,\frac{1}{2})$ of $P$
determines a $\cc$-colored  graph $ \Gamma^{z,w}$  in $S^2$:
\begin{equation*}
 \psfrag{i}[Bc][Bc]{\scalebox{.9}{$i$}}
 \psfrag{j}[Bc][Bc]{\scalebox{.9}{$j$}}
  \psfrag{z}[Bc][Bc]{\scalebox{.9}{$z$}}
 \psfrag{t}[Bc][Bc]{\scalebox{.9}{$w$}}
 P=
 \rsdraw{.45}{.9}{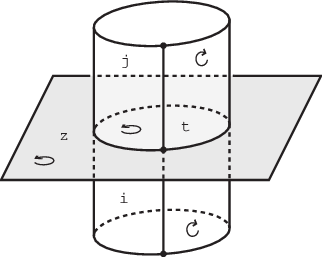}\,, \qquad \qquad
 \psfrag{i}[Bc][Bc]{\scalebox{.9}{$i$}}
 \psfrag{j}[Bc][Bc]{\scalebox{.9}{$j$}}
 \psfrag{z}[Bc][Bc]{\scalebox{.9}{$z$}}
 \psfrag{t}[Bc][Bc]{\scalebox{.9}{$w$}}
 \Gamma^{z,w}  =\rsdraw{.45}{.9}{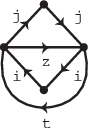}\,.
\end{equation*}
Let $u,v$ be the bottom and the top vertices of $\Gamma^{z,w}  $, respectively. Then
\begin{equation*}
|S^2 \times [0,1], G_i^j |_\cc= \frac{\dim(i)\dim(j)}{\dim (\cc)^{4}} \sum_{z,w \in I}  \dim(z)\dim(w)  \,
 \mu_{i,j}^{z,w}
\end{equation*}
with
\begin{equation*}
 \psfrag{i}[Bc][Bc]{\scalebox{.9}{$i$}}
 \psfrag{j}[Bc][Bc]{\scalebox{.9}{$j$}}
 \psfrag{z}[Bc][Bc]{\scalebox{.9}{$z$}}
 \psfrag{t}[Bc][Bc]{\scalebox{.9}{$w$}}
  \mu_{i,j}^{z,w}= {\ast}_P (\inv_\cc (\Gamma^{z,w}))= \inv_\cc\left ( \rsdraw{.45}{.9}{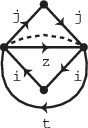} \right )\in H_u(\Gamma^{z,w}  )^* \otimes H_v(\Gamma^{z,w}  )^*,
\end{equation*}
where the dotted line represents  the tensor contraction. Note that $H_u(\Gamma^z)=H(G_i)$ and $H_v(\Gamma^z)=H(G_j^\opp)$.
Set $a_i=\tau_i^{-1}(\rcoev_i) \in H(G_i)$ where $\tau_i \co H(G_i) \to \Hom_\cc(\un,i^* \otimes i)$ is the cone isomorphism. Then the vector $a_i$ forms a basis of $H(G_i)$. Since $H(G_j^\opp)=H(G_j)$, $\Omega=(\dim(j))^{-1} a_j \otimes a_j \in H(G_j) \otimes H(G_j^\opp)$ is the inverse of the canonical pairing $H(G_j^\opp) \otimes H(G_j) \to \kk$. Now
\begin{equation*}
 \psfrag{i}[Bc][Bc]{\scalebox{.9}{$i$}}
 \psfrag{j}[Bc][Bc]{\scalebox{.9}{$j$}}
 \psfrag{z}[Bc][Bc]{\scalebox{.9}{$z$}}
 \psfrag{t}[Bc][Bc]{\scalebox{.9}{$w$}}
  \mu_{i,j}^{z,w}(a_i \otimes a_j)= \inv_\cc\left ( \rsdraw{.45}{.9}{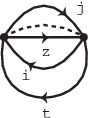} \right )=N^\un_{j^*\otimes z^* \otimes i \otimes w}
\end{equation*}
Therefore
\begin{align*}
p_i^j(a_i)
&=\frac{\dim (\cc)^{2}}{\dim(j)^2}\,|S^2 \times [0,1], G_i^j |_\cc(a_i \otimes a_j)\, a_j \\
&= \frac{\dim(i)}{\dim(j)\dim (\cc)^{2}} \sum_{z,w \in I}  \dim(z)\dim(w) N^\un_{j^*\otimes z^* \otimes i \otimes w} a_j\\
&=\frac{\dim(i)^2}{\dim (\cc)} \,a_j \quad \text{by the first formula of Lemma~\ref{bubbleidentity}.}
\end{align*}
We conclude that the image of $p(G,G)$ is generated by $v=\sum_{j \in I} a_j \in |G; S^2|^\circ$, and so that $|S^2|_\cc\simeq\kk$.

\section{Modular categories and categorical centers}

We recall the basics on  modular categories and the Drinfeld
center.

\subsection{Modular categories (\cite{Tu1})}\label{sec-modulcat}
 A \emph{braiding} in a monoidal category $\bb$ is a natural isomorphism
$ \tau=\{\tau_{X,Y} \co  X \otimes Y\to Y \otimes X\}_{X,Y \in
\Ob(\bb)} $ such that
\begin{equation*}
\tau_{X, Y\otimes Z}=(\id_Y \otimes \tau_{X, Z})(\tau_{X, Y} \otimes \id_Z) \quad \text{and} \quad
\tau_{X\otimes Y,Z}=(\tau_{X,Z} \otimes \id_Y)(\id_X \otimes \tau_{Y,Z}).
\end{equation*}
These conditions imply that $\tau_{X,\un}=\tau_{\un,X}=\id_X$ for
all $X\in\Ob(\bb)$ and so that \eqref{special} is satisfied.

 A monoidal category endowed with a braiding is said
to be \emph{braided}. The braiding    and its inverse are depicted
as follows
\begin{center}
\psfrag{X}[Bc][Bc]{\scalebox{.8}{$X$}}
\psfrag{Y}[Bc][Bc]{\scalebox{.8}{$Y$}}
$\tau_{X,Y}=\,$\rsdraw{.45}{.9}{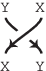} \quad \text{and} \quad
$\tau^{-1}_{Y,X}=\,$\rsdraw{.45}{.9}{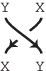}.
\end{center}

%

 For any object $X$ of a braided pivotal category
$\bb$, one defines a morphism   $$ \theta_X  =
\psfrag{X}[Bc][Bc]{\scalebox{.8}{$X$}}\rsdraw{.45}{.9}{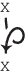}\,=(\id_X
\otimes \rev_X)(\tau_{X,X} \otimes \id_{X^*})(\id_X \otimes
\lcoev_X) \co X\to X.
$$ This morphism, called the  \emph{twist}, is invertible and
$$\theta_X^{-1}=\psfrag{X}[Bc][Bc]{\scalebox{.8}{$X$}}\rsdraw{.45}{.9}{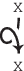}=(\lev_X \otimes \id_X)(\id_{X^*} \otimes \tau_{X,X}^{-1})(\rcoev_X \otimes \id_X)\co X \to X.$$
Note that
 $\theta_\un=\id_\un$,
$\theta_{X\otimes Y}=(\theta_X \otimes \theta_Y)\tau_{Y,X}\tau_{X,Y}$ for any $X,Y\in \Ob(\bb)$. The twist is natural:
 $\theta_Y f=f \theta_X$ for any morphism $f\co X\to Y$ in~$\bb$.

A \emph{ribbon category} is a braided pivotal category $\bb$ whose
twist  is self-dual, i.e., $(\theta_X)^*=\theta_{X^*}$ for all $X\in
\Ob(\bb)$. This   is equivalent to the equality of morphisms
$\psfrag{X}[Bc][Bc]{\scalebox{.8}{$X$}} \rsdraw{.45}{.9}{theta1}
=\rsdraw{.45}{.9}{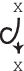} $ for any  $X\in \Ob(\bb)$.
  In   a ribbon category   $\theta_X^{-1}=\psfrag{X}[Bc][Bc]{\scalebox{.8}{$X$}}
  \rsdraw{.45}{.9}{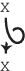}=\rsdraw{.45}{.9}{theta2}$.
  A ribbon category $\bb$  is spherical and  gives rise to  topological invariants of links in $S^3$. Namely,   every   $\bb$-colored framed oriented link   $L\subset S^3$ determines an endomorphism of the unit object  $F_\bb(L)\in \End (\un)$ which turns out to be a topological invariant of $L$. Here  $ L$ is  $\bb$-colored if every component of $L$ is endowed with an object of $\bb$ (called the color of this component).
The definition of $F_\bb(L)$ goes by an application of the Penrose calculus  to a diagram of $L$; a new feature is that    with the positive and negative crossings of the diagram one associates the braiding and its inverse, respectively. For more on this, see \cite{Tu1}.

A \emph{modular category} (over $\kk$) is a ribbon fusion category
$\bb$ (over $\kk=\End (\un)$) such that the matrix
$S=[\tr(\tau_{j,i}\tau_{i,j})]_{i,j \in \mathcal J}$ is invertible, where $\mathcal J$
is a representative set of simple objects of $\bb$ and $\tau$ is the
braiding of $\bb$. The matrix $S$ is called the \emph{$S$-matrix}
of~$\bb$.  Note that for any simple object $J$ of $\bb$, the twist
$\theta_J\co J \to J$ is multiplication by an invertible scalar
$v_J\in \kk$ called the {\it twist scalar} of $J$. This scalar
depends only on the isomorphism class of $J$ and $v_{J^*} =v_{J} $.
Set $\Delta_{\pm} =\sum_{i\in \mathcal J} v_i^{\pm 1} (\dim (i))^2\in \kk$.
It is known that $ \Delta_+\Delta_-=\dim (\bb) = \sum_{i\in \mathcal J} (\dim
(i))^2$, see \cite{Tu1}. We say that $\bb$ is {\it anomaly free} if
$\Delta_+=\Delta_-$.

\subsection{The center of a  monoidal category}\label{sect-centerusual}
Let $\cc$ be a monoidal category. A \emph{half braiding} of $\cc$ is
a pair $({{A}},\sigma)$, where ${{A}} \in \Ob(\cc)$ and
\begin{equation*}
\sigma=\{\sigma_X \co  {{A}} \otimes X\to X \otimes {{A}}\}_{X \in \Ob(\cc)}
\end{equation*}
is a natural isomorphism such that
 \begin{equation}\label{axiom-half-braiding}  \sigma_{X \otimes Y}=(\id_X \otimes
\sigma_Y)(\sigma_X \otimes \id_Y)\end{equation} for all
$X,Y \in \Ob(\cc)$. This implies that $\sigma_\un=\id_{{A}}$.
We call ${{A}}$ the {\it   underlying object} of the half braiding
$({{A}},\sigma)$.

The \emph{center of $\cc$} is the braided category $\zz(\cc)$
defined as follows. The objects of~$\zz(\cc)$ are half braidings of
$\cc$. A morphism $({{A}},\sigma)\to ({{A}}',\sigma')$ in $\zz(\cc)$ is a
morphism $f \co {{A}} \to {{A}}'$ in $\cc$ such that $(\id_X \otimes
f)\sigma_X=\sigma'_X(f \otimes \id_X)$ for all $X\in\Ob(\cc)$. The
  unit object of $\zz(\cc)$ is $\un_{\zz(\cc)}=(\un,\{\id_X\}_{X \in
\Ob(\cc)})$ and the monoidal product is
\begin{equation*}
({{A}},\sigma) \otimes ({{B}}, \rho)=\bigl({{A}} \otimes {{B}},(\sigma \otimes \id_{{B}})(\id_{{A}} \otimes \rho) \bigr).
\end{equation*}
The braiding $\tau$ in $\zz(\cc)$ is defined by
$$\tau_{({{A}},\sigma),({{B}}, \rho)}=\sigma_{{{B}}}\co ({{A}},\sigma) \otimes
({{B}}, \rho) \to ({{B}}, \rho) \otimes ({{A}},\sigma).$$

There is a  {\it forgetful functor} $\zz(\cc)\to \cc$ assigning to
every half braiding $({{A}}, \sigma)$ the underlying object ${{A}}$ and
acting in the obvious way on the morphisms. This is a strict monoidal functor.

If $\cc$ is a monoidal \kt category, then so~$\zz(\cc)$ and
the forgetful functor is  \kt linear. Observe that $\End_{\zz(\cc)}(\un_{\zz(\cc)})=\End_\cc(\un) = \kk$.

If $\cc$ is pivotal, then so is $\zz(\cc)$ with
$({{A}},\sigma)^*=({{A}}^*,\sigma^\vee)$, where
$$
 \psfrag{M}[Bc][Bc]{\scalebox{.9}{${{A}}$}}
 \psfrag{X}[Bc][Bc]{\scalebox{.9}{$X$}}
 \psfrag{s}[Bc][Bc]{\scalebox{.9}{$\sigma_{X^*}$}}
\sigma^\vee_X= \rsdraw{.45}{1}{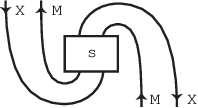} \co {{A}}^* \otimes X \to X \otimes {{A}}^*,
$$
and $\lev_{({{A}},\sigma)}= \lev_{{A}}$, $\lcoev_{({{A}},\sigma)}= \lcoev_{{A}}$,
$\rev_{({{A}},\sigma)}= \rev_{{A}}$, $\rcoev_{({{A}},\sigma)}= \rcoev_{{A}}$. The
(left and right) traces of morphisms and   dimensions of objects
in~$\zz(\cc)$ are the same as in~$\cc$. If $\cc$ is spherical, then
so is $\zz(\cc)$.


\subsection{The center of a fusion category}
Let $\cc$ be a spherical fusion category over~$\kk$. Fix a
representative set~$I$ of simple objects of $\cc$.

\begin{lem}\label{lem-center-ribbon}
The center $\zz(\cc)$ of $\cc$ is ribbon.
\end{lem}
\begin{proof}
  We need only to verify that
$\psfrag{X}[Bc][Bc]{\scalebox{.8}{$X$}}
\rsdraw{.45}{.9}{theta1}=\rsdraw{.45}{.9}{theta2inv}  $ for any half
braiding $X=({{A}},\sigma)$ of~$\cc$. Let $(p_\alpha\co {{A}} \to i_\alpha,
q_\alpha\co i_\alpha \to {{A}})_{\alpha \in \Lambda}$ be an
$I$-partition of ${{A}}$. For any $\alpha, \beta \in \Lambda$ such that
$i_\alpha=i_\beta=i\in I$  we obtain using the naturality of
$\sigma$ that
\begin{align*}
\psfrag{M}[Bc][Bc]{\scalebox{.9}{${{A}}$}}
\psfrag{i}[Bc][Bc]{\scalebox{.9}{$i$}}
\psfrag{c}[Bc][Bc]{\scalebox{1}{$\sigma_{{A}}$}}
\psfrag{r}[Bc][Bc]{\scalebox{1}{$\sigma_{i}$}}
\psfrag{p}[Bc][Bc]{\scalebox{1}{$p_\alpha$}}
\psfrag{q}[Bc][Bc]{\scalebox{1}{$q_\beta$}}
\rsdraw{.45}{.9}{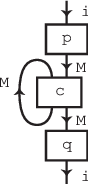} & \psfrag{M}[Bc][Bc]{\scalebox{.9}{${{A}}$}}
\psfrag{i}[Bc][Bc]{\scalebox{.9}{$i$}}
\psfrag{c}[Bc][Bc]{\scalebox{1}{$\sigma_{{A}}$}}
\psfrag{r}[Bc][Bc]{\scalebox{1}{$\sigma_{{i}}$}}
\psfrag{p}[Bc][Bc]{\scalebox{1}{$p_\alpha$}}
\psfrag{q}[Bc][Bc]{\scalebox{1}{$q_\beta$}}
  = (\dim({i}))^{-1} \;\; \rsdraw{.45}{.9}{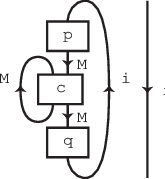}\; = (\dim({i}))^{-1} \rsdraw{.45}{.9}{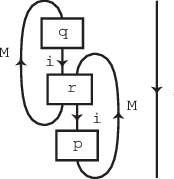} \\
& \psfrag{M}[Bc][Bc]{\scalebox{.9}{${{A}}$}}
\psfrag{i}[Bc][Bc]{\scalebox{.9}{${i}$}}
\psfrag{c}[Bc][Bc]{\scalebox{1}{$\sigma_{{A}}$}}
\psfrag{r}[Bc][Bc]{\scalebox{1}{$\sigma_{{i}}$}}
\psfrag{a}[Bc][Bc]{\scalebox{1}{$\alpha$}}
\psfrag{p}[Bc][Bc]{\scalebox{1}{$p_\alpha$}}
\psfrag{q}[Bc][Bc]{\scalebox{1}{$q_\beta$}}= (\dim({i}))^{-1} \; \rsdraw{.45}{.9}{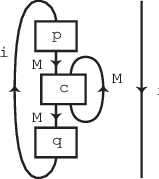} \; =\; \rsdraw{.45}{.9}{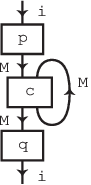}\,.
\end{align*}
We conclude using that any $f\in\End_\cc({{A}})$ expands as
$f=\sum_{\alpha, \beta \in \Lambda} q_\alpha (p_\alpha f q_\beta)
p_\beta$ and that $p_\alpha f q_\beta=0$ if $i_\alpha\neq i_\beta$.
\end{proof}

\begin{lem}[{\cite[Lemma 3.10]{Mu}}]\label{lem-proj-alg}
If $\dim(\cc)$ is invertible in $\kk$, then   for any    half
braidings $({{A}},\sigma)$ and $({{B}}, \rho)$ of~$\cc$, the \kt linear
endomorphism $\pi_{({{A}},\sigma)}^{({{B}}, \rho)}$ of $\Hom_\cc({{A}},{{B}})$
defined by
$$
\psfrag{M}[Bc][Bc]{\scalebox{.9}{${{A}}$}} \psfrag{N}[Bc][Bc]{\scalebox{.9}{${{B}}$}} \psfrag{f}[Bc][Bc]{\scalebox{.9}{$f$}}\psfrag{i}[Bc][Bc]{\scalebox{.9}{$i$}} \psfrag{a}[Bc][Bc]{\scalebox{1}{$\sigma_{i^*}$}}\psfrag{c}[Bc][Bc]{\scalebox{1}{$\rho_i$}}
\pi_{({{A}},\sigma)}^{({{B}}, \rho)}(f)=(\dim(\cc))^{-1}\sum_{i\in I}\dim(i) \rsdraw{.45}{.9}{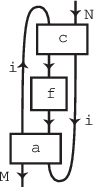}
$$
is a projector onto $\Hom_{\zz(\cc)}(({{A}},\sigma),({{B}}, \rho))$.
\end{lem}

\begin{thm}[{\cite[Theorem 1.2, Proposition 5.18]{Mu}}]\label{thm-center-modular}
Let $\cc$ be a spherical fusion category over an algebraically
closed field such that $\dim \cc\neq 0$. Then $\zz(\cc)$ is an
anomaly free modular category with $\Delta_+=\Delta_-=\dim \cc$.
\end{thm}

Note that $\dim \zz(\cc)=\Delta_+ \Delta_-= (\dim \cc)^2$.

\section{Main theorems and applications}\label{sec-TQFT+}

\subsection{Main theorems}
The Reshetikhin-Turaev construction (see~\cite{Tu1}) derives from any modular
category $\bb$ over $\kk$ equipped with a distinguished square root
of $\dim (\bb)$     a 3-dimensional \lq \lq   extended  TQFT" $\tau_\bb$. The latter
is a functor from a certain extension of the category $  \mathrm{Cob}_3 $ to $\mathrm{vect}_\kk$; the extension in question is formed by surfaces
with a Lagrangian subspace in the real 1-homology. For an anomaly
free $\bb$, we   take the element $\Delta=\Delta_\pm \in \kk$
defined in Section~\ref{sec-modulcat} as the distinguished square
root of $\dim (\bb)$. The corresponding extended TQFT  $\tau_\bb$ does not
involve  Lagrangian spaces and is a TQFT in the sense of Section~\ref{Preliminaries on  TQFTs}.

  We recall   the definition  of $\tau_\bb (M)\in \kk$ for a closed oriented 3-manifold
$M$ and anomaly free $\bb$. Pick a   representative set  $\mathcal J$
  of simple objects of $\bb$. Present  $M$ by surgery on $S^3$ along a framed link $L=L_1\cup \cdots \cup L_N$.  Denote  $\mathrm{col}(L)$  the set of maps $\{ 1, \dots,  N\} \to \mathcal J$ and, for $\lambda\in\mathrm{col}(L)$, denote $L_\lambda$   the framed  link $L$ whose component $L_q$ is oriented in an arbitrary way and colored by $\lambda( q)$ for all $q=1,..., N$. Then
  $$
\tau_{\bb}(M)=\Delta^{-N-1}\sum_{\lambda\in\mathrm{col}(L)} \left(\prod_{q=1}^N \dim\bigl(\lambda( q)\bigr) \right) F_{\bb}(L_\lambda)
$$
where $F_\bb$ is the invariant of $\bb$-colored framed oriented links in $S^3$ discussed in Section \ref{sec-modulcat}. In particular, we can apply these   results to the
anomaly free modular category $\bb=\zz(\cc)$  provided by  Theorem~\ref{thm-center-modular}.
We can   now state the main theorems of this paper.

\begin{thm}\label{thm-equivalence-}
Let $\cc$ be a spherical fusion category over an algebraically
closed field  such that $\dim \cc\neq 0$. Then $\vert M \vert_\cc=\tau_{\zz(\cc)} (M)$ for any closed oriented 3-manifold
$M$.\end{thm}

This equality extends to an isomorphism of TQFTs as follows.

\begin{thm}\label{thm-equivalence}
Under the conditions of Theorem \ref{thm-equivalence-},  the TQFTs $\vert \cdot
\vert_\cc$ and $\tau_{\zz(\cc)}$ are isomorphic.
\end{thm}

 The rest of the paper is devoted to the proof of
these  two theorems.  The proof is based on the
 following key lemma.

\begin{lem}\label{le-didim}   Under the conditions of Theorem~\ref{thm-equivalence-},   the vector spaces  $\vert\Sigma  \vert_\cc$ and   $\tau_{\zz(\cc)}(\Sigma)$  associated   with any closed connected oriented surface $\Sigma$ have equal dimensions. \end{lem}

In the next two sections we  deduce  Theorems~\ref{thm-equivalence-} and~\ref{thm-equivalence} from
 Lemma~\ref{le-didim}. (Only the case $\Sigma=S^1\times S^1$ of this lemma is needed for the proof of Theorems~\ref{thm-equivalence-}.)   Lemma~\ref{le-didim}  will be proved in Sections~\ref{thecoendof} and~\ref{proofoflemma}.

\subsection{Corollaries}
Theorem~\ref{thm-equivalence-} allows us to clarify relationships between invariants of 3-manifolds derived from involutory Hopf algebras. Let $H$ be a finite-dimensional involutory Hopf algebra over an algebraically closed field $\kk$ such that the characteristic of $\kk$ does not divide $\dim(H)$. By a well-known theorem of Radford, $H$ is semisimple, so that the category of finite-dimensional left $H$\ti modules $\leftidx{_H}{\mathrm{mod}}{}$ is a spherical fusion category. The category of finite-dimensional left $D(H)$-modules $\leftidx{_{D(H)}}{\mathrm{mod}}{}$, where $D(H)$ is the Drinfeld double of $H$, is a modular category (see \cite{EG, Mu}). Denote by   $\mathrm{Ku}_H$ the Kuperberg invariant of 3-manifolds  derived from $H$, see \cite{Ku}.  Denote by $\mathrm{HKR}_{D(H)}$ the Hennings-Kauffman-Radford invariant of 3-manifolds  derived from $D(H)$, see \cite{He,KR}.
\begin{cor}
For any closed oriented 3-manifold $M$,
$$ \mathrm{Ku}_H(M)= \mathrm{HKR}_{D(H)}(M)=\dim(H)\, \tau_{\leftidx{_{D(H)}}{\mathrm{mod}}{}}(M)=\dim(H)\, |M|_{\leftidx{_H}{\mathrm{mod}}{}}.$$
\end{cor}
 \begin{proof} By \cite{BW2}, we have $ \mathrm{Ku}_H(M)= \dim(H)\, |M|_{\leftidx{_H}{\mathrm{mod}}{}}$. By \cite{Ke, KL}, we have   $    \mathrm{HKR}_{D(H)}(M)= \dim(H)\, \tau_{\leftidx{_{D(H)}}{\mathrm{mod}}{}}(M)$.
Therefore we need only to prove the third equality. This equality follows from Theorem~\ref{thm-equivalence-} and   the fact that the category $\leftidx{_{D(H)}}{\mathrm{mod}}{}$ is braided equivalent to $\zz(\leftidx{_H}{\mathrm{mod}}{})$, see  for example  \cite{Ka}.
 \end{proof}

We say that two fusion categories are \emph{equivalent} if their centers are braided equivalent. For example, two fusion categories weakly Morita equivalent in the sense of M\"uger \cite{Mu2} are equivalent in our sense. Theorem~\ref{thm-equivalence} implies:
\begin{cor}
Equivalent spherical fusion categories of non-zero dimension over an algebraically closed field give rise to isomorphic TQFTs.
\end{cor}

By a \emph{Hermitian fusion category} we mean a spherical fusion category $\cc$ over $\CC$ endowed with antilinear homomorphisms
$$
\{f \in\Hom_\cc(X,Y) \to \bar{f} \in \Hom_\cc(Y,X)\}_{X,Y \in \Ob(\cc)}
$$
such that
$\bar{\bar{f}}=f$, $\overline{gf}=\bar{f}\bar{g}$, $\overline{f \otimes g}=\bar{f} \otimes\bar{g}$, $\overline{\lev_X}=\rcoev_X$, and $\overline{\lcoev_X}=\rev_X$ for all morphisms $f,g$ in $\cc$ and all $X \in \Ob(\cc)$. Note that these axioms imply that $\overline{\id_X}=\id_X$, $\dim(X)\in\R$ for all $X\in\Ob(\cc)$ and $\tr(f)\in \R$ for all endomorphisms $f$ in~$\cc$.
A \emph{unitary fusion category} is a Hermitian fusion category $\cc$ such that $\tr(f\bar{f})> 0$ for any non-zero morphism $f$ in $\cc$. This condition is equivalent to the condition that $\dim(V) >0$ for all simple objects $V$ of $\cc$. Clearly, the dimension of a unitary fusion category is a positive real number.

\begin{cor}\label{cor-unitary}
The TQFT $\vert \cdot \vert_\cc$  associated with a unitary fusion category $\cc$ is unitary in the sense of \cite[Chapter III]{Tu1}. In particular $\vert -M \vert_\cc=\overline{\vert M \vert}_\cc$ for any closed oriented 3-manifold $M$.
\end{cor}

\begin{proof}
A half braiding $(A,\sigma)$ of $\cc$ is said to be \emph{unitary} if $\overline{\sigma_X}=\sigma_X^{-1}$ for all $X\in \Ob(\cc)$. The \emph{unitary center of $\cc$}, denoted $\zz^u(\cc)$, is the full subcategory of $\zz(\cc)$ formed by the unitary half braidings. The unitary center of $\cc$ is a braided subcategory of $\zz(\cc)$. The inclusion $\zz^u(\cc) \subset \zz(\cc)$ is a braided equivalence, see~\cite[Theorem~6.4]{Mu}. Therefore $\zz^u(\cc)$ is a modular category. The conjugation in $\cc$ induces a conjugation in $\zz^u(\cc)$ so that $\zz^u(\cc)$ becomes a unitary modular category. Hence the TQFT $\tau_{\zz^u(\cc)}$ is unitary. Now there are isomorphisms of TQFTs $\tau_{\zz^u(\cc)}\simeq \tau_{\zz(\cc)} \simeq  \vert \cdot
\vert_\cc$. The first one is induced by the braided equivalence $\zz^u(\cc) \simeq \zz(\cc)$ and the second one is given by Theorem~\ref{thm-equivalence}. The unitary structure of $\tau_{\zz^u(\cc)}$ is transported to $\vert \cdot
\vert_\cc$ via the isomorphism $\tau_{\zz^u(\cc)} \simeq  \vert \cdot
\vert_\cc$.
\end{proof}

From Corollary~\ref{cor-unitary}, Theorem~\ref{thm-equivalence}, and Theorem~11.5 of \cite{Tu1}, we deduce:
\begin{cor}
If $\cc$ is a unitary fusion category, then $\left|\vert M \vert_\cc\right|\leq (\dim(\cc))^{g(M)-1}$ for any closed oriented 3-manifold $M$, where $g(M)$ is the Heegaard genus of $M$.
\end{cor}

\section{State sum invariants of   links in 3-manifolds}\label{State sum invariants of   links}

Fix    a   spherical fusion category~$\cc$ over an algebraically
closed field~$\kk$ such that $\dim \cc\neq 0$.
  In this section,
we extend the state-sum invariant  of 3-cobordisms $\vert \cdot
\vert_\cc$  to 3-cobordisms with $\zz(\cc)$-colored   links
inside.

\subsection{Knotted graphs in $S^2$}\label{sect-knotted-graph}
 By a \emph{knotted graph} in $S^2$, we mean a graph immersed in $S^2$ such
that  the multiple points of the immersion are double transversal
crossings  of edges, and  at each crossing,
  one of the two intersecting edges is distinguished and said to
be \lq\lq lower" or \lq\lq under-passing", the second edge being
\lq\lq upper" or \lq\lq over-passing". Note that some of the
crossings may be   self-crossings of edges. A knotted graph may have
only a finite number of crossings, and   they all lie away from the
vertices of the graph.

   A knotted graph~$G$ in $S^2$ is $\cc$-colored if   all its edges   are oriented and endowed with an object
of $\cc$ or $\zz(\cc)$ (called the \emph{color} of the edge) so
that, at each crossing, the color of the over-passing edge is an
object of $\zz(\cc)$. In other words, the color of an edge $e$
of~$G$
  is an object of   $\zz(\cc)$ if $e$ passes at least once over other edges or over itself  and  is an object of $\zz(\cc)$ or $ \cc $ otherwise.
   By the \emph{quasi-color} of $e$, we mean the color of~$e$ if it  is in $\Ob (\cc)$ and  the underlying
    object of the color of~$e$ if it is in $\Ob (\zz(\cc))$. In all cases, the quasi-color  of $e$ is in $\Ob (\cc)$.
     For example, the $\cc$-colored graphs in
$S^2$ as in Section~\ref{sect-graph}  are   $\cc$-colored knotted
graphs   with no crossings.

Given two $\cc$-colored knotted graphs $G$ and $G'$ in $S^2$, an
\emph{isotopy} of~$G$ into~$G'$  is an isotopy of~$G$ into~$G'$ in
the class of $\cc$-colored knotted graphs in $S^2$ preserving the
vertices, the edges, the crossings, and   the orientation and the
color of the edges.

A vertex $v$ of a $\cc$-colored knotted graph~$G$ in~$S^2$
determines a cyclic $\cc$\ti set $ (E_v, c_v, \varepsilon_v)$ as in
Section~\ref{sect-graph}  with the only difference that the map
$c_v\co E_v \to \Ob(\cc)$ assigns to each half-edge $e\in E_v$ its
quasi-color. As in Section~\ref{sect-graph}, we set $H_v(G)=H(E_v )$
and $H(G)=\otimes_v \, H_v(G)$, where $v$ runs over all vertices
of~$G$.

The invariant $\inv_\cc$ of Section~\ref{sect-graph-S2} extends to
$\cc$-colored knotted graphs in $S^2$ as follows. Let $G$ be a
$\cc$-colored knotted graph in $S^2$. Pushing, if necessary, $G$
away from $
 \infty $, we obtain a $\cc$-colored knotted  graph in $\R^2$,  also denoted by $G$.  For
each vertex $v$   of~$G$, we choose a half-edge $e_v\in E_v$ and
isotope  $G$ near~$v$ so that   the half-edges incident to $v$   lie
above $v$ with respect to the second coordinate on $\R^2$ and $e_v$
is the leftmost of them. Pick any $\alpha_v \in H_v(G)$ and replace
$v$ by a box colored with $\tau^v_{e_v}(\alpha_v)$, where $\tau^v$
is the universal cone of $H_v(G)$ (see Figure~\ref{fig-alg-inv1}).
For each crossing $c$ of $G$, isotope~$G$ near~$c$ so that both
strands of~$G$ meeting at~$c$  are oriented downwards. If
$({{A}},\sigma)$ is the color of the over-passing strand and   $X$ the
quasi-color of the under-passing strand, then we replace~$c$ by a
box labeled by $\sigma_X$ provided the over-passing strand goes from
top-right to bottom-left and by $\sigma_X^{-1}$  provided the
over-passing strand goes from top-left to bottom-right. This
transforms $G$ into a planar diagram which determines, by the
Penrose calculus,  an element of $ \End_\cc(\un)=\kk$. This element
is denoted $\inv_\cc(G)(\otimes_v\alpha_v) $. By linear extension,
our procedure defines a vector $\inv_\cc(G)\in
H(G)^\star=\Hom_\kk(H(G),\kk)$.

Consider the local moves on $\cc$-colored knotted graphs in~$S^2$
shown in Figure~\ref{fig-knotted-graphs-moves}. It is understood
that   all strands are oriented (in an arbitrary way), and the
orientations of the strands are the same before and after the moves.
The colors of all edges are objects of $\cc$ and $\zz(\cc)$
preserved under the moves. As above, the colors of all over-passing
edges are objects of $\zz(\cc)$. For example, the second move of
Figure~\ref{fig-knotted-graphs-moves} is allowed only when the color
of the left strand is in $\zz(\cc)$.

\begin{figure}[h,t]
\begin{center}
 \rsdraw{.45}{.8}{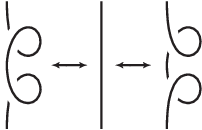}\; ,
\quad \quad
  \rsdraw{.45}{.8}{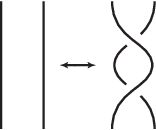}\; ,
\quad \quad
  \rsdraw{.45}{.8}{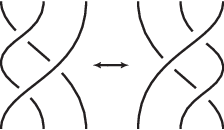}\\[1em]
  \rsdraw{.45}{.8}{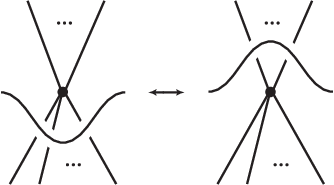}
 \end{center}
\caption{}
\label{fig-knotted-graphs-moves}
\end{figure}

\begin{lem}\label{lem-inv-knotted-graph}
The vector  $\inv_\cc (G) \in H(G)^\star$ is a well-defined isotopy
invariant of~$G$ preserved  under the moves in
Figure~\ref{fig-knotted-graphs-moves}.
\end{lem}

\begin{proof} Independence of $\inv_\cc (G)$ of the choice of the half-edges $e_v$
follows from the definition of $H_v(G)$. Invariance of $\inv_\cc
(G)$ under isotopies of~$G$ follows from the sphericity of $\cc$.

We claim that, as in Section~\ref{sect-graph-S2}, if a $\cc$-colored knotted graph $G' \subset S^2$ is obtained from a
$\cc$-colored knotted graph $G \subset S^2$ through   replacement of the
color of an edge $e$  by its dual and   simultaneous reversion of the
orientation of $e$, then the isomorphism $\psi_X\co X\to X^{**}$ of
Remark~\ref{rem-pivotal} (where $X\in \Ob (\cc)$ is the quasi-color of $e$) induces an isomorphism $\widehat{\psi} \co
H(G)\to H(G')$   such that    $(\widehat{\psi})^\star ( \inv_\cc(G') )=\inv_\cc(G)$. This claim is verified by  comparing the contributions   of the vertices and crossings of $G$, $G'$ to $
 \inv_\cc(G) $,   $\inv_\cc(G')$ respectively.
 Indeed, without loss of generality, we can assume that $G\subset \RR^2=S^2 \setminus \{\infty\}$ and $G$ is generic with respect to the   second coordinate on $\RR^2$.
 Consider a crossing $x$ of $G$ with over-passing edge $e_o$ and under-passing edge $e_u$. Applying if necessary an isotopy to $G$ in a neighborhood of $x$, we can assume that $e_o$ is directed from top-left to bottom-right. Let $(A,\sigma)\in\zz(\cc)$ be the color of $e_o$ and $X \in \Ob(\cc)$ be the quasi-color of $e_u$. There are four cases to consider depending on  the orientation of $e_u$  at $x$ (downwards or upwards)
and depending on whether we reverse the orientation  of $e_o$ or $e_u$.
 Assume that $e_u$ is directed downwards at $x$, so that the contribution of $x$ to $
 \inv_\cc(G) $ is
   $$
\psfrag{N}[Bl][Bl]{\scalebox{.9}{$(A,\sigma)$}}
\psfrag{X}[Bc][Bc]{\scalebox{.9}{$X$}}
\psfrag{M}[Bc][Bc]{\scalebox{.9}{$A$}}
\psfrag{u}[Bc][Bc]{\scalebox{.9}{$\sigma_{X}^{-1}$}}
\rsdraw{.45}{.9}{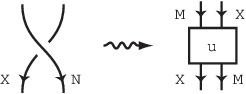}\;.
$$
If $G'$ is obtained by reversing   $e_u$ and replacing its color with the dual object, then   the contribution of $x$ to $
 \inv_\cc(G') $ is
  $$
\psfrag{N}[Bl][Bl]{\scalebox{.9}{$(A,\sigma)$}}
\psfrag{m}[Bc][Bc]{\scalebox{1.1}{$\simeq$}}
\psfrag{n}[Bc][Bc]{\scalebox{1.1}{$=$}}
\psfrag{X}[Bc][Bc]{\scalebox{.9}{$X$}}
\psfrag{Y}[Bc][Bc]{\scalebox{.9}{$X^*$}}
\psfrag{M}[Bc][Bc]{\scalebox{.9}{$A$}}
\psfrag{u}[Bc][Bc]{\scalebox{.9}{$\sigma_{X}^{-1}$}}
\psfrag{v}[Bc][Bc]{\scalebox{.9}{$\sigma_{X^*}$}}
\psfrag{g}[Bc][Bc]{\scalebox{.9}{$\psi_{X}$}}
\psfrag{h}[Bc][Bc]{\scalebox{.9}{$\psi_{X}^{-1}$}}
\rsdraw{.45}{.9}{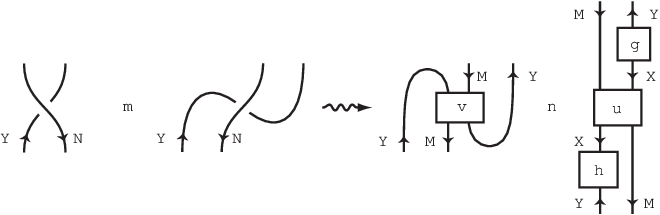}\;.
$$
The last equality follows from the definition of $\psi_X$ and   the formula $$ \sigma_{X}^{-1}=
(\rev_X \otimes \id_{{{A}} \otimes X})(\id_X \otimes \sigma_{X^*}
\otimes \id_X)(\id_{X \otimes {{A}}} \otimes \rcoev_X).$$ The latter formula
can be easily deduced  from the naturality of   $\sigma$ and
Formula~\eqref{axiom-half-braiding}. Similarly, if $G'$ is obtained by reversing   $e_u$ and replacing its color with the dual object, then the contributions of $x$ to $\inv_\cc(G)$ and $\inv_\cc(G') $ are obtained from each other through conjugation by  $\psi_A$ (to see this, one should use the expression for the dual of a  half braiding given in Section~\ref{sect-centerusual}). The case where $e_u$ is directed upwards is treated  similarly. Combining  these local computations, we obtain  the required claim.

Consider   the second move in Figure~\ref{fig-knotted-graphs-moves}.
Applying if necessary the property of  $\inv_\cc$ obtained in the previous
paragraph, we can reduce ourselves  to the case where both
strands are oriented downward. In this case the
required identity follows from the definition of $\inv_\cc (G)$.
Similarly, to prove the invariance of $\inv_\cc (G)$ under the
fourth  move, it suffices to consider the case  where all edges
incident to the vertex are oriented to the left. The required
identity follows from the   definitions and the properties of
$\sigma$. The invariance  of $\inv_\cc (G)$ under the third move can
be proved similarly treating the crossing point of the two lower
branches as a vertex. Finally, the first move expands as a
composition of the other moves through pushing the over-passing
branch to the left across the rest of the diagram and across the
point $\infty\in S^2 $.
\end{proof}

Note that Lemma~\ref{lem-calc-diag} extends mutatis mutandis to
$\cc$-colored knotted graphs in~$S^2$.

\subsection{Link diagrams}\label{sect-linkdiag}
Let $M$ be a compact
oriented  3-manifold.  Given a s-skeleton~$P$ of~$M$, that is, a s-skeleton of the pair $(M, \emptyset)$ in the sense of Section~\ref{sec-skes}, we can   present   links in~$M$
by diagrams in~$P $, see \cite[Chapter~IX]{Tu1}. (Similar
presentations can be defined on arbitrary skeletons of~$M$, but
s-skeletons will be sufficient for our aims.)   A \emph{link diagram} in $P$ is a finite set of loops in $P$ such
that:
\begin{enumerate}
\renewcommand{\labelenumi}{{\rm (\roman{enumi})}}
\item the loops do not meet the vertices of $P$  and
  meet the edges of~$P$ transversely;
\item the loops
 have only double transversal crossings and self-crossings lying
in $\Int (P)=P\setminus P^{(1)}$;
\item at each crossing point of the loops one of the two
intersecting branches is distinguished and said to be \lq\lq
lower", the second branch being \lq\lq upper".
\end{enumerate}

The {\it underlying $4$-valent graph}   of a link diagram is formed by the loops of the diagram with over/under-crossing data forgotten.
By abuse of notation, we shall usually denote a link diagram and its  underlying graph by the same symbol.

 Each link diagram  $d$ in~$P$ determines a link $ \ell_d\subset M$ as follows. The
orientations of~$M$ and $P$ determine a distinguished normal direction on
$\Int(P)$ in~$M$. Pushing  slightly all upper branches  of~$d$ in
$M\setminus P$ in this   direction, we transform~$d$ into $\ell_d$.
The   underlying loops of~$d$ have a  well-defined  normal line
bundle ${b}_P (d)$
 in~$P$; this bundle is  defined in the points  of $d\cap P^{(1)}$ since, in a
neighborhood of  any such   point, $d$ traverses two regions of $P$
locally forming a 2-disk. The bundle ${b}_P (d)$ induces a line
subbundle ${b}_d $ of the (2-dimensional) normal bundle of $\ell_d$
in~$M$.

An \emph{enriched link diagram}  in $P$ is a link diagram  in $P$
whose loops are equipped with  half-integers called
\emph{pre-twists}. The pre-twist of a loop  $L$ is required to
belong  to $\Bbb Z$  if the normal line bundle of $L$ in $P$ is
trivial and  to $\frac{1}{2}+\Bbb Z$ otherwise.  The pre-twists determine
a framing of  $\ell_d\subset M$ as follows. Twist  ${b}_d$ around
each component of $\ell_d$ as many times as the pre-twist of the
corresponding loop. (The positive direction of the twist is
determined by the orientation of $M$. For instance, a pre-twist  of
$\frac{1}{2}$ gives rise to a positive half-twist of ${b}_d$.) This   produces
a trivial normal line bundle on $\ell_d$. Its  non-zero sections
yield a framing of $\ell_d$.

It is easy to see  that every framed link in~$M$ may be represented
by an enriched link diagram  in~$P$. Consider the moves  $\Omega_1,\dots, \Omega_8$  on enriched link diagrams shown in
  Figure~\ref{fig-link-moves} where   the orientation of~$M$ corresponds to
  the right-handed orientation in ${\RR}^3$. The moves $\Omega_1,\Omega_2, \Omega_3$ proceed in
$\Int(P)$, and $\Omega_4 ,\dots,  \Omega_8$ proceed in a
neighborhood of~$P^{(1)}$. In Figure~\ref{fig-link-moves}  the link
diagrams on~$P$
 are drawn in red bold  in order to distinguish them from  the  edges
of~$P$. The move $\Omega_1$ decreases the pre-twist by~$1$ and
$\Omega_4$ increases the pre-twist by~$\frac{1}{2}$; the other moves do not change the pre-twists.      The move $\Omega_6$ has four versions; in  the one in
  Figure~\ref{fig-link-moves}   we assume that the orientations of  the horizontal regions   are compatible, i.e., are induced by an orientation of the horizontal plane. If these orientations are  incompatible, then the picture must be modified:
  the overcrossing on the left (or on the right) should be replaced with an undercrossing. The third and fourth versions of $\Omega_6$ are obtained from the first two by changing both overcrossings to undercrossings.  The move $\Omega_7$ has two versions; both apply only when the  left horizontal region is oriented counterclockwise.    By \cite{Tu1}, two  enriched link diagrams  in~$P$
represent isotopic framed links if  and only if  these diagrams may be
related  by a finite sequence of moves $\Omega_1^{\pm 1},\dots, \Omega_8^{\pm 1}$
  and  ambient
isotopies in $P$.

Oriented  framed  links in $M$ can be similarly presented by
oriented link diagrams in~$P$. These are the (enriched) link
diagrams as above formed by oriented loops. Two such diagrams
present isotopic  oriented framed links   if and only if they
may be related by the moves   in Figure~\ref{fig-link-moves} where
all strands are oriented   (the orientations of the strands must be
the same before and after the moves).

\begin{figure}
\begin{center}
 \psfrag{T}[Bc][Bc]{\scalebox{1}{$\Omega_1$}}
 \psfrag{A}[Bc][Bc]{\scalebox{1}{\textcolor{red}{$\mathbf{-1}$}}}
 \rsdraw{.45}{.8}{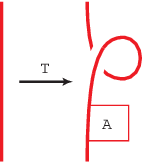}\; ,
\quad \quad
 \psfrag{T}[Bc][Bc]{\scalebox{1}{$\Omega_2$}}
  \rsdraw{.45}{.8}{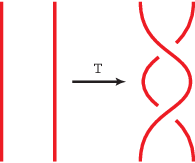}\; ,
\quad \quad
 \psfrag{T}[Bc][Bc]{\scalebox{1}{$\Omega_3$}}
  \rsdraw{.45}{.8}{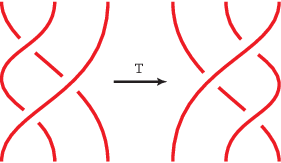}\;,\\[1.5em]
 \psfrag{T}[Bc][Bc]{\scalebox{1}{$\Omega_4$}}
 \psfrag{h}[Bc][Bc]{\scalebox{.6}{\textcolor{red}{$\displaystyle  \mathbf{\frac{1}{2}}$}}}
 \rsdraw{.45}{.8}{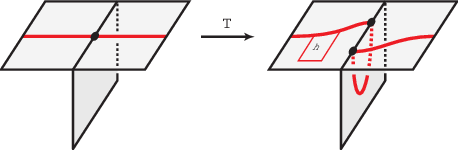}
 \hfill \psfrag{T}[Bc][Bc]{\scalebox{1}{$\Omega_5$}}
 \rsdraw{.45}{.8}{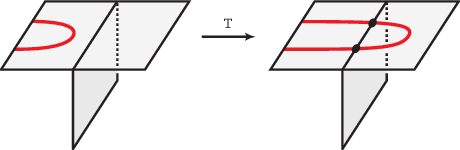}\\[1.5em]
 \psfrag{T}[Bc][Bc]{\scalebox{1}{$\Omega_6$}}
 \rsdraw{.45}{.8}{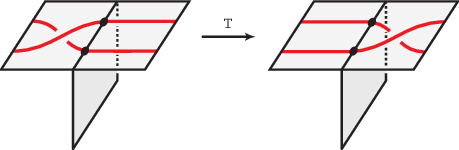}
 \hfill\psfrag{T}[Bc][Bc]{\scalebox{1}{$\Omega_8$}}
 \rsdraw{.45}{.8}{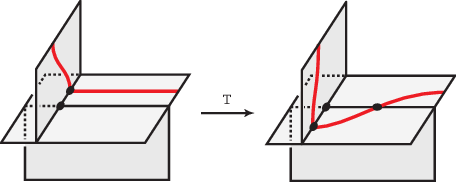}\\[1.5em]
 \psfrag{T}[Bc][Bc]{\scalebox{1}{$\Omega_{7.1}$}}
 \rsdraw{.45}{.8}{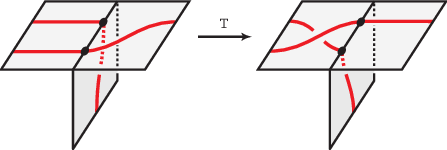}
 \hfill\psfrag{T}[Bc][Bc]{\scalebox{1}{$\Omega_{7.2}$}}
 \rsdraw{.45}{.8}{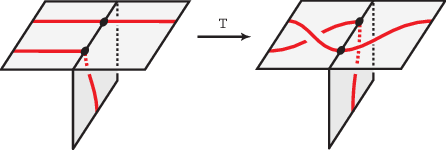}
 \end{center}
\caption{Moves $\Omega_1 - \Omega_8$ on link diagrams}
\label{fig-link-moves}
\end{figure}

\subsection{Invariants of links} \label{sect-invariants-links} 
A component $K$ of a link $L$ in a 3-manifold $M$ is said to be {\it trivial} if there is an embedded 2-disk  $D\subset M$ such that $K=\partial D=D\cap L$.
Consider a pair $(M,L)$, where $M$ is a compact oriented   3-manifold
 and $L\subset \Int(M)$ is an oriented framed link  whose all components  are non-trivial and colored by {simple} objects of $\zz(\cc)$.
We   define an invariant $|M,L|_\cc \in \kk$ as follows.   Fix a
 (finite)
representative set~$I$ of simple objects of~$\cc$.
Pick a s-skeleton $P$ of $M$ and  an oriented  enriched link diagram $d$ in
$P$ representing~$L$. The set ${\widetilde d}=P^{(1)}\cup d$ is a graph embedded in $P$
with edges   contained either in $P^{(1)}$ or  in~$d$  and vertices of three types: the vertices of $P^{(1)}$,
 the  points of $P^{(1)}\cap d$, and the crossings of~$d$. The assumption that all components of $L$ are non-trivial ensures that every  component of ${\widetilde d}$ contains at least one vertex, so that ${\widetilde d}$ is   a graph in the sense of Section~\ref{sect-graph}. Cutting out $P$ along ${\widetilde d}$, we
obtain a compact oriented surface whose components are disks and disks with holes. These components are called the
{\it regions of $d$}. The set of regions of $d$ is denoted by
$\Reg(d)$.

Pick a map $c \co \Reg(d) \to I$.   For any oriented edge $e$ of the
graph ${\widetilde d}$, we define  a \kt module $H_c(e)$. If
$e\subset P^{(1)}$, then  $H_c(e)=H(P_e)$ as in
Section~\ref{sec-computat}.  If $e\subset d$, then there are two
regions $r_-, r_+\in \Reg(d)$ adjacent to $e$. We choose   notation
so that the orientation of $r_+$ (resp.\ $r_-$) induces on $e$ the
  orientation compatible with that of $e$ (resp.\ opposite to that of $e$):
\begin{center}
\medskip
 \psfrag{M}[Bl][Bl]{\scalebox{.9}{\textcolor{red}{${{A}}$}}}
 \psfrag{r}[Bc][Bc]{\scalebox{.9}{$r_+$}}
 \psfrag{v}[Bc][Bc]{\scalebox{.9}{$r_-$}}
 \psfrag{l}[Bc][Bc]{\scalebox{.9}{$e$}}
 \rsdraw{.45}{.9}{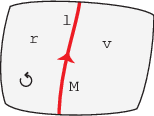}.
\medskip
\end{center}
Let ${{A}}  \in \Ob (\cc)$ be the    quasi-color of the component of
$L$ containing $e$.  We turn  the 3-element  set $P_e=\{r_- , e ,
r_+\}$   into a cyclic
 $\cc$-set by providing it with the
 cyclic order $r_- < e < r_+ < r_-$ and with the map  to $\Ob(\cc) \times \{+,-\}$ carrying
  $r_\pm $ to $(c(r_\pm), \pm)$ and $e$ to  $({{A}},\varepsilon )$, where $\varepsilon=+$ if the orientations of $e$ and $d$ are compatible and
  $\varepsilon=-$ otherwise.
Set  $$H_c(e)=H(P_e) \cong \Hom_\cc(\un,c(r_-)^*\otimes {{A}}^\varepsilon \otimes
c(r_+) ).$$ It is clear that in  both cases  $P_{e^\opp}=(P_e)^\opp$. This
induces a duality between
 the modules
 $H_c(e)$, $H_c(e^\opp)$ and a
 contraction  $ \ast_e \co H_c(e)^\star \otimes H_c(e^\opp)^\star \to\kk$.

We now associate    to every  vertex $x$ of ${\widetilde d}$ a
certain   vector $|x|_c$. We  distinguish three cases. If $x$ is a
vertex of $P^{(1)}$, then  as in Section~\ref{sec-computat},  the
link of $x$ determines a $\cc$-colored  graph $\Gamma_x \subset
S^2$. Set
  $|x|_c=\inv_\cc (\Gamma_x) \in
H_c(\Gamma_x)^*$.

If $x$ is a crossing of $d$, then a
neighborhood of $x$ in~$P$  looks as follows:
\begin{equation*}
 \psfrag{i}[Bc][Bc]{\scalebox{.9}{$i$}}
 \psfrag{j}[Bc][Bc]{\scalebox{.9}{$j$}}
 \psfrag{k}[Bc][Bc]{\scalebox{.9}{$k$}}
 \psfrag{M}[Br][Br]{\scalebox{.9}{\textcolor{red}{$J$}}}
 \psfrag{N}[Bl][Bl]{\scalebox{.9}{\textcolor{red}{$J'$}}}
 \psfrag{l}[Bc][Bc]{\scalebox{.9}{$l$}}
 \rsdraw{.45}{.9}{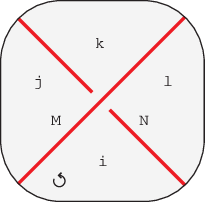}
\end{equation*}
Here $J$ and $J'$ are the colors of the strands of $d$ meeting at
$x$, and $i,j,k,l \in I$ are the $c$-colors of the regions in
$\Reg(d)$ adjacent to $x$. Let $\Gamma_x$ be the following
$\cc$-colored knotted graph in   $\RR^2\subset S^2$:
\begin{equation*}
 \psfrag{i}[Bc][Bc]{\scalebox{.9}{$i$}}
 \psfrag{j}[Bc][Bc]{\scalebox{.9}{$j$}}
 \psfrag{k}[Bc][Bc]{\scalebox{.9}{$k$}}
 \psfrag{M}[Br][Br]{\scalebox{.9}{\textcolor{red}{$J$}}}
 \psfrag{N}[Bl][Bl]{\scalebox{.9}{\textcolor{red}{$J'$}}}
 \psfrag{l}[Bc][Bc]{\scalebox{.9}{$l$}}
  \Gamma_x= \,\rsdraw{.45}{.9}{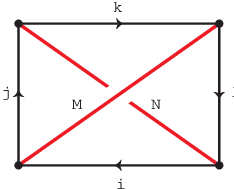}
\end{equation*}
The orientation  (not shown in the picture) of the diagonals   is
induced by that of $d$. Section~\ref{sect-knotted-graph} yields a
tensor $|x|_c=\inv_\cc (\Gamma_x) \in H_c(\Gamma_x)^*$.

The  definition of $|x|_c$ in the remaining  case  $x \in P^{(1)}
\cap d$  uses  the assumption  that the colors of the components of
$L$ are simple objects. Since $\kk$ is an algebraically closed
field, we can choose for each component of $L$   a square root
$\nu_{J} \in \kk$ of the twist scalar $v_{J}\in \kk$ of the color
$J$ of this component, see Section~\ref{sec-modulcat}. (At the end,
the invariant $|M,L|_\cc \in \kk$ will be independent of this choice.)  A  neighborhood of $x \in P^{(1)} \cap d$
in $P$ looks as follows (as usual, the orientation of~$M$ is right-handed):
\begin{equation*}
 \psfrag{r}[Bc][Bc]{\scalebox{.9}{$r^-$}}
 \psfrag{c}[Bc][Bc]{\scalebox{.9}{$r^+$}}
 \psfrag{n}[Bc][Bc]{\scalebox{.9}{$r$}}
 \psfrag{i}[Bc][Bc]{\scalebox{.9}{$i$}}
 \psfrag{j}[Bc][Bc]{\scalebox{.9}{$j$}}
 \psfrag{k}[Bc][Bc]{\scalebox{.9}{$k$}}
 \psfrag{M}[Bc][Bc]{\scalebox{.9}{$M$}}
 \psfrag{N}[Bc][Bc]{\scalebox{.9}{\textcolor{red}{$J$}}}
 \psfrag{l}[Bc][Bc]{\scalebox{.9}{$l$}}
 \psfrag{x}[Bc][Bc]{\scalebox{.9}{$x$}}
 \psfrag{m}[Bc][Bc]{\scalebox{.9}{$m$}}
 \rsdraw{.45}{.9}{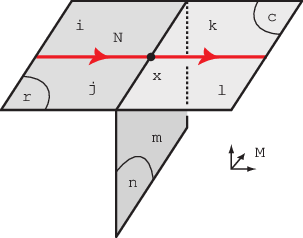}\,.
\end{equation*}
The skeleton $P$  has  3 regions (possibly coinciding) adjacent to
$x$.  We denote them  $r,r^-, r^+$  so that the strand of $d$ at $x$
goes from $r^-$ to $r^+$. Let $J$ be the color of this strand. The
diagram $d$ has  5 regions  adjacent to $x$.  We denote their
$c$-colors  by $i,j,k,l,m$ as in the picture. Let $e_0$ be a tangent vector at $x$
directed inside $r$ and let $(e^{\pm}_1,e^{\pm}
_2)$ be a positive (i.e., positively oriented) basis of the tangent
space of $r^{\pm}$ at $x$.  Consider the basis $(e_0,e^{\pm}_1,e^{\pm}_2)$
   of the tangent space of~$M$ at $x$. Set $\varepsilon^{\pm}_x=1$ if this basis is  positive and $\varepsilon^{\pm}_x=-1$ if it is negative.
Let $\Gamma_x$ be the following $\cc$-colored knotted graph in $\RR^2\subset S^2$:
\begin{equation*}
 \psfrag{i}[Bc][Bc]{\scalebox{.9}{$i$}}
 \psfrag{j}[Bc][Bc]{\scalebox{.9}{$j$}}
 \psfrag{k}[Bc][Bc]{\scalebox{.9}{$k$}}
 \psfrag{l}[Bc][Bc]{\scalebox{.9}{$l$}}
 \psfrag{m}[Bc][Bc]{\scalebox{.9}{$m$}}
 \psfrag{M}[Bc][Bc]{\scalebox{.9}{\textcolor{red}{$J$}}}
 \psfrag{N}[Bl][Bl]{\scalebox{.9}{\textcolor{red}{$J'$}}}
 \psfrag{1}[Bc][Bc]{\scalebox{1}{$1$}}
 \psfrag{2}[Bc][Bc]{\scalebox{1}{$2$}}
  \Gamma_x= \rsdraw{.45}{.9}{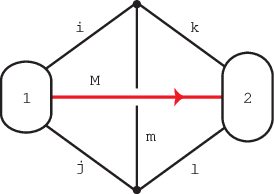}
\end{equation*}
where
\begin{equation*}
 \psfrag{i}[Bc][Bc]{\scalebox{.9}{$i$}}
 \psfrag{j}[Bc][Bc]{\scalebox{.9}{$j$}}
 \psfrag{k}[Bc][Bc]{\scalebox{.9}{$k$}}
 \psfrag{l}[Bc][Bc]{\scalebox{.9}{$l$}}
 \psfrag{m}[Bc][Bc]{\scalebox{.9}{$m$}}
 \psfrag{M}[Bc][Bc]{\scalebox{.9}{\textcolor{red}{$J$}}}
 \psfrag{N}[Bl][Bl]{\scalebox{.9}{\textcolor{red}{$J'$}}}
 \psfrag{1}[Bc][Bc]{\scalebox{1}{$1$}}
 \psfrag{2}[Bc][Bc]{\scalebox{1}{$2$}}
  \rsdraw{.45}{.9}{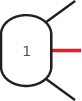}=\left\{\begin{array}{ll} \rsdraw{.45}{.9}{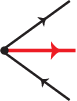} & \text{if $\varepsilon^{{-}}_x=-1$,}\\ \rsdraw{.45}{.9}{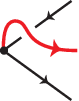} & \text{if $\varepsilon^{{-}}_x=1$,}\end{array}\right. \quad 
  \rsdraw{.45}{.9}{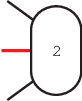}=\left\{\begin{array}{ll} \rsdraw{.45}{.9}{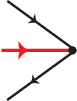} & \text{if $\varepsilon^{{+}}_x=-1$,}\\ \rsdraw{.45}{.9}{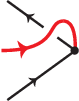} & \text{if $\varepsilon^{{+}}_x=1$.}\end{array}\right.
\end{equation*}
The latter pictures determine the orientations of the edges  of $\Gamma_x$ colored by $i,j,k,l$.  The  edge of $\Gamma_x$ colored by $m$   is directed
downward  if  the orientation of $r$ followed by  that of
$d$ at $x$ yields the positive orientation of $M$ and upward otherwise. Set
\begin{equation*}
\varepsilon_x=\frac{\varepsilon^{{-}}_x-\varepsilon^{{+}}_x}{2} \in \{-1,0,1\}
 \quad\quad \text{and} \quad \quad |x|_c=\nu_{J}^{\varepsilon_x}\, \inv_\cc (\Gamma_x) \in H_c(\Gamma_x)^\star.
\end{equation*}

For any vertex $x$ of $\widetilde d$, we have  $H_c(\Gamma_x)=
\otimes_e\, H_c(e)$, where
 $e$ runs over the edges of ${\widetilde d}$ incident to $x$ and oriented away from $x$.  The tensor product $\otimes_x \,|x|_c$ over all vertices~$x$ of~${\widetilde d}$
is a vector in
 $\otimes_{e } \, H_c(e )^\star $, where   $e$ runs over all oriented edges of ${\widetilde d}$. Set $\ast_P=\otimes_{e} \, \ast_e \colon \otimes_{e } \, H_c(e )^\star\to \kk $.

  Let   $L_1,..., L_N$ be the components of $L$. Let $J_q \in \Ob (\zz(\cc))$ be the color of $L_q$. Recall
    the distinguished square root $\nu_{J_q}\in \kk$
   of the twist scalar $v_{J_q}$ of $ {J_q}$. Let $n_q\in \frac{1}{2} \ZZ$ be the pre-twist of   the loop of $d$ representing  $L_q$, where $q=1,...,N$. Set
\begin{equation*}
|M,L|_\cc=(\dim (\cc))^{-\vert P\vert}   \prod_{q=1}^N \nu_{J_q}^{2n_q}     \sum_{c}   \left ( \prod_{r \in \Reg(d)} \! (\dim c(r))^{\chi(r)} \right )
  {\ast}_P ( \otimes_x \,|x|_c) \in \kk,
  \end{equation*}
  where   $\vert P\vert$ is the number of components of $M\setminus P$,
  $c$ runs over all maps $ \Reg(d) \to I$, and $\chi(r)$ is the
Euler characteristic of $r$.

\begin{lem}\label{lem-state-3manNEW}
$|M,L|_\cc$ is a topological invariant of the pair $(M,L)$
independent of  the choice  of~$I$ and of  the choice  of   square
roots    of the twist scalars.
\end{lem}

\begin{proof}
The independence  of~$I$ follows from the naturality of $\inv_{\cc}$
and of the contraction homomorphisms. We   prove the independence of
the choice of ${\nu_{J_q}}$. The term of $|M,L|_\cc$ determined by a map $
\Reg(d) \to I$ is a  product of an expression independent of~${\nu_{J_q}}$
and $ \nu_{J_q}^{2n_q+s_q}$ for $s_q=\sum_x \varepsilon_x$, where $x$ runs
over the intersections of the $q$-th component of $d$ with
$P^{(1)}$. It suffices to show that $2n_q+s_q\in 2\ZZ$. Observe that
$\varepsilon_x= 0 $ if the orientations of the regions $r_-$ and
$r_+$ at $x$ are compatible, and $\varepsilon_x= \pm 1 $ otherwise.
Therefore $s_q\in 2\ZZ$ if and only if the normal bundle of the $i$-th
component of $d$ in $P$ is trivial.  By the definition of a
pre-twist, the latter condition holds if and only if $ n_q \in
 \ZZ$. Therefore, $2n_q+s_q\in 2\ZZ$ in all cases.

Recall that any two s-skeletons of $M$ can be related by a sequence
of Matveev-Piergallini moves. Using   $\Omega_5^{\pm 1}, \Omega_6$,
and $\Omega_8$ we can   deform any     diagram of~$L$ on~$P$ away
from the place where an   MP-move is performed on $P$. Then the
same arguments as in the proof of
Theorems~\ref{thm-state-3man-smipl} and~\ref{thm-state-3man}  show
that $|M,L|_\cc$ is invariant under the MP-move in question.
Therefore we need only to verify that $|M,L|_\cc$ is invariant under
the moves $\Omega_1,\dots,\Omega_8$.

Without loss of generality, we can assume that in the pictures of $\Omega_1, \Omega_2, \Omega_3$ the orientation of the ambient region corresponds to the counterclockwise orientation of the plane of the picture.  For any orientation of the red-bold strand, colored by a {simple}
object $J$ of $\zz(\cc)$,
\begin{center}
 \psfrag{M}[Bl][Bl]{\scalebox{.9}{\textcolor{red}{$J$}}}
  \psfrag{N}[Br][Br]{\scalebox{.9}{\textcolor{red}{$J$}}}
 \psfrag{x}[Bc][Bc]{\scalebox{.9}{${\letterf}$}}
 \psfrag{i}[Bc][Bc]{\scalebox{.9}{$i$}}
 \psfrag{j}[Bc][Bc]{\scalebox{.9}{$j$}}
 \rsdraw{.45}{.9}{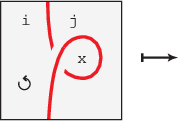}$\ds\;\sum_{{\letterf}\in I}\dim({\letterf})$
  \, \rsdraw{.45}{.9}{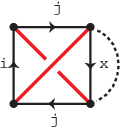} $\;=\;$ \rsdraw{.45}{.9}{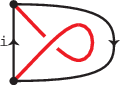}$\;=\,\ds v_J$\, \rsdraw{.45}{.9}{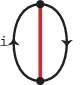}\,.
\end{center}
The  arrow   indicates that we compute the contribution of the picture (the curl)  to the state sum.  The dotted arc indicates   the tensor contraction of the   vector spaces corresponding to the endpoints of the arc. The   equalities follow  respectively from Lemma \ref{lem-calc-diag}.c and  the definition of the twist scalar $v_J$ (if the  red-bold strand is oriented upward,  one should also use the equality $v_{J^*}=v_J$). Taking into account the normalization factor
$\prod_{q}  \nu_{q}^{2n_q}$  and the   additional pre-twist $-1$ introduced by $\Omega_1$, we conclude that       $|M,L|_\cc$ is invariant under $\Omega_1$.

For any   orientations of the red-bold strands (colored by
{simple} objects of $\zz(\cc)$),
\begin{center}
 \psfrag{M}[Br][Br]{\scalebox{.9}{\textcolor{red}{$J$}}}
  \psfrag{N}[Bl][Bl]{\scalebox{.9}{\textcolor{red}{$J'$}}}
 \psfrag{x}[Bc][Bc]{\scalebox{.9}{${\letterf}$}}
 \psfrag{i}[Bc][Bc]{\scalebox{.9}{$i$}}
 \psfrag{j}[Bc][Bc]{\scalebox{.9}{$j$}}
 \psfrag{k}[Bc][Bc]{\scalebox{.9}{$k$}}
 \psfrag{l}[Bl][Bl]{\scalebox{.9}{$l$}}
 \rsdraw{.45}{.9}{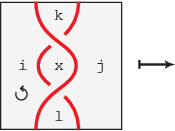}$\ds\sum_{{\letterf}\in I}\dim({\letterf})$ \; \rsdraw{.45}{.9}{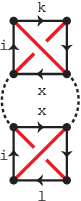}$\ds\; =\sum_{{\letterf}\in I}\dim({\letterf})$ \, \rsdraw{.45}{.9}{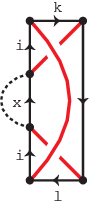}$\;=\;$ \rsdraw{.45}{.9}{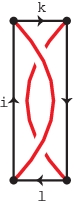}\\
 $\;=\;$ \rsdraw{.45}{.9}{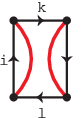}$\;=\;\ds\delta_{k,l}$ \, \psfrag{j}[Bc][Bc]{\scalebox{.9}{$k$}}\rsdraw{.45}{.9}{proof-omega1d}  \psfrag{i}[Bc][Bc]{\scalebox{.9}{$k$}} \psfrag{j}[Bc][Bc]{\scalebox{.9}{$j$}}\,\rsdraw{.45}{.9}{proof-omega1d}\;\;\psfrag{i}[Bc][Bc]{\scalebox{.9}{$i$}} \rsdraw{.45}{.9}{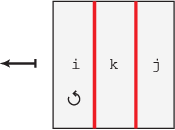}\;.
\end{center}
The first two equalities follow  respectively from Claims (d) and (c) of Lemma \ref{lem-calc-diag}.
The third equality follows from
Lemma~\ref{lem-inv-knotted-graph}. We conclude that   $|M,L|_\cc$ is
invariant under $\Omega_2$.  The invariance under $\Omega_3$ is
verified similarly.

Below we verify the invariance of  $|M,L|_\cc$ under the moves $\Omega_4, \dots
\Omega_8$ for a certain  orientation  of the link diagram. The proof of the invariance for the opposite orientation of a component of the  diagram can be obtained by repeating exactly the same arguments but using everywhere the opposite orientation of the relevant edges of the $\cc$-colored knotted graphs. In particular, the tensors associated with all vertices are represented by the same graphs with opposite orientation of the appropriate red-bold edges.  For the tensors associated with the crossings, this follows directly from the definitions. For the tensors associated with the vertices of $P^{(1)}\cap d$, we use that
\begin{center}
 \psfrag{M}[Br][Br]{\scalebox{.9}{\textcolor{red}{$J$}}}
 \rsdraw{.45}{.9}{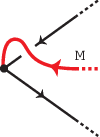} $\quad = \quad$ \rsdraw{.45}{.9}{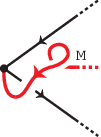} $\quad= \; \nu_J^{-2}\quad $ \rsdraw{.45}{.9}{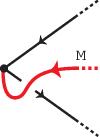}\;.
\end{center}
and
\begin{center}
 \psfrag{M}[Br][Br]{\scalebox{.9}{\textcolor{red}{$J$}}}
 \rsdraw{.45}{.9}{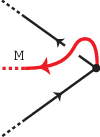} $\quad = \quad$ \rsdraw{.45}{.9}{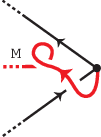} $\quad= \; \nu_J^2\quad $ \rsdraw{.45}{.9}{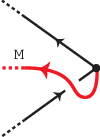}\;.
\end{center}
For example, we have
\begin{center}
 \psfrag{M}[Bc][Bc]{\scalebox{.9}{\textcolor{red}{$J$}}}
 \psfrag{i}[Bc][Bc]{\scalebox{.9}{$i$}}
 \psfrag{j}[Bc][Bc]{\scalebox{.9}{$j$}}
 \psfrag{k}[Bc][Bc]{\scalebox{.9}{$k$}}
 \psfrag{l}[Bc][Bc]{\scalebox{.9}{$l$}}
 \psfrag{m}[Bc][Bc]{\scalebox{.9}{$m$}}
 \rsdraw{.45}{.9}{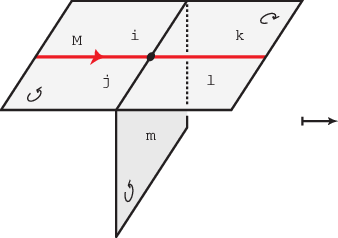}$\quad \nu_J^{\frac{-1-1}{2}}\, $ \rsdraw{.45}{.9}{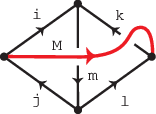}
\end{center}
and
\begin{center}
 \psfrag{M}[Bc][Bc]{\scalebox{.9}{\textcolor{red}{$J$}}}
  \psfrag{X}[Br][Br]{\scalebox{1.11}{$\nu_J^{\frac{1+1}{2}}$}}
   \psfrag{Y}[Br][Br]{\scalebox{1.11}{$=\, \nu_J^{-1}$}}
 \psfrag{i}[Bc][Bc]{\scalebox{.9}{$i$}}
 \psfrag{j}[Bc][Bc]{\scalebox{.9}{$j$}}
 \psfrag{k}[Bc][Bc]{\scalebox{.9}{$k$}}
 \psfrag{l}[Bc][Bc]{\scalebox{.9}{$l$}}
 \psfrag{m}[Bc][Bc]{\scalebox{.9}{$m$}}
 \rsdraw{.45}{.9}{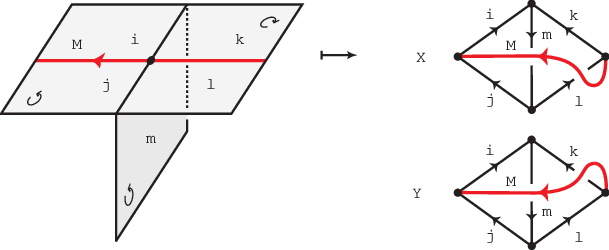}\;.
\end{center}
The last graph differs from the graph obtained in the previous picture only by  the orientation of the red-bold edge.

Let us now verify the invariance of $|M,L|_\cc$ under the moves $\Omega_4-\Omega_8$ for a certain orientation of the s-skeleton and of the link diagram. We begin with $\Omega_4$, $\Omega_5$, and $\Omega_8$. In these computations,   the (unique) red-bold strand is colored with a {simple} object $J$ of
$\zz(\cc)$. We have
\begin{center}
 \psfrag{M}[Bc][Bc]{\scalebox{.9}{\textcolor{red}{$J$}}}
 \psfrag{x}[Bc][Bc]{\scalebox{.9}{${\letterf}$}}
 \psfrag{i}[Bc][Bc]{\scalebox{.9}{$i$}}
 \psfrag{j}[Bc][Bc]{\scalebox{.9}{$j$}}
 \psfrag{k}[Bc][Bc]{\scalebox{.9}{$k$}}
 \psfrag{l}[Bc][Bc]{\scalebox{.9}{$l$}}
 \psfrag{m}[Bc][Bc]{\scalebox{.9}{$m$}}
 \rsdraw{.45}{.9}{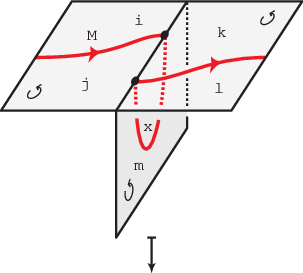}\\
 $\ds\nu_{J}^{\frac{-1-1}{2}}\,\nu_{J}^{\frac{1-1}{2}}\;\sum_{{\letterf}\in I}\dim({\letterf})\; $  \rsdraw{.45}{.9}{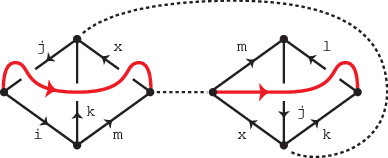}\\
 $\ds=\,\nu_{J}^{-1}\;\sum_{{\letterf}\in I}\dim({\letterf})\; $  \rsdraw{.45}{.9}{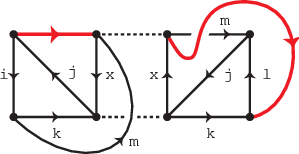}\\
 $\ds=\,\nu_{J}^{-1}\;\sum_{{\letterf}\in I}\dim({\letterf})\; $  \rsdraw{.45}{.9}{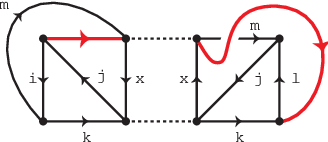}\\
 $\ds=\,\nu_{J}^{-1}\;$  \rsdraw{.45}{.9}{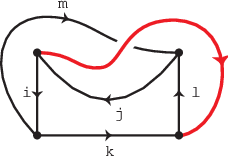} $\ds=\,\nu_{J}^{-1}\;$  \rsdraw{.45}{.9}{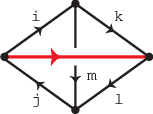}\\
\end{center}
 All these equalities  except the penultimate one  follow  from
Lemma~\ref{lem-inv-knotted-graph}; the penultimate equality   follows from  Lemma \ref{lem-calc-diag}.
 By definition,
\begin{center}
 \psfrag{M}[Bc][Bc]{\scalebox{.9}{\textcolor{red}{$J$}}}
 \psfrag{x}[Bc][Bc]{\scalebox{.9}{${\letterf}$}}
 \psfrag{i}[Bc][Bc]{\scalebox{.9}{$i$}}
 \psfrag{j}[Bc][Bc]{\scalebox{.9}{$j$}}
 \psfrag{k}[Bc][Bc]{\scalebox{.9}{$k$}}
 \psfrag{l}[Bc][Bc]{\scalebox{.9}{$l$}}
 \psfrag{m}[Bc][Bc]{\scalebox{.9}{$m$}}
 \rsdraw{.45}{.9}{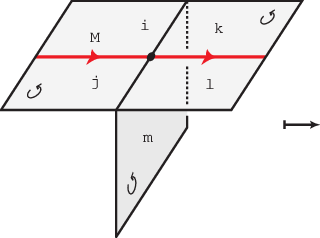}
 $\ds\;\nu_{J}^{\frac{-1+1}{2}}\; $  \rsdraw{.45}{.9}{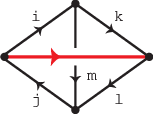}\;.
\end{center}
Taking into account the normalization factor
$\prod_{q}  \nu_{q}^{2n_q}$  and the   additional pre-twist $\frac{1}{2}$ introduced by $\Omega_4$, we conclude that       $|M,L|_\cc$ is invariant under  $\Omega_4$.

Next, we have
\begin{center}
 \psfrag{M}[Bc][Bc]{\scalebox{.9}{\textcolor{red}{$J$}}}
 \psfrag{x}[Bc][Bc]{\scalebox{.9}{${\letterf}$}}
 \psfrag{i}[Bc][Bc]{\scalebox{.9}{$i$}}
 \psfrag{j}[Bc][Bc]{\scalebox{.9}{$j$}}
 \psfrag{k}[Bc][Bc]{\scalebox{.9}{$k$}}
 \psfrag{l}[Bc][Bc]{\scalebox{.9}{$l$}}
 \psfrag{m}[Bc][Bc]{\scalebox{.9}{$m$}}
 \rsdraw{.45}{.9}{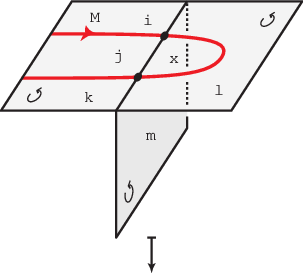}\\[.4em]
 $\ds\nu_{J}^{\frac{1-1}{2}}\,\nu_{J}^{\frac{1-1}{2}}\;\sum_{{\letterf}\in I}\dim({\letterf})$  \rsdraw{.45}{.9}{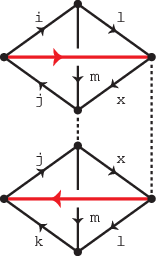} $\;=\;$ \rsdraw{.45}{.9}{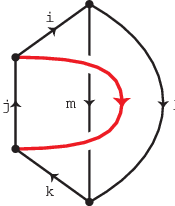} \\
 $\;=\;$ \rsdraw{.45}{.9}{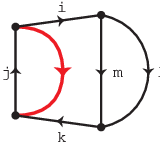} $\;=\; \delta_{i,k}\;$ \rsdraw{.45}{.9}{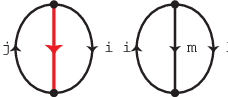} \;.
\end{center}
The right-hand side is the contribution to the state sum of the following piece of the diagram:
\begin{center}
 \psfrag{M}[Bc][Bc]{\scalebox{.9}{\textcolor{red}{$J$}}}
 \psfrag{x}[Bc][Bc]{\scalebox{.9}{${\letterf}$}}
 \psfrag{i}[Bc][Bc]{\scalebox{.9}{$i$}}
 \psfrag{j}[Bc][Bc]{\scalebox{.9}{$j$}}
 \psfrag{k}[Bc][Bc]{\scalebox{.9}{$k$}}
 \psfrag{l}[Bc][Bc]{\scalebox{.9}{$l$}}
 \psfrag{m}[Bc][Bc]{\scalebox{.9}{$m$}}
 \rsdraw{.45}{.9}{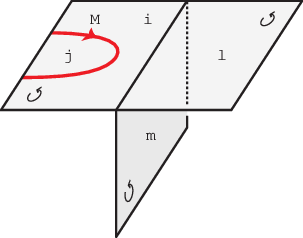} \;.
\end{center}
Therefore the state sum is invariant under $\Omega_5$.

Next, we have
\begin{center}
 \psfrag{x}[Bc][Bc]{\scalebox{.9}{${\letterf}$}}
 \psfrag{i}[Bc][Bc]{\scalebox{.9}{$i$}}
 \psfrag{j}[Bc][Bc]{\scalebox{.9}{$j$}}
 \psfrag{k}[Bc][Bc]{\scalebox{.9}{$k$}}
 \psfrag{l}[Bc][Bc]{\scalebox{.9}{$l$}}
 \psfrag{m}[Bc][Bc]{\scalebox{.9}{$m$}}
 \psfrag{n}[Bc][Bc]{\scalebox{.9}{$n$}}
 \psfrag{p}[Bc][Bc]{\scalebox{.9}{$p$}}
 \psfrag{q}[Bc][Bc]{\scalebox{.9}{$q$}}
 \rsdraw{.45}{.9}{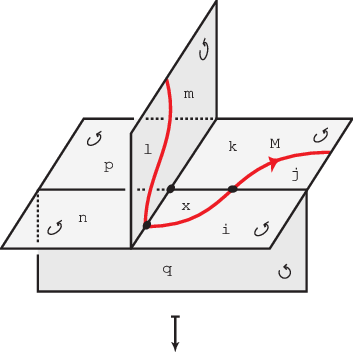}\\
$\ds\nu_{J}^{\frac{1-1}{2}}\,\nu_{J}^{\frac{1-1}{2}}\;\sum_{{\letterf}\in I}\dim({\letterf})$  \rsdraw{.45}{.9}{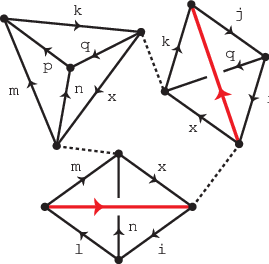}\\[.5em]
 $=\;$ \rsdraw{.45}{.9}{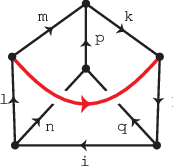}  $\; =\;$ \rsdraw{.45}{.9}{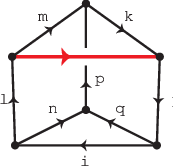} \\
 $\ds=\nu_{J}^{\frac{1-1}{2}}$ \rsdraw{.45}{.9}{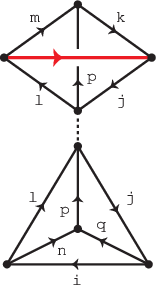}
\rsdraw{.45}{.9}{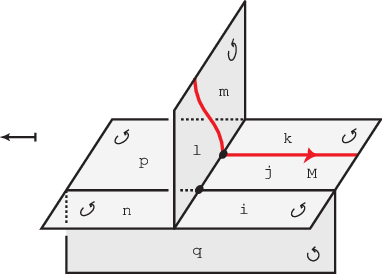}
\end{center}
Therefore the state sum is invariant under   $\Omega_8$.

Consider now the moves $\Omega_6$ and $\Omega_7$. We assume that the two red-bold strands are colored  with   simple objects $J$,  $J'$ of $\zz(\cc)$.  We have
\begin{center}
 \psfrag{M}[Bl][Bl]{\scalebox{.9}{\textcolor{red}{$J$}}}
 \psfrag{N}[Bl][Bl]{\scalebox{.9}{\textcolor{red}{$J'$}}}
 \psfrag{x}[Bc][Bc]{\scalebox{.9}{${\letterf}$}}
 \psfrag{i}[Bc][Bc]{\scalebox{.9}{$i$}}
 \psfrag{j}[Bc][Bc]{\scalebox{.9}{$j$}}
 \psfrag{k}[Bc][Bc]{\scalebox{.9}{$k$}}
 \psfrag{l}[Bc][Bc]{\scalebox{.9}{$l$}}
 \psfrag{m}[Bc][Bc]{\scalebox{.9}{$m$}}
 \psfrag{n}[Bc][Bc]{\scalebox{.9}{$n$}}
 \psfrag{p}[Bc][Bc]{\scalebox{.9}{$p$}}
 \rsdraw{.45}{.9}{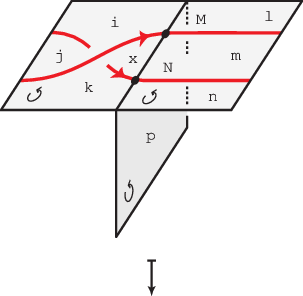}\\[-2em]
 $\ds\nu_{J}^{\frac{1-1}{2}}\,\nu_{J'}^{\frac{1-1}{2}}\;\sum_{{\letterf}\in I}\dim({\letterf})$ \;  \rsdraw{.45}{.9}{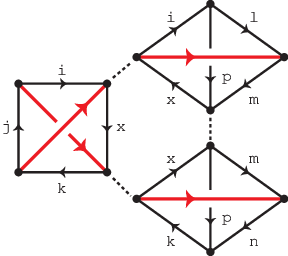}\\
 $=\;$ \rsdraw{.45}{.9}{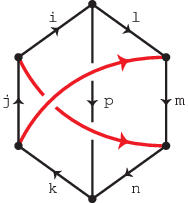}  $\; =\;$ \rsdraw{.45}{.9}{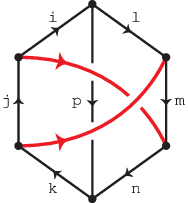} \\
  $\ds=\nu_{J}^{\frac{1-1}{2}}\,\nu_{J'}^{\frac{1-1}{2}}\;\sum_{{\letterf}\in I}\dim({\letterf})$ \;  \rsdraw{.45}{.9}{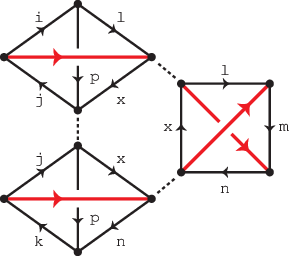}\\
 \rsdraw{.45}{.9}{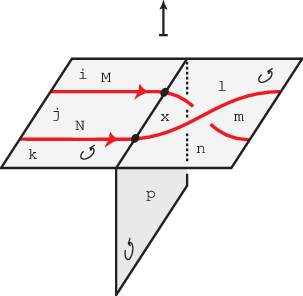}  \;.
\end{center}
Therefore the state sum is invariant under $\Omega_6$.

We verify the invariance for $\Omega_{7.1}$ (the case of $\Omega_{7.2}$ is similar). We have
\begin{center}
 \psfrag{M}[Bl][Bl]{\scalebox{.9}{\textcolor{red}{$J$}}}
 \psfrag{N}[Bl][Bl]{\scalebox{.9}{\textcolor{red}{$J'$}}}
 \psfrag{x}[Bc][Bc]{\scalebox{.9}{${\letterf}$}}
 \psfrag{i}[Bc][Bc]{\scalebox{.9}{$i$}}
 \psfrag{j}[Bc][Bc]{\scalebox{.9}{$j$}}
 \psfrag{k}[Bc][Bc]{\scalebox{.9}{$k$}}
 \psfrag{l}[Bc][Bc]{\scalebox{.9}{$l$}}
 \psfrag{m}[Bc][Bc]{\scalebox{.9}{$m$}}
 \psfrag{n}[Bc][Bc]{\scalebox{.9}{$n$}}
 \psfrag{p}[Bc][Bc]{\scalebox{.9}{$p$}}
 \psfrag{q}[Bc][Bc]{\scalebox{.9}{$q$}}
 \rsdraw{.45}{.9}{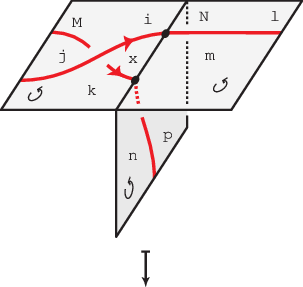}\\[-2em]
 $\ds\nu_{J}^{\frac{-1+1}{2}}\,\nu_{J'}^{\frac{1-1}{2}}\;\sum_{{\letterf}\in I}\dim({\letterf})$ \;  \rsdraw{.45}{.9}{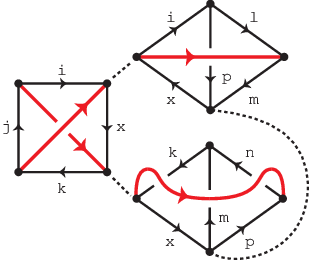}\\
 $=\;\ds\sum_{{\letterf}\in I}\dim({\letterf})$ \rsdraw{.45}{.9}{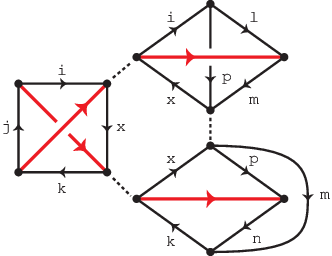}\\
 $=\;$ \rsdraw{.45}{.9}{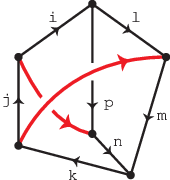}  $\;\,=\,\;$ \rsdraw{.45}{.9}{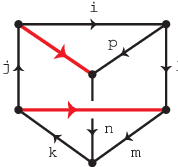}  $\;\,=$ \rsdraw{.45}{.9}{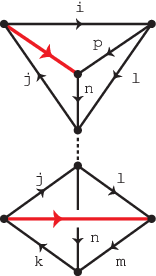} \\[.5em]
 $=\ds\nu_{J}^{\frac{-1+1}{2}}\,\nu_{J'}^{\frac{1-1}{2}}$ \rsdraw{.45}{.9}{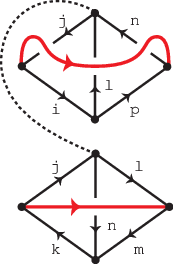}  \rsdraw{.45}{.9}{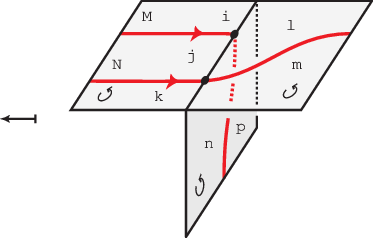}
\end{center}
Therefore the state sum is invariant under $\Omega_7$.

The invariance of $|M,L|_\cc$ under the moves $\Omega_4-\Omega_8$ with other orientations of the s-skeleton can be verified similarly.
As above, this verification is easy for $\Omega_5$ and $\Omega_6$ but longer for the other moves. However, having verified this invariance for $\Omega_5$ and $\Omega_6$, we can prove that $|M,L|_\cc$ does not depend on the orientation of the s-skeleton $P$. Then the invariance for the other moves follows from the special cases considered above. To prove our claim, it is enough to prove the invariance of $|M,L|_\cc$  under reversion of the orientation in an arbitrary region, $X$, of $P$. Applying bubble moves $T_4^{''}$ (see Figure~\ref{fig-movesTT}) at the boundary of $X$,  then using $\Omega_5$ to push the newly attached disks   near the crossings  and  using  $\Omega_6$ to pull the crossings into these disks, we can  reduce  the claim to the case where $X$ contains no crossings of the diagram  (note that all these transformations keep $|M,L|_\cc$). Then $X$ meets the link diagram in  several disjoint embedded arcs. Similarly, applying bubble moves $T_4^{''}$ and then    lune moves $\mathcal L^{\pm 1}$ (see Figure~\ref{fig-movesee}) to push  the newly attached disks between   the arcs of the diagram in $X$,  we can reduce the claim to the case where  the   diagram meets the disk $X$ along a single embedded arc. Finally, applying   the moves  $(T^{2,1})^{-1}$  (see Figure~\ref{fig-32moves}) to $P$ outside the diagram, we can reduce the claim to the case where $X$ has only two vertices in $P$ as in the picture below. There are two cases to consider, depending on whether or not this arc arrives and leaves in the same direction:
\begin{center}
\rsdraw{.45}{.9}{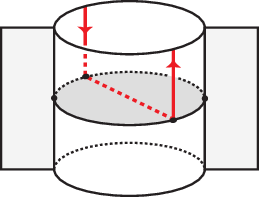}\qquad \text{and} \qquad
\rsdraw{.45}{.9}{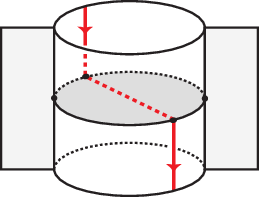}
\end{center}
Let us check that the two orientations of the disk do contribute similarly to the state sum. In the first case, we have
\begin{center}
 \psfrag{M}[Bl][Bl]{\scalebox{.9}{\textcolor{red}{$J$}}}
 \psfrag{x}[Bc][Bc]{\scalebox{.9}{$x$}}
 \psfrag{i}[Bc][Bc]{\scalebox{.9}{$i$}}
 \psfrag{j}[Bc][Bc]{\scalebox{.9}{$j$}}
 \psfrag{y}[Bc][Bc]{\scalebox{.9}{$y$}}
 \rsdraw{.45}{.9}{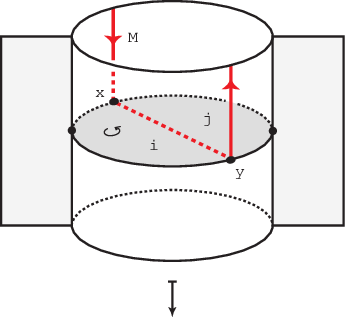} \\[-2em]
$\ds\nu_{J}^{\frac{\varepsilon_x^-+1}{2}}\,\nu_{J}^{\frac{-1-\varepsilon_y^+}{2}}\;\sum_{i,j\in I}\dim(i)\dim(j)$ \;  \rsdraw{.45}{.9}{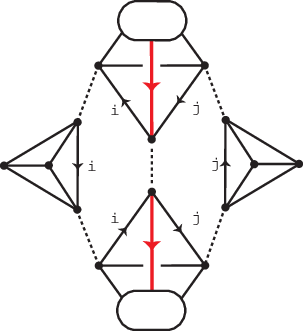}\\[-2em]
$=\ds\nu_{J}^{\frac{\varepsilon_x^--\varepsilon_y^+}{2}}\;$
\rsdraw{.45}{.9}{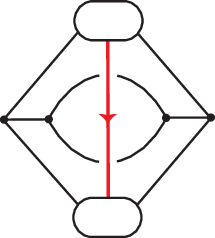}\phantom{XXXXXXXXXXXXXXXXXX}\\[-2.1em]
$=\ds\nu_{J}^{\frac{\varepsilon_x^--1}{2}}\,\nu_{J}^{\frac{+1-\varepsilon_y^+}{2}}\;\sum_{i,j\in I}\dim(i)\dim(j)$ \;  \rsdraw{.45}{.9}{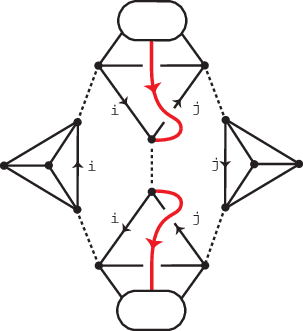}\\[-2em]
\rsdraw{.45}{.9}{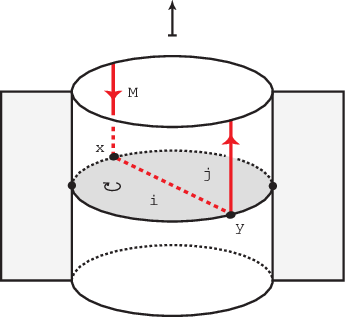}
\end{center}
In the second case, we have
\begin{center}
 \psfrag{M}[Bl][Bl]{\scalebox{.9}{\textcolor{red}{$J$}}}
 \psfrag{x}[Bc][Bc]{\scalebox{.9}{$x$}}
 \psfrag{i}[Bc][Bc]{\scalebox{.9}{$i$}}
 \psfrag{j}[Bc][Bc]{\scalebox{.9}{$j$}}
 \psfrag{y}[Bc][Bc]{\scalebox{.9}{$y$}}
 \rsdraw{.45}{.9}{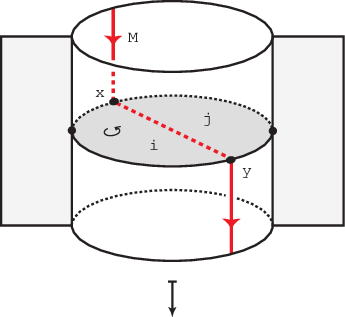} \\[-2em]
$\ds\nu_{J}^{\frac{\varepsilon_x^-+1}{2}}\,\nu_{J}^{\frac{1-\varepsilon_y^+}{2}}\;\sum_{i,j\in I}\dim(i)\dim(j)$ \;  \rsdraw{.45}{.9}{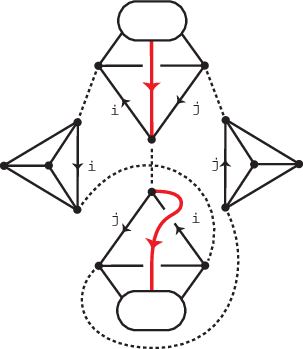}\\[.5em]
$=\ds\nu_{J}^{1+\frac{\varepsilon_x^--\varepsilon_y^+}{2}}\;\sum_{i,j\in I}\dim(i)\dim(j)$ \;
\rsdraw{.45}{.9}{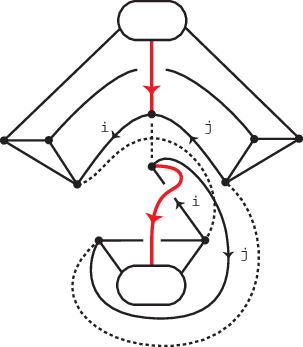}\\[.5em]
  $=\ds\nu_{J}^{1+\frac{\varepsilon_x^--\varepsilon_y^+}{2}}\;\sum_{i,j\in I}\dim(i)\dim(j)\;$ \rsdraw{.35}{.9}{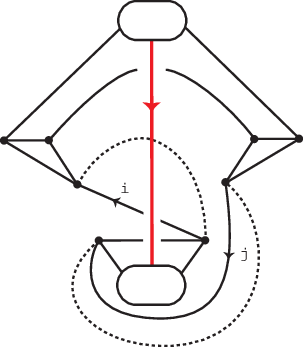}\\
$=\ds\nu_{J}^{1+\frac{\varepsilon_x^--\varepsilon_y^+}{2}}\;\sum_{j\in I}\dim(j)$ \;
\rsdraw{.45}{.9}{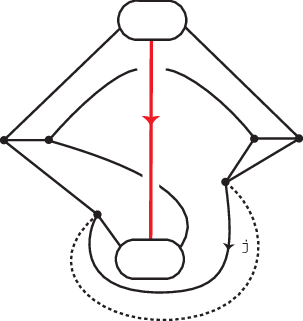} \\[.5em]
 $=\ds\nu_{J}^{1+\frac{\varepsilon_x^--\varepsilon_y^+}{2}}$ \;\rsdraw{.35}{.9}{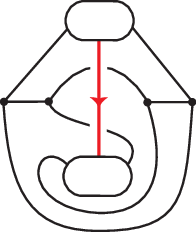}\;
$=\ds\nu_{J}^{-1+\frac{\varepsilon_x^--\varepsilon_y^+}{2}}$ \;\rsdraw{.35}{.9}{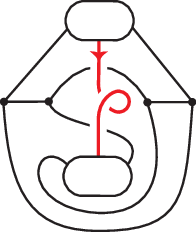}\\[.5em] 
$=\ds\nu_{J}^{-1+\frac{\varepsilon_x^--\varepsilon_y^+}{2}}\;\sum_{i,j\in I}\dim(i)\dim(j)$ \;  \rsdraw{.45}{.9}{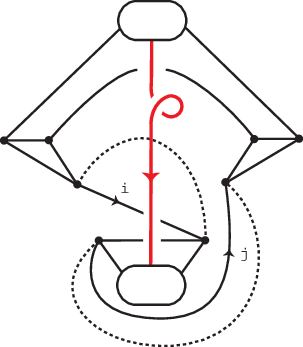}\\[.5em]
$=\ds\nu_{J}^{-1+\frac{\varepsilon_x^--\varepsilon_y^+}{2}}\;\sum_{i,j\in I}\dim(i)\dim(j)$ \;  \rsdraw{.45}{.9}{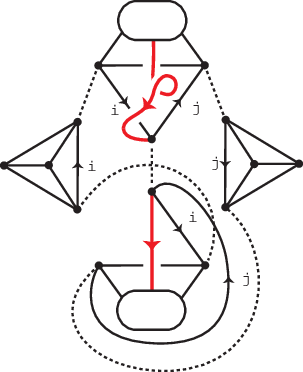}\\[.5em]
$=\ds\nu_{J}^{\frac{\varepsilon_x^--1}{2}}\,\nu_{J}^{\frac{-1-\varepsilon_y^+}{2}}\;\sum_{i,j\in I}\dim(i)\dim(j)$ \;  \rsdraw{.45}{.9}{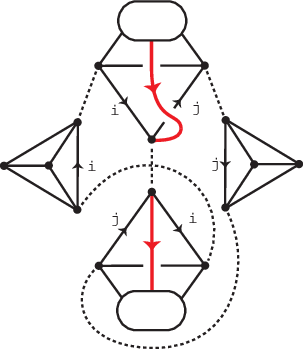}\\[-2em]
\rsdraw{.45}{.9}{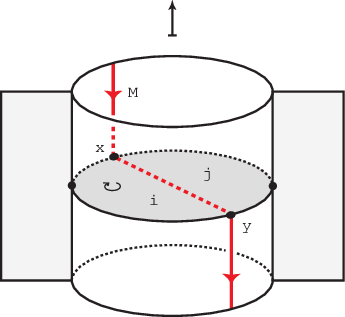}
\end{center}
This concludes the proof of Lemma~\ref{lem-state-3manNEW}.
\end{proof}

\subsection{The case of arbitrary links}  We extend the invariant $|M,L|_\cc$ of    links without trivial components in a compact oriented   3-manifold $M$  (see Lemma~\ref{lem-state-3manNEW}) to arbitrary links in $M$.
Consider   an oriented framed link  $L\subset \Int(M)$ whose   components   are colored by {simple} objects of $\zz(\cc)$. Let $L_1, \ldots , L_k$ be the trivial components of $L$, and let $\widetilde{L}$ be the sublink of $L$ formed by all the other components.  For  $i=1, \ldots, k$,  let $J_i$ be the
 object of $\zz(\cc)$  coloring $L_i$ and let $f_i\in \ZZ$ be the framing number of~$L_i$ (defined as in the standard case of knots in the 3-sphere).
Set
$$|M,L|_\cc = |M,\widetilde{L}|_\cc \,\prod_{i=1}^k v_{J_i}^{f_i} \dim(J_i) \in \kk. $$

\begin{thm}\label{thm-state-3manNEW}
$|M,L|_\cc$ is a topological invariant of the pair $(M,L)$.
\end{thm}

This theorem directly follows from Lemma~\ref{lem-state-3manNEW}.

The invariant $|M,L|_\cc$ can be computed via the same formula as in Section~\ref{sect-invariants-links} applied to an arbitrary s-skeleton $P$ of $M$ and   an  oriented  enriched link diagram $d$ in
$P$ representing~$L$ such that every component of $d$  has a self-crossing or meets the 1-skeleton of $P$.

\subsection{Remarks}\label{rem-noneedclosed}
(1)  Since all objects of $\zz(\cc)$ are
direct sums of simple objects, the invariant  $|M,L|_\cc$
  extends by  linearity to arbitrary colors of the
components of~$L$.

(2) Theorem~\ref{thm-equivalence-} directly extends to 3\ti manifolds with links, see Formula~\eqref{genethm104} in Section~\ref{proof104}.

\subsection{A special case}\label{rem-aspecial case}   It is obvious that for $L=\emptyset$,  we have  $|M,L|_\cc=|M|_\cc  $, where $|M|_\cc \in \kk$ is the invariant  defined for closed $M$  in Section~\ref{sec-ssot3m}  and generalized to   compact $M$
 in Section~\ref{sec-Io3m+}.  We obtain   a similar formula  for any
$\zz(\cc)$-colored framed oriented link  $L$  lying  in a 3-ball  in $  M$. Denote by $L'$ the  $\zz(\cc)$-colored framed oriented link in $S^3$ obtained
  by cutting out the ball containing $L$ from $M$ and  embedding   this ball - with $L$ inside - in $S^3$.
  The   embedding in question should be orientation-preserving.

\begin{lem}\label{lem-computforlinksinthesphere-}
$|M,L|_\cc = |M|_\cc \, F_{\zz(\cc)}(L' )$.
\end{lem}
\begin{proof}    Pick an   s-skeleton $P$ of $M$ and a  diagram $d$ of $L$ contained in   a region $r$ of~$P$. Inserting if necessary small curls in $d$, we can assume that $d$ has at least one self-crossing and all pre-twists of $d$ are equal to zero. Denote the   underlying ($4$-valent) graph  of $d$ by  ${\underline d}$.   Applying if necessary $\Omega_2$ to $d$ in $r$, we can ensure that  ${\underline d}$ is connected and all its edges have distinct endpoints.  Then  all regions of $d$   are  disks except  the  \lq\lq exterior region"  $r_0$ of $d$ in $r$, which is an annulus. Let $R_1 $ be the set of the disk regions of $d$ contained  in~$r$ and $R_2$   be the set of all other disk regions of $d$. Thus, $\Reg (d)= R_1\amalg R_2 \amalg \{r_0\}$.

Since   $P^{(1)} \cap  {  d}=\emptyset$,   the scalar ${\ast}_P ( \otimes_x \,|x|_c)\in \kk$   in the definition of $|M,L|_\cc$      expands as a product of two scalars $t^{c}_1, t^{c}_2\in \kk$. The scalar $t^{c}_1$ (resp.\@  $t^{c}_2$) is
obtained   from  the tensor associated with the vertices of   $ P^{(1)} $ (resp.\@  ${\underline d}$) through   the tensor contraction associated with the edges of     $ P^{(1)} $ (resp.\@  ${\underline d}$). Each $t^{c}_k $ is determined by  $i=c(r_0)\in I$ and the map  $c_k=c\vert_{ R_k }\co R_k \to I$. We therefore denote $t^{c}_k $ by  $t^{i,c_k}$. We have
\begin{gather*}
|M,L|_\cc
      = (\dim (\cc))^{-\vert P\vert}    \sum_{i\in I}
  \sum_{\substack {c\co \Reg (d)\to I \\ c(r_0)=i}}   \,\,  \left (\prod_{r \in R_1\cup  R_2} \!  \dim c(r) \right ) \,
t^{c}_1 t^{c}_2\\
=  (\dim (\cc))^{-\vert P\vert}   \sum_{i\in I} \sum_{c_1\co R_1\to I}   \left (\prod_{r \in R_1} \!  \dim c_1(r)   \right ) t^{i,c_1}
  \sum_{c_2\co R_2\to I} \left (\prod_{r \in R_2} \!  \dim c_2(r)   \right ) t^{i,c_2}.
\end{gather*}
Below we prove that for all $i\in I$,
\begin{equation}\label{equ-locall} \sum_{c\co R_1\to I} \left (\prod_{r \in R_1} \!  \dim c(r)   \right ) t^{i,c}= \dim (i) \, F_{\zz(\cc)}(L' ).\end{equation}
Substituting this  in the expression  for $|M,L|_\cc$ above, we obtain that
\begin{gather*}
|M,L|_\cc=  (\dim (\cc))^{-\vert P\vert}   \sum_{i\in I} \sum_{c\co R_2\to I}   \left (\prod_{r \in R_2} \!  \dim c(r)   \right ) t^{i,c_2}
\dim (i) \, F_{\zz(\cc)}(L' )\\
= |M|_\cc \, F_{\zz(\cc)}(L' ).
\end{gather*}

To prove  \eqref{equ-locall}, we need to study the   graph $\underline d$   in more detail.  Let $N\geq 1$ be the number of vertices of $\underline d$  (i.e.,  the number of crossings of $d$). Since the   graph $\underline d$   is 4-valent, it  has $2N$ edges.
A   computation of the Euler characteristic of $r$ shows that   $\underline  d$ splits $r$ into  $N+1$ disk regions and the exterior region $r_0$.  The diagram ${  d}$ also determines a      graph  ${d}^*\subset r$ as follows. Fix a point in each region  of $d$ in $r$  called the {\it center}  of the  region. These $N+2$ points are the vertices of ${d}^*$.  Every edge $e$ of ${\underline d}$ determines a {\it dual edge} $e^*$ of ${d}^*$ which connects the centers of the two regions adjacent to $e$, meets the interior of $e$ transversely in one point, and is disjoint from ${\underline d}$  otherwise. Note that the two regions adjacent to $e$ are always distinct so that the edges of $d^*$ are not loops.
We   choose the edges of ${d}^*$ so that they meet only in common vertices. The vertex of ${d}^*$ represented by the center of $r_0$ is denoted $O$.

By a {\it  subgraph} of a graph $G$  we mean a graph formed by some vertices and   edges of $G$.   A subgraph $F$ of $G$ is  {\it full} if all vertices   of $G$ are vertices of $F$. A {\it maximal tree} in $G$ is a full subgraph of $G$ which is a tree.   Each  subgraph $F$ of ${\underline d}$ determines a   full subgraph  $F^*$ of  ${d}^*$ whose edges are dual to the edges of  $\underline d$ not belonging to~$F$. Clearly, $F\cap F^* =\emptyset$.   If $F$ is a maximal tree in ${\underline d}$, then $F^*$ is   a maximal  tree in  ${d}^*$.  Indeed, since  every component of $r \setminus F^*$ contains a vertex of $\underline d$ and any two vertices of $\underline d$ can be related by a path in $  F\subset r \setminus F^*$, the  set $r \setminus F^*$ is connected. Hence $F^*$ is a forest  with $N+2$ vertices and $2N-(N-1)=N+1$ edges. Such a forest  is necessarily a tree.

Let $e_1,\ldots, e_{2N}$ be   the edges of $\underline d$ enumerated so that the following conditions are met. For $k=1,\ldots, 2N$, set  $F_k=\cup_{l=1}^{k} e_l \subset \underline d$ and observe that   $F^*_k$ is the full subgraph of $d^*$ with edges $e_{k+1}^*,\ldots, e_{2N}^*$. We require  that (a)  the graph $F_{N-1} $ is   a maximal tree in $\underline d$ and (b) for all
$k=N, \ldots,  2N$,  the graph  $F_{k-1}^*$ has a 1-valent vertex distinct from $ O$ and  incident to   $e_k^*$.  It is easy to choose       $e_1,\ldots, e_{N-1}$    to ensure (a). We explain now how  to
choose  $e_k$  with $k\geq N$   to ensure (b).  For $k=N$,  pick a 1-valent vertex  $v_N $   of the maximal  tree  $F_{N-1} ^*\subset d^*$ distinct from $O$. (Such a vertex exists because a tree having at least one edge necessarily has $\geq 2$  vertices of valency 1.)  Let  $e_N$ be the edge of $\underline d$ such that $e^*_N$ is the   edge of $F_{N-1} ^*$ adjacent to $v_N$.   The graph $F_N^*$ is obtained from $F_{N-1} ^*$ by removing    $e^*_N$ (keeping all the vertices).  Clearly,   $F_N^*$ is a disjoint union of the isolated vertex $v_N$ and a tree with $N+1$ vertices.  We choose $e_{N+1}$ so that $e^*_{N+1}$ is the edge of the latter  tree   adjacent to a 1-valent vertex  $v_{N+1}\neq O$.  Continuing by induction we obtain  that the graph    $F_k^*$ constructed  at the $k$-th step consists of   isolated vertices $v_N, \ldots, v_k$ and a tree with $2N+1-k$ vertices.  We choose $e_{k+1}$ so that $e^*_{k+1}$ is the edge of the latter  tree   adjacent to a 1-valent vertex  $v_{k+1} \neq O$. This process   stops at $k=2N$ because the graph $F_{2N}^*$ has only isolated vertices.


We can now  prove \eqref{equ-locall}. Recall that the term $t^{i,c}$ is obtained by placing a small colored tetrahedron-type graph in every  crossing of   $d$, taking the tensor product of the  associated $\inv_{\zz(\cc)}$-invariants,  and tensor contracting  this product along the edges of $\underline d$. The colors  of the edges of these tetrahedron-type graphs are determined by $i, c$, and the given coloring of the link components. We shall   perform  the  tensor contraction at one edge  at a time  following the order  of  the edges
$  e_1,\ldots, e_{2N}$
fixed above.  Condition (a) shows that  for $k=1, \ldots, N-1$, the $k$-th tensor contraction
involves two different  pieces of the diagram so that we can apply  Claim  (d) of Lemma~\ref{lem-calc-diag}.
At each of these $N-1$ steps, the tensor contraction of the  tensor product of two $\inv_{\zz(\cc)}$-invariants yields the $\inv_{\zz(\cc)}$-invariant
of a   \lq\lq fused" diagram.     For  $k \geq N $, the endpoint  $v_k$ of $e^*_k$ is the center of a disk region, $V_k \subset r$,  of $d$
adjacent to the edge $e_k$. Note that all other sides of $V_k$ (i.e., all other edges of $d$ adjacent to $V_k$) must have appeared at the previous steps
among $e_1, \ldots, e_{k-1}$. This follows from the fact that  neither of the edges $e_{k+1}^*, \ldots, e_{2N}^*$ is adjacent to~$v_k$. Therefore   the fusions corresponding to all sides of $V_k$ except  $e_k$ have been   done before the $k$-th step.
The  tetrahedron-type graph  associated with a vertex of $V_k$ has a side   in $V_k$; under the
fusions in question,  these sides are united into a single arc labeled with $c(V_k)$.  This shows that   the $k$-th tensor contraction
involves two vertices of the same  connected  piece of the diagram. We apply Claim  (c) of Lemma~\ref{lem-calc-diag} where the roles of $i, u, v$ are played respectively by $c(V_k)$ and the endpoints of $e_k$ (here we use that all edges of $\underline d$ have distinct endpoints). After fusion along $e_k$ our diagram will contain an embedded $c(V_k)$-colored circle.  Lemma~\ref{lem-calc-diag}.c  says that to preserve the state sum  under the fusion  we must delete this circle,  the summation over
$c(V_k)\in I$, and the factor $\dim  c(V_k)$. The rest of the diagram  is the \lq\lq fused diagram" obtained at the $k$-th step.
Continuing by induction, we obtain that  the left-hand side of \eqref{equ-locall} is equal to the   $\inv_{\zz(\cc)}$-invariant of the fused diagram obtained at the  last step $k=2N$. This   diagram consists of a diagram of $L'$ surrounded by a big circle colored with $i$.  The   $\inv_{\zz(\cc)}$-invariant of this fused diagram
is equal to $\dim (i) \, F_{\zz(\cc)}(L' )$. This proves our claim. \end{proof}

Applying Lemma~\ref{lem-computforlinksinthesphere-} to $M=S^3$ and using the equality $\vert S^3  \vert_\cc=(\dim(\cc))^{-1}$ (see Section~\ref{sec-computat}), we obtain that for any  $\zz(\cc)$-colored framed oriented link $L\subset S^3$,
\begin{equation}\label{computforlinksinthesphere}\vert S^3,L \vert_\cc=(\dim(\cc))^{-1} F_{\zz(\cc)}(L). \end{equation}


\section{Deduction of Theorems~\ref*{thm-equivalence-} and~\ref*{thm-equivalence} from Lemma~\ref*{le-didim}}

\subsection{Conventions} In this section, the symbol  $D^2$ denotes  the unit disk in $\CC$ with counterclockwise orientation
and $S^1=\partial D^2$ is the unit circle  with   counterclockwise orientation. Unless explicitly stated to the contrary, the torus $S^1\times S^1$ and the solid tori $S^1\times D^2$ and $D^2\times S^1$   are provided with the product orientations.

\subsection{A surgery formula}\label{asurgeryformula} Let $Z\co  \mathrm{Cob}_3 \to \mathrm{vect}_\kk$ be a 3-dimensional TQFT  over a field~$\kk$. We establish a surgery formula for the values of $Z$ on closed   3-manifolds.

Given a  framed oriented link $L=\cup_{q=1}^N L_q$ in $S^3$, denote by $E_L$ its exterior, i.e., the complement in $S^3$ of an open regular neighborhood of $L$. We endow $E_L$ with the orientation induced by the right-handed orientation of $S^3$. There is  a homeomorphism $f=f_L$ from   the disjoint union $N (S^1 \times S^1)=\amalg_{q=1}^N  (S^1 \times S^1)_q$ of $N$ ordered $2$-tori  to $\partial E_L$  carrying the $q$-th copy   of   the torus  to the boundary of a closed regular neighborhood of $L_q$ so that    $f(S^1\times {\text {pt}})$ is   a positively oriented meridian of $L_q$ and  $ f({\text {pt}} \times S^1)$ is   a positively oriented longitude of $L_q$ determined by the framing for   $q =1 , ..., N$.
Observe that  $f$ is an orientation preserving homeomorphism   $ N (S^1 \times S^1) \simeq  -\partial E_L$. We use $f$ to identify $-\partial E_L$ and $N (S^1 \times S^1) $.
Then  $Z
(-\partial E_L)= A^{\otimes N}$, where $A=Z(S^1 \times S^1)$. Consider the homomorphism
$$Z\bigl(E_L, -\partial E_L ,\emptyset\bigr)\co  A^{\otimes N}=
Z
(-\partial E_L) \to Z(\emptyset)=\kk.$$
For any  $y_1, \dots, y_N\in A$, set
$$
Z(L;y_1, \dots,y_N)=Z\bigl(E_L, -\partial E_L ,\emptyset\bigr)( {{y}_1} \otimes \cdots  \otimes {{y}_N}) \in \kk.
$$

Consider the solid torus ${V}=-( S^1\times D^2)$ with orientation opposite to the product orientation. Then
 $\partial {V} =S^1\times S^1$ in the category of oriented manifolds.   Let $w\in A=Z(S^1 \times S^1)$ be the image of $1\in \kk$ under   the homomorphism $Z({V}, \emptyset,  \partial {V})\co \kk \to  A$. We call $w$ the {\it canonical vector} associated with $Z$. Pick an arbitrary   basis   $Y$ of the vector space $A$ and expand $w=\sum_{y\in Y} w_y y$ where  $w_y\in \kk$.

\begin{lem}\label{lem-TQFTsurg}
Let $M$ be a  closed oriented $3$-manifold   obtained by surgery on $S^3$ along a framed  link $L=L_1\cup \cdots \cup L_N \subset S^3$. For any orientation of $L$,
\begin{equation*}
Z(M)=\sum_{{y}_1, \dots , {y}_N \in Y} \left(\prod_{q=1}^N w_{{y}_q} \right) Z(L;{y}_1, \dots,{y}_N).
\end{equation*}
\end{lem}
\begin{proof}   Let ${V}_N $ be a disjoint union of   $N$   copies of ${V}$.
The associated homomorphism $Z({V}_N, \emptyset,  \partial {V}_N)\co \kk \to  A^{\otimes N}$ carries $1\in \kk$ to
  $$w^{\otimes N}=\sum_{{y}_1, \dots , {y}_N \in Y} \left(\prod_{q=1}^N w_{y_q} \right)   {{y}_1} \otimes \cdots  \otimes {{y}_N}.$$
The 3-cobordism $M=(M, \emptyset, \emptyset)$ can be obtained  by attaching the cobordism $(E_L, -\partial E_L, \emptyset)$ on top of the cobordism $({V}_N, \emptyset, \partial  {V}_N )$
along the homeomorphism $f\co  \partial  {V}_N = N(S^1\times S^1)\to -\partial E_L$ specified above.
Therefore $$Z(M, \emptyset, \emptyset)= Z(E_L, -\partial E_L, \emptyset) \circ Z({V}_N, \emptyset, \partial  {V}_N ) \co \kk \to \kk $$
and
\begin{align*}
Z(M)&=Z(M, \emptyset, \emptyset) (1)= Z(E_L, -\partial E_L, \emptyset) (\sum_{{y}_1, \dots , {y}_N \in Y} \left(\prod_{q=1}^N w_{y_q} \right)   {{y}_1} \otimes \cdots  \otimes {{y}_N})\\
&= \sum_{{y}_1, \dots , {y}_N \in Y} \left(\prod_{q=1}^N w_{y_q} \right) \, Z\bigl(E_L, -\partial E_L ,\emptyset\bigr)( {{y}_1} \otimes \cdots  \otimes {{y}_N})\\
&= \sum_{{y}_1, \dots , {y}_N \in Y} \left(\prod_{q=1}^N w_{y_q} \right) Z(L;{y}_1, \dots,{y}_N).
\end{align*}
\end{proof}

\subsection{Link TQFTs }\label{sect-enrichedTQFT}  For any category $\bb$, we define a category   $  \mathcal{L}_\bb$ of 3-cobordisms with $\bb$-colored framed oriented links inside. The objects  of $\mathcal{L}_\bb$  are closed oriented surfaces. A morphism   $\Sigma_0 \to \Sigma_1$ in $ \mathcal{L}_\bb$ is represented by  a triple $(M,{h}, K)$,  where $M$ is
 a compact oriented 3-manifold, ${h}$ is      an orientation-preserving homeomorphism $  (-\Sigma_0) \sqcup\Sigma_1 \simeq  \partial M$, and $K$ is a $\bb$-colored framed oriented link in $M\setminus \partial M$. (A link $K$ is $\bb$-colored if every component of $K$ is endowed with an object of $\bb$ called its color.)
Two such triples $(M, {h}, K )$ and
$(M', {h}' , K')$ represent the same morphism   if there is an orientation-preserving homeomorphism  $F\co  M \to M'$ such that ${h}'=F{h}$ and $K'=F(K)$ in the class of $\bb$-colored framed oriented links. The composition of morphisms in  $  \mathcal{L}_\bb$  is defined via the gluing of cobordisms and the tensor product in  $  \mathcal{L}_\bb$ is defined via disjoint union. This turns $  \mathcal{L}_\bb$ into a symmetric monoidal category. The links in question may be empty so that the category $\mathrm{Cob}_3$ of Section~\ref{Preliminaries on  TQFTs}  is a subcategory of
  $ \mathcal{L}_\bb$.

By a  {\it link TQFT} we mean a symmetric monoidal functor $Z\co  \mathcal{L}_\bb \to \mathrm{vect}_\kk$.    We   establish a version of Lemma~\ref{lem-TQFTsurg} for such a $Z$.
 Consider disjoint
  framed oriented links $K$ and $L=\cup_{q=1}^N L_q$  in $S^3$ and assume that $K$ is $\bb$-colored. Then $K$ lies in the exterior $E_L$ of $L$.
  As in Section~\ref{asurgeryformula},     $Z
(-\partial E_L)= A^{\otimes N}$, where $A=Z(S^1 \times S^1)$.
For any  $y_1, \dots, y_N\in A$, set
$$
Z(K,L;y_1, \dots,y_N)=Z\bigl((E_L,K), -\partial E_L ,\emptyset\bigr)( {{y}_1} \otimes \cdots  \otimes {{y}_N}) \in \kk,
$$
where the pair $ (E_L,K)$ is viewed as a morphism $-\partial E_L \to
\emptyset$ in  $\mathcal{L}_\bb$. Let $M$ be the closed oriented 3-manifold obtained from $S^3$ by surgery along $L$.
Then    $K\subset E_L\subset M$ is   a $\bb$-colored framed oriented link in $M$.
\begin{lem}\label{lem-surgysurgy}
Let $Y$ be any basis of $A$ and $w=\sum_{y\in Y} w_y y\in A$ be the canonical vector. Then
$$
Z(M, K)=\sum_{{y}_1, \dots , {y}_N \in Y} \left(\prod_{q=1}^N w_{{y}_q} \right) Z(K,L;{y}_1, \dots,{y}_N).
$$
\end{lem}
For $K=\emptyset$, we recover Lemma~\ref{lem-TQFTsurg}. The  proof of Lemma~\ref{lem-surgysurgy} repeats that of Lemma~\ref{lem-TQFTsurg} with the obvious changes.

\subsection{The  link TQFT  $\vert \cdot \vert_{\cc}$}\label{The  link TQFT} Fix  a spherical fusion category~$\cc$ over an algebraically
closed field~$\kk$ such that $\dim \cc\neq 0$.
The results of Section~\ref{State sum invariants of   links} allow us to extend the TQFT $\vert \cdot \vert_{\cc}\co \mathrm{Cob}_3 \to \mathrm{vect}_\kk$ of Section~\ref{sec-TQFT} to  a link TQFT $ \mathcal{L}_{\zz (\cc)} \to \mathrm{vect}_\kk$, where $\zz (\cc)$ is the center of $\cc$.  On surfaces and 3-cobordisms with empty links these TQFTs are equal.
In general, the construction  follows the same lines as in Sections~\ref{sec-Io3m++}    and~\ref{sec-TQFT-} but involves 3-cobordisms with $\zz (\cc)$-colored framed oriented links inside.   The resulting link TQFT  is also  denoted $\vert \cdot \vert_{\cc}$.

   Let $0\in D^2$ be  the center of   $D^2$.   For any $j \in \Ob (\zz(\cc))$, denote by $U^j$ the solid torus $D^2 \times S^1$ endowed with
 the $j$-colored framed oriented knot $\{0\}\times S^1$  whose  orientation is induced by that of $S^1$   and whose framing is  constant, i.e.,   determined by a non-zero tangent vector of $D^2$ at $0$. Clearly, $\partial U^j= S^1 \times S^1$ so that the link TQFT
$\vert \cdot
\vert_\cc$ produces a vector
$$
y^j=\vert U^j, \emptyset, \partial U^j \vert_\cc \in A  =\vert S^1 \times S^1 \vert_\cc.
$$

 By  M\"{u}ger's theorem (Theorem \ref{thm-center-modular}),  the category $\zz(\cc)$   is modular and anomaly free. We fix a
 (finite)
representative set~${\mathcal J}$  of simple objects of~$\zz(\cc)$.

\begin{lem}\label{le-bassis}   The  set $Y=(y^j)_{j \in {\mathcal J}}$ is a basis of the vector space $A$. The canonical vector $w\in A$ expands as $w= (\dim(\cc))^{-1}  \sum_{j\in {\mathcal J}} \dim(j) y^j$.
\end{lem}

\begin{proof}  Consider      the   framed oriented Hopf link $ K\cup L\subset S^3$
whose components $K$, $L$  have   framing $0$  and linking number $1$.  Endow $K$ with a color   $i\in \mathcal J$.  The definitions of Section~\ref{sect-enrichedTQFT}  applied to the link TQFT  $\vert \cdot \vert_{\cc}$ yield a $\kk$-linear map $A\to \kk, y\mapsto \vert K ,L;{y} \vert_\cc\in \kk$. We   compute the value of this  map on $y^j\in A$ for   $j\in \mathcal J$. Since the gluing of $ D^2\times S^1 $ to the exterior $E_L$ of $L$ along the homeomorphism $f_L\co  S^1\times S^1 \to \partial E_L$ yields   $ S^3$,   the  functoriality of   $\vert \cdot \vert_\cc$  implies that
 $$\vert K ,L;{y^j} \vert_\cc = \vert ((E_L,K), -\partial E_L ,\emptyset) \vert_\cc  ( {{y}^j}  )= \vert S^3, H_{i,j} \vert_\cc ,$$
 where $H_{i,j}=K\cup L$ is the framed oriented  Hopf link whose components $K, L$ are colored  with $i,j$ respectively.  Formula \eqref{computforlinksinthesphere} gives then
 $$\vert K ,L;{y^j} \vert_\cc = (\dim(\cc))^{-1} F_{\zz(\cc)}(H_{i,j})  =(\dim(\cc))^{-1}   S_{i,j},$$
 where $[S_{i,j} ]_{ i, j\in {\mathcal J}}$ is   the $S$-matrix of   $\zz(\cc)$. By the definition of a modular category, this matrix is non-degenerate.
Hence the vectors $(y^j)_{j}$ are linearly independent. By Lemma~\ref{le-didim}, $\dim A=\dim \tau_{\zz(\cc)} (S^1\times S^1)$. By \cite{Tu1}, $\dim \tau_{\zz(\cc)} (S^1\times S^1)=\mathrm{card}\, {\mathcal J}$. Therefore $\dim A=\mathrm{card}\, {\mathcal J}$ and
the  set $Y=(y^j)_{j \in {\mathcal J}}$ is a basis of   $A$.

Consider   the $\zz(\cc)$-colored framed oriented  link   $ H^+_{i,j}=K^+\cup L^+ \subset S^3 $   obtained from $H_{i,j}$
by adding a positive twist to the framing of each component.  Since   $L^+$    is an unknot with framing $1$, the surgery on $S^3$ along $L^+$
gives   $S^3$. The  knot  $K^+$   gives after this surgery an  unknot
in $S^3$ with   framing $0$ and color $i$. Denote this unknot by $K_0^i$.  By \eqref{computforlinksinthesphere},
$$
\vert S^3, K_0^i \vert_\cc=(\dim(\cc))^{-1} F_{\zz(\cc)}(K_0^i)=(\dim(\cc))^{-1}\dim (i).
$$
On the other hand, Lemma~\ref{lem-surgysurgy} gives
$$
\vert S^3, K_0^i \vert_\cc=\sum_{_{j\in {\mathcal J}}}  w_{{j} } \vert K^+ ,L^+;{y^j} \vert_\cc ,
$$
where $w=\sum_{j\in {\mathcal J}} w_j y^j$ with $w_j\in \kk$ for all $j$.
As above,
$$
\vert K^+ ,L^+;{y^j} \vert_\cc =\vert S^3, H^+_{i,j} \vert_\cc=(\dim(\cc))^{-1} F_{\zz(\cc)}(H^+_{i,j})  =(\dim(\cc))^{-1} v_i v_j  S_{i,j},
$$
where $v_k\in \kk$ is the twist scalar of $k \in {\mathcal J}$.
Combining these equalities, we obtain that for all   $i\in {\mathcal J}$,
$$
\dim (i)=\sum_{_{j\in {\mathcal J}}}  w_{{j} }   v_i v_j S_{i,j}.
$$
Since the $S$-matrix  and the twist scalars are   invertible, this system of equations has a unique solution. By \cite{Tu1}, Chapter II,  Formula (3.8.d),
$$
\dim (i)= (\dim( \cc ))^{-1} \sum_{_{j\in {\mathcal J}}}    \dim(j)     v_i v_j S_{i,j}.
$$
Hence $w_j= (\dim(\cc))^{-1}   \dim(j) $ for all $j$.
\end{proof}

\subsection{Proof of Theorem~\ref{thm-equivalence-}}\label{proof104}
We shall prove that  for any $\zz(\cc)$-colored framed oriented link
 $K $ in a closed oriented 3-manifold $M$,
\begin{equation}\label{genethm104}
|M,K|_\cc=\tau_{\zz(\cc)}(M,K).
\end{equation}
For $K=\emptyset$, this gives  Theorem~\ref{thm-equivalence-}.

Present $M$ by surgery on $S^3$ along a framed oriented  link $L=L_1\cup \cdots \cup L_N \subset S^3$. Pushing $K$ in the exterior $E_L$ of $L$ in $S^3$, we can assume that $K\subset E_L$.
For any $j_1,..., j_N\in {\mathcal J}$, denote by $L_{(j_1, \dots,j_N)}$   the link $L$ whose  components $L_1, \ldots , L_N$ are colored with  $j_1, \ldots , j_N$ respectively. We shall apply  the notation of Section~\ref{sect-enrichedTQFT} to the link TQFT $Z=\vert \cdot
\vert_\cc$. Observe that
$$
\vert K, L;{y}^{j_1}, \dots,{y}^{j_N}\vert_\cc=\vert S^3, K\cup L_{(j_1, \dots,j_N)}\vert_\cc .
$$
This follows from the definitions, the functoriality of   $\vert \cdot
\vert_\cc$, and the fact that the gluing of   $\amalg_{q=1}^N U^{j_q}$ to  $(E_L,K)$ along the homeomorphism $f_L\co N (S^1 \times S^1)\to \partial E_L$ introduced in Section~\ref{asurgeryformula}  yields  the pair $(S^3,K\cup L_{(j_1, \dots,j_N)})$.
By~\eqref{computforlinksinthesphere},
  $$\vert S^3,K\cup  L_{(j_1, \dots,j_N)}\vert_\cc=(\dim(\cc))^{-1} F_{\zz(\cc)}(K\cup  L_{(j_1, \dots,j_N)}).$$
Applying Lemma~\ref{lem-surgysurgy} to   $Z= \vert \cdot \vert_{\cc}$ and the basis of $\vert S^1 \times S^1 \vert_\cc$ given by Lemma~\ref{le-bassis}, we obtain
\begin{align*}
\vert M , K\vert_\cc &=  \sum_{{j}_1, \dots , {j}_N \in {\mathcal J}} \left(\prod_{q=1}^N \frac{\dim({j}_q)}{\dim(\cc)}  \right)  \vert K, L;{y}^{j_1}, \dots,{y}^{j_N}\vert_\cc\\
 &=  \sum_{{j}_1, \dots , {j}_N \in {\mathcal J}} \left(\prod_{q=1}^N \frac{\dim({j}_q)}{\dim(\cc)}  \right) (\dim(\cc))^{-1} F_{\zz(\cc)}(K\cup L_{({j}_1, \dots,{j}_N)})\\
&=(\dim(\cc))^{-N-1}\sum_{{j}_1, \dots , {j}_N \in {\mathcal J}} \left(\prod_{q=1}^N \dim({j}_q)  \right) F_{\zz(\cc)}(K\cup L_{({j}_1, \dots,{j}_N)})\\
&=\tau_{\zz(\cc)}(M, K),
\end{align*}
where the last equality is the definition of $\tau_{\zz(\cc)}(M, K)$ in \cite{Tu1}.

\subsection{Proof of Theorem~\ref{thm-equivalence}}
For a  category $\bb$, we define a category   $   \mathcal{G}_\bb$ of 3-cobordisms with $\bb$-colored ribbon graphs inside, see \cite{Tu1} for a definition of colored ribbon graphs. (Here we consider only ribbon graphs disjoint from the bases of cobordisms.) The category   $   \mathcal{G}_\bb$  is defined as  $ \mathcal{L}_\bb $   replacing   \lq\lq framed oriented links" with \lq\lq ribbon graphs".  The category   $   \mathcal{G}_\bb$ contains $\mathrm{Cob}_3$ as a subcategory and  is a symmetric monoidal category in the obvious way.
By a  {\it graph TQFT} we mean a symmetric monoidal functor $   \mathcal{G}_\bb \to \mathrm{vect}_\kk$. In the rest of the argument $\bb=\zz(\cc)$.

The TQFT $\tau_{\zz(\cc)}\co  \mathrm{Cob}_3 \to \mathrm{vect}_\kk$ extends to a graph TQFT $\mathcal{G}_{\zz (\cc)} \to \mathrm{vect}_\kk$ still denoted $\tau_{\zz(\cc)}$, see \cite{Tu1}, Chapter IV.    This   TQFT   is   non-degenerate in the following sense:  the vector space $\tau_{\zz(\cc)} (\Sigma)$ associated with any closed oriented surface $\Sigma$ is generated by the vectors $\tau_{\zz(\cc)} (M, \emptyset, \Sigma) (1) $, where $M$ runs over all compact   oriented 3-manifolds
with $\zz(\cc)$-colored ribbon graphs inside and with $\partial M=\Sigma$.

The TQFT $\vert \cdot
\vert_\cc\co  \mathrm{Cob}_3 \to \mathrm{vect}_\kk$  also can be extended to a graph TQFT. To do this, one proceeds similarly to Section~\ref{State sum invariants of   links} by representing  the ribbon graphs in a 3-manifold $M$ by diagrams on a s-skeleton $P$ of $M$. Then one defines a state sum on such a diagram $d$ as in  Section~\ref{State sum invariants of   links}.  A typical coupon of $d$
\begin{equation*}
 \psfrag{i}[Bc][Bc]{\scalebox{.9}{$i$}}
 \psfrag{j}[Bc][Bc]{\scalebox{.9}{$j$}}
  \psfrag{k}[Bc][Bc]{\scalebox{.9}{$k$}}
 \psfrag{l}[Bc][Bc]{\scalebox{.9}{$l$}}
  \psfrag{m}[Bc][Bc]{\scalebox{.9}{$m$}}
 \psfrag{f}[Bc][Bc]{\scalebox{.9}{$f$}}
   \psfrag{A}[Bc][Bc]{\scalebox{.9}{\textcolor{red}{$A$}}}
 \psfrag{B}[Bc][Bc]{\scalebox{.9}{\textcolor{red}{$B$}}}
   \psfrag{C}[Bc][Bc]{\scalebox{.9}{\textcolor{red}{$C$}}}
 \psfrag{D}[Bc][Bc]{\scalebox{.9}{\textcolor{red}{$D$}}}
   \psfrag{E}[Bc][Bc]{\scalebox{.9}{\textcolor{red}{$E$}}}
 \rsdraw{.45}{.9}{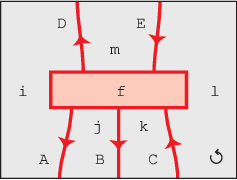}\;,
\end{equation*}
 contributes to the state sum the factor
\begin{equation*}
 \psfrag{i}[Bc][Bc]{\scalebox{.9}{$i$}}
 \psfrag{j}[Bc][Bc]{\scalebox{.9}{$j$}}
  \psfrag{k}[Bc][Bc]{\scalebox{.9}{$k$}}
 \psfrag{l}[Bc][Bc]{\scalebox{.9}{$l$}}
  \psfrag{m}[Bc][Bc]{\scalebox{.9}{$m$}}
 \psfrag{f}[Bc][Bc]{\scalebox{.9}{$f$}}
   \psfrag{A}[Bc][Bc]{\scalebox{.9}{\textcolor{red}{$A$}}}
 \psfrag{B}[Bc][Bc]{\scalebox{.9}{\textcolor{red}{$B$}}}
   \psfrag{C}[Bc][Bc]{\scalebox{.9}{\textcolor{red}{$C$}}}
 \psfrag{D}[Bc][Bc]{\scalebox{.9}{\textcolor{red}{$D$}}}
   \psfrag{E}[Bc][Bc]{\scalebox{.9}{\textcolor{red}{$E$}}}
 \inv_\cc \left (\rsdraw{.45}{.9}{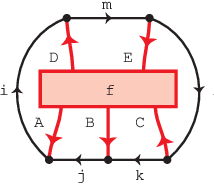} \right )\,,
\end{equation*}
where the invariant $\inv_\cc$ of $\cc$-colored knotted graphs in $S^2$  is extended straightforwardly (by the Penrose calculus) to $\cc$-colored ribbon graphs in $S^2$.  (In this example of a coupon,   $A,B,C,D,E \in \Ob (\zz(\cc))$,
 $f\in \Hom_{\zz(\cc)}(A \otimes B \otimes C^*,D^* \otimes E)$, and $i,j,k,l,m \in I$ are the colors of the regions of $d$ in $P$ adjacent to the coupon.)
The resulting  state sum is  invariant  under the moves $\Omega_1-\Omega_8$ (away from the coupons) and under the moves pushing a coupon  over or under a strand or  across an edge of~$P$. Therefore the state sum in question yields an isotopy invariant of $\zz(\cc)$-colored ribbon   graphs  in $M$. The extension of $\vert \cdot
\vert_\cc$  to $\mathcal{G}_{\zz (\cc)}$  proceeds   as in Sections~\ref{sec-Io3m++},~\ref{sec-TQFT-}, and~\ref{The  link TQFT} replacing   \lq\lq framed oriented links" with \lq\lq ribbon graphs".
The resulting graph TQFT
$\mathcal{G}_{\zz (\cc)} \to \mathrm{vect}_\kk$ is still denoted $\vert \cdot
\vert_\cc$.

 We claim that   there is a    natural monoidal isomorphism of  the functors
$\tau_{\zz(\cc)} $ and
$\vert \cdot
\vert_\cc$ from $\mathcal{G}_{\zz (\cc)}  $ to $ \mathrm{vect}_\kk$. Restricting both functors to $\mathrm{Cob}_3$, we obtain the theorem. Our claim follows from a general criterion establishing isomorphism of two TQFTs, cf.\ \cite{Tu1}, Chapter III, Section 3. Namely, if at least one of the TQFTs is non-degenerate,    the values of these TQFTs on cobordisms with empty bases are equal, and the vector spaces associated by these TQFTs with any closed oriented surface   have equal dimensions, then these TQFTs are isomorphic. Here by a TQFT we mean a generalized TQFT incorporating graph TQFTs.  The first condition holds because
the graph TQFT $\tau_{\zz(\cc)}$ is non-degenerate. That $
\vert  M,K
\vert_\cc=\tau_{\zz(\cc)}(M,K)$ for any $\zz(\cc)$-colored ribbon   graph $K$ in a closed oriented 3-manifold $M$  is proven along the same lines  as Formula~\eqref{genethm104}. The equality of dimensions  is provided by  Lemma~\ref{le-didim}. This completes the proof of our claim and of the theorem.

\section{The coend of the center}\label{thecoendof}
The aim of this section is to compute the coend of the center   of a fusion category. This computation will be instrumental in the proof of Lemma~\ref{le-didim} given  in the next section. It is based on the theory of Hopf monads, which was introduced precisely to this end in \cite{BV2,BV3}.
We briefly recall this theory and state the relevant results of \cite{BV2,BV3}.

\subsection{Coends}\label{sect-coend}
Let $\cc$ and $\dd$ be categories.    A \emph{dinatural transformation}
from  a functor $F\co \cc^\opp
\times \cc \to \dd$  to an object $D$ of $\dd$  is a family $$d=\{d_X \co
F(X,X) \to D\}_{X \in \Ob(\cc)}$$ of morphisms in~$\dd$ such that
$d_X F(f, \id_X)=d_Y F(\id_Y,f)$ for every morphism $f\co X \to Y$
in~$\cc$.   The {\it composition}  of such a $d$ with a morphism $\varphi\co D\to D'$ in $\dd$
is the dinatural transformation $ \varphi\circ d= \{\varphi \circ d_X \co
F(X,X) \to D'\}_{X \in \Ob(\cc)}$ from $F$ to $D'$. A \emph{coend} of~$F$ is a pair $(C,\rho)$ consisting in an object $C$ of $\dd$ and a dinatural transformation $\rho$ from $F$ to $C$ satisfying the following
universality condition:   every dinatural transformation  $d$
from $F$ to  an object   of $\dd$ is the composition of $\rho$ with a morphism   in $\dd$  and the latter
morphism is uniquely determined by $d$. If $F$ has a coend $(C,\rho)$, then  it is unique (up to unique isomorphism).  One writes
$C= \int^{X \in \cc}F(X,X)$.   For more on coends, see
\cite{ML1}.

For a left rigid category $\cc$ (that is, a monoidal category such that every  object $X$ of $\cc$  has a left dual $\leftidx{^\vee}{X}{}$), the formula
$ (X,Y)\mapsto \leftidx{^\vee}{X}{}\otimes Y$ defines a
functor $
\cc^\opp \times \cc \to \cc$. The coend  of this functor (if it exists)    is called the \emph{coend of~$\cc$}.

Consider in more detail the case where  $\cc$ is a fusion category over a commutative ring $\kk$.
Let $I$ be a (finite) representative set of simple objects of $\cc$.
If $\dd$ is a \kt category which admits finite direct sums, then any  \kt linear functor  $F\co \cc^\opp
\times \cc \to \dd$ has a  coend $(C, \rho)$. Here  $C=\oplus_{i \in I} F(i,i)$ and  $\rho=\{\rho_X \co
F(X,X) \to C\}_{X \in \Ob(\cc)}$
 is computed by $\rho_X=\sum_{\alpha } F(q^\alpha_X,p^\alpha_X)$, where $(p_X^\alpha,q_X^\alpha)_{\alpha  }$ is any $I$-partition of $X$. An arbitrary dinatural transformation $d$ from $F$ to  an object $D$ of $\dd$ is the composition of $\rho$ with    $ \sum_{i \in I} d_i \co C \to D$. In particular, $\cc$ has a coend  $ \bigoplus_{i \in I} i^* \otimes
i$.


\subsection{Centralizable functors}\label{sect-centralizer-}
Let  $\cc$ be a left rigid category. A functor $T\co \cc\to \cc $ is \emph{centralizable} if for  every object $X$ of $\cc$, the functor $  \cc^\opp
\times \cc \to \cc$  carrying  any pair $(Y_1,Y_2)$ to $ \leftidx{^\vee}{T}{}(Y_1) \otimes X \otimes Y_2$  has a  coend
\begin{equation*}
Z_T(X)=\int^{Y \in \cc} \leftidx{^\vee}{T}{}(Y) \otimes X \otimes Y.
\end{equation*}
The correspondence $X\mapsto Z_T(X)$  extends to a  functor $Z_T\co \cc\to \cc$,  called the {\it centralizer} of $T$, so that
the associated
universal dinatural transformation
\begin{equation}\label{rhorho}
\rho_{X,Y}\co  \leftidx{^\vee}{T}{}(Y) \otimes X \otimes Y \to Z_T(X)
\end{equation}
is
natural in $X$ and dinatural in~$Y$.

For example, if $\cc$ is a fusion category over $\kk$, then any \kt linear functor $T\co \cc \to \cc$ is centralizable, and its centralizer $Z_T\co \cc \to \cc$ is given by \begin{equation}\label{eq-formulaforzofx} Z_T(X)=\bigoplus_{i \in I}T(i)^* \otimes X \otimes i \end{equation}
for all $X\in \Ob (\cc)$, where $I$ is a representative set of simple objects of $\cc$. The associated universal dinatural transformation   is $$ \rho_{X,Y}=\sum_{\beta}T(q^\beta_Y)^* \otimes \id_X \otimes p^\beta_Y \co T(Y)^* \otimes X \otimes Y \to Z_T(X ),$$ where $(p_Y^\beta, q_Y^\beta)_{\beta  }$ is any $I$-partition of $Y$.

\subsection{Hopf monads}
The Hopf monads generalize Hopf
algebras to an abstract categorical setting. We recall  the basic  definitions of the theory of  Hopf monads referring to \cite{BV2} for a detailed treatment.

Any category $\cc$ gives rise to a   category ${\text {End}}\,(\cc)$ whose objects are  functors   $\cc\to \cc$ and whose morphisms are natural transformations of such functors. The category ${\text {End}}\,(\cc)$  is a (strict) monoidal category with tensor product being composition of functors and unit object being
the identity functor $1_\cc\co \cc\to \cc$.
A \emph{monad} on $\cc$  is an algebra in the category ${\text {End}}\,(\cc)$, that is, a triple $(T,\mu,\eta)$ consisting of a functor
$T\co \cc \to \cc$ and two natural transformations $$\mu=\{\mu_X\co T^2(X) \to T(X)\}_{X \in  \Ob(\cc)}\quad {\text {and}} \quad  \eta=\{\eta_X\co X \to T(X)\}_{X \in \Ob(\cc)}$$  called the \emph{product} and the \emph{unit} of $T$, such that for all $X\in\Ob(\cc)$, $$\mu_X T(\mu_X)=\mu_X\mu_{T(X)} \quad {\text {and}} \quad  \mu_X\eta_{T(X)}=\id_{T(X)}=\mu_X T(\eta_X).$$
For example, the identity functor $1_\cc\co \cc \to \cc $ is a monad on $\cc$ (with identity as product and unit), called the {\it trivial monad}.

Given a monad $T$ on $\cc$, a  $T$\ti module in $\cc$ is a pair $(M,r)$ where $M\in\Ob(\cc)$ and $r\co T(M) \to M$ is a morphism  in $\cc$ such that $r T(r)= r \mu_M$ and $r \eta_M= \id_M$. A morphism from a $T$\ti module  $(M,r)$ to a $T$\ti module $(N,s)$ is a morphism $f \co M \to N$ in $\cc$ such that $f r=s T(f)$.   This defines the {\it category $T\ti \cc$ of $T$-modules in $\cc$}, with composition induced by that in $\cc$. We denote by $U_T$
 the forgetful functor  $T\ti \cc \to \cc$, defined by $U_T(M,r)=M$ and $U_T(f)=f$. Note that $1_\cc\ti \cc=\cc$.

To define Hopf monads, we recall  the notion of a comonoidal functor.
 A functor $F\co C \to D$ between monoidal categories is \emph{comonoidal} if it is endowed with
 a morphism $F_0\co F(\un) \to \un$ and a natural transformation
 $$F_2=\{F_2(X,Y) \co  F(X \otimes Y)\to F(X) \otimes F(Y)\}_{X,Y \in \Ob(\cc)}$$
   which are coassociative and counitary, i.e.,  for all $X,Y,Z\in\Ob(\cc)$, $$(\id_{F(X)} \otimes F_2(Y,Z)) F_2(X,Y \otimes Z)= (F_2(X,Y) \otimes \id_{F(Z)}) F_2(X \otimes Y, Z)$$ and
$$(\id_{F(X)} \otimes F_0) F_2(X,\un)=\id_{F(X)}=(F_0 \otimes \id_{F(X)}) F_2(\un,X).$$ A natural transformation $\varphi=\{\varphi_X\co F(X) \to G(X)\}_{X \in \Ob(\cc)}$ between comonoidal functors   is \emph{comonoidal} if $G_0  \varphi_\un= F_0$ and,  for all $X, Y \in \Ob(\cc)$, $$G_2(X,Y) \, \varphi_{X \otimes Y}= (\varphi_X \otimes \varphi_Y) F_2(X,Y).$$

Let $\cc$ be a monoidal category. A \emph{bimonad} on $\cc$ is a monad $(T,\mu,\eta)$ on $\cc$ such that the underlying functor $T\co \cc \to \cc$ and the natural transformations $\mu$ and $\eta$ are comonoidal.  For a   bimonad $T$ on $\cc$, the category  of $T$\ti modules $T\ti \cc$  has a monoidal structure with unit object $(\un,T_0)$ and monoidal product $$(M,r) \otimes (N,s)=\bigl(M \otimes N, (r \otimes s) \, T_2(M,N)\bigr). $$ Note that the forgetful functor $U_T\co
T\ti \cc \to \cc$ is strict monoidal.

By a rigid category we mean  a monoidal category $\cc$ such that every object $X$ of~$\cc$ has a left dual $\leftidx{^\vee}{}{X}$ and a right dual $X^\vee$.  Every morphism $f\co X\to Y$ in a rigid category gives rise to two dual morphisms   $\leftidx{^\vee}{}{f}\co \leftidx{^\vee}{}{Y}\to \leftidx{^\vee}{}{X}$ and    $ {f}^\vee\co {Y}^\vee\to {X}^\vee$.  A bimonad $T$ on  a rigid category $\cc$ is a \emph{Hopf monad} if its category of modules $T\ti \cc$ is rigid. This condition can be reformulated in terms  of    morphisms  $$s^l=\{s^l_X\co T(\leftidx{^\vee}{T}{}(X)) \to \leftidx{^\vee}{X}{}\}_{X \in \Ob(\cc)}\quad {\text {and}} \quad s^r=\{s^r_X\co T(T(X)^\vee) \to X^\vee\}_{X \in \Ob(\cc)}$$   encoding the left and right duals of any $T$\ti module $(M,r)$ via the formulas $$\leftidx{^\vee}{(}{} M,r)=(\leftidx{^\vee}{M}{}, s^l_M T(\leftidx{^\vee}{r}{})) \quad {\text {and}} \quad (M,r)^\vee=(M^\vee, s^r_M T(r^\vee)).$$
The morphisms $s^l$ and $s^r$ are called the \emph{left and right antipodes}, respectively.
For example, the trivial  monad   on $\cc$ is a Hopf monad (with identity morphisms for
comonoidal structure and antipodes), called the {\it trivial Hopf monad}.

%

\subsection{Distributive laws}\label{sect-dist-law}
Let $(P,m,u)$ and $(T,\mu,\eta)$ be monads on a category $\cc$. Following Beck \cite{Be}, a \emph{distributive law of $T$ over $P$} is a natural transformation $\Omega=\{\Omega_X \co TP(X) \to PT(X)\}_{X \in \Ob(\cc)}$ satisfying appropriate axioms which ensure that the functor $PT\co \cc \to \cc$ is a monad on $\cc$ with product $p$ and unit $e$ given by
\begin{equation*}
p_X=m_{T(X)}P^2(\mu_X)P(\Omega_{T(X)}) \quad \text{and} \quad e_X=u_{T(X)}\eta_X \quad \text{for any} \quad X\in \Ob(\cc).
\end{equation*}
The monad $(PT, p,e)$ is denoted by $P \circ_\Omega T$. A distributive law $\Omega$  of $T$ over $P$ also defines a lift  of  $P$ to a monad $(\Tilde{P}, \Tilde{m},
\Tilde{u})$ on the category $T\ti\cc$ by
\begin{equation*}
\Tilde{P}(M,r)=\bigl(P(M), P(r)\Omega_M\bigr), \quad \Tilde{m}_{(M,r)}=m_M, \quad  \Tilde{u}_{(M,r)}=u_M,
\end{equation*}
and the categories $\Tilde{P}\ti(T\ti\cc)$ and $(P\circ_\Omega T)\ti\cc$ are isomorphic.

If $\cc$ is rigid, $P$ and $T$ are Hopf monads, and $\Omega$ is comonoidal, then $P \circ_\Omega T$ is a Hopf monad on $\cc$, $\Tilde{P}$ is a Hopf monad on $T\ti\cc$, and $\Tilde{P}\ti(T\ti\cc) \simeq (P\circ_\Omega T)\ti\cc $ as monoidal categories (see \cite[Corollary~4.11]{BV3}).

\subsection{Centralizer and double of a Hopf monad}\label{sect-centralizer-double} Let  $\cc$ be a rigid category. We state here some results
of \cite{BV3} concerning the centralizer and the double of a Hopf monad on~$\cc$. As an application, we shall
compute (under certain additional assumptions on $\cc$)  the coend of the center $\zz(\cc)$ of $\cc$.

Let $T\co \cc\to \cc $  be a centralizable Hopf monad on  $\cc$. Then its centralizer  $Z_T\co \cc \to \cc $  has a natural   structure of a Hopf monad on $\cc$, see \cite[Theorems~5.6]{BV3}. Consider the universal dinatural transformation $\rho$  associated with $Z_T$, see  \eqref{rhorho}.   Since $T$ is a Hopf monad on $\cc$, for any $X\in\Ob(\cc)$, the dinatural transformation $$\{T(\rho_{X,Y})\co  T\bigl(\leftidx{^\vee}{}{T}(Y) \otimes X \otimes Y\bigr) \to TZ_T(X)\}_{Y \in \Ob(\cc)}$$ is universal. Therefore there exists a unique morphism $\Omega^T_X\co TZ_T(X) \to Z_TT(X)$ such that, for any $Y \in \Ob(\cc)$,
\begin{equation}\label{def-omega}
\Omega^T_X T(\rho_{X,Y})= \rho_{T(X),T(Y)}\bigl(\leftidx{^\vee}{\mu}{_Y}s^l_{T(Y)}T(\leftidx{^\vee}{\mu}{_Y}) \otimes T_2(X,Y)\bigr)\,T_2(\leftidx{^\vee}{}{T}(Y),X
\otimes Y),
\end{equation}
where $\mu$ and $s^l$ are the  product and the left antipode of $T$.
By \cite[Theorem~6.1]{BV3}, $\Omega^T=\{\Omega^T_X\co TZ_T(X) \to Z_TT(X)\}_{X \in \Ob(\cc)}$ is an invertible comonoidal distributive law, called the \emph{canonical distributive law of $T$ over $Z_T$}. By Section~\ref{sect-dist-law}, this has two consequences.

Firstly, $D_T=Z_T \circ_{\Omega^T} T$ is a Hopf monad on $\cc$ and the rigid categories $D_T \ti \cc $ and $ \zz(T\ti\cc)$ are isomorphic, see \cite[Theorem~6.5]{BV3}. An explicit isomorphism   $\Phi_T\co D_T \ti \cc \to \zz(T\ti\cc)$ carries any morphism in  $D_T \ti \cc$ to itself viewed as a  morphism  in $\zz(T\ti\cc)$ and carries  any   $(M,r)\in \Ob (D_T \ti \cc)$  to   $ \bigl((M,ru_{T(M)}),\sigma\bigr)\in \Ob (\zz(T\ti\cc))$ with $$\sigma_{(N,s)}=(s \otimes r \rho_{T(M),N})(\lcoev_{T(N)} \otimes \eta_M \otimes \id_N) \co M\otimes N \to N\otimes M,$$ where $\eta$ and $u$ are the units of $T$ and   $Z_T$ respectively. The Hopf monad $D_T$ is called the \emph{double of~$T$}. Though we shall not use it, it is important to note that $D_T$ is quasitriangular in an appropriate sense and $\Phi_T\co D_T \ti \cc \to \zz(T\ti\cc)$ is an isomorphism of braided categories.

Secondly, $\Omega^T$ determines a lift of $Z_T$ to a Hopf monad $\Tilde{Z}_T$ on $T\ti \cc$, which turns out to be the centralizer of the identity functor of $T\ti \cc$, see \cite[Theorem~6.9]{BV3}. By the definition of a centralizer,  $\Tilde{Z}_T (\un,T_0)=\bigl(Z_T(\un),Z_T(T_0)\Omega^T_\un\bigr) $ is the coend of~$T\ti \cc$.

Now let $\cc$ be a rigid category such that the trivial Hopf monad $1_\cc $ is centralizable. Clearly,  the centralizer $Z=Z_{1_\cc}$ of $1_\cc$ coincides with the double of~$1_\cc$. Applying the results above to $T=1_\cc$, we obtain that $Z$   is a quasitriangular Hopf monad and $\Phi=\Phi_{1_\cc} \co Z\ti \cc \to \zz(\cc)$ is an isomorphism of braided categories. Assume furthermore that   $Z$ is centralizable. Applying the results above to $T=Z$, we obtain   a morphism   $\Omega=\Omega^Z_{\un} \co ZZ_Z(\un)\to Z_Z  Z(\un)$. Then $\bigl(Z_Z(\un),Z_Z(Z_0)\Omega\bigr)$ is the coend of $Z\ti \cc$. So,
 $$(C,\sigma)=\Phi \bigl(Z_Z(\un),Z_Z(Z_0)\Omega \bigr)\in \Ob (\zz(\cc))$$ is the coend of $\zz(\cc)$. We have $C=Z_Z(\un)$ and, for any $X\in \Ob (\cc)$,
\begin{equation}\label{eq-formulaforsigma}
  \sigma_X=(\id_X \otimes Z_Z(Z_0)\Omega \varrho_{C,X})(\lcoev_{X} \otimes \id_{C \otimes X})\co C \otimes X \to X \otimes C,
\end{equation}
where $\varrho_{X,Y}\co \leftidx{^\vee}{Y}{} \otimes X \otimes Y \to Z(X)$ is the universal dinatural transformation for $Z(X)=\int^{Y \in \cc} \leftidx{^\vee}{Y}{} \otimes X \otimes Y$.

\subsection{The case of fusion categories}\label{sect-fusioncoend} We apply the computations of the previous subsections to  a fusion category $\cc$ over $\kk$. Fix a representative set  $I$ of simple objects of~$\cc$. For $X \in \Ob(\cc)$, denote by $(p_X^\alpha\co X\to i_\alpha ,q_X^\alpha\co i_\alpha \to X)_{\alpha \in \Lambda_X}$ an $I$-partition of~$X$. For $i \in I$, let $\Lambda_X^i$ be the subset of $\Lambda_X$ consisting of all  $\alpha\in \Lambda_X$ such that  $i_\alpha= i$.

By the results above,  the trivial Hopf monad $1_\cc$, being \kt linear,  is centralizable and its centralizer $Z\co \cc \to \cc$ is the  Hopf monad given by Formula~\eqref{eq-formulaforzofx} for $T=1_\cc$, that is, $Z(X)=\bigoplus_{i \in I} i^* \otimes X \otimes i$.  The  structural morphisms of $Z$ can be computed  as follows, see \cite{BV3}.    Let $\psi$ be the pivotal structure of $\cc$ (see Remark~\ref{rem-pivotal}). Then   for any $X,Y\in \Ob(\cc)$, \begin{center}
          $\displaystyle Z_2(X,Y)=\sum_{i \in I}$\,
 \psfrag{i}[Bc][Bc]{\scalebox{.85}{$i$}}
 \psfrag{X}[Bc][Bc]{\scalebox{.85}{$X$}}
 \psfrag{Y}[Bc][Bc]{\scalebox{.85}{$Y$}}
\rsdraw{.45}{.9}{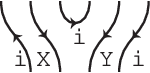}  $\co Z(X\otimes Y) \to Z(X)\otimes Z(Y)$, \\[.5em] $\displaystyle Z_0=\sum_{i \in I}$\,
 \psfrag{i}[Bc][Bc]{\scalebox{.85}{$i$}}
\rsdraw{.45}{.9}{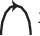} $\co Z(\un) \to \un$, \\[.5em]
$\displaystyle \mu_X=\hspace*{-.2cm}\sum_{\substack{i,j,k \in I \\ \alpha \in \Lambda_{i \otimes j}^k }}$\;
 \psfrag{i}[Bc][Bc]{\scalebox{.85}{$i$}}
 \psfrag{j}[Bc][Bc]{\scalebox{.85}{$j$}}
 \psfrag{k}[Bc][Bc]{\scalebox{.85}{$k$}}
 \psfrag{X}[Bc][Bc]{\scalebox{.85}{$X$}}
 \psfrag{p}[c][c]{\scalebox{.9}{$(q^\alpha_{i \otimes j})^*$}}
 \psfrag{q}[c][c]{\scalebox{.9}{$p^\alpha_{i \otimes j}$}}
\rsdraw{.5}{.9}{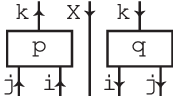}   $\co Z^2(X ) \to Z(X) $, \\[.5em]
 $\displaystyle \eta_X=\id_X \co X \to X=\un^* \otimes X \otimes \un \hookrightarrow Z(X)$, \\[.5em]
$\displaystyle s^l_X=s^r_X=\sum_{\substack{i,j \in I \\ \alpha \in \Lambda_{j^*}^i }}$
 \psfrag{i}[Bc][Bc]{\scalebox{.85}{$i$}}
 \psfrag{u}[Bc][Bc]{\scalebox{.85}{$i^{**}$}}
 \psfrag{j}[Bc][Bc]{\scalebox{.85}{$j$}}
 \psfrag{X}[Bc][Bc]{\scalebox{.85}{$X$}}
 \psfrag{p}[c][c]{\scalebox{.9}{$q^\alpha_{j^*}\psi_i^{-1}$}}
 \psfrag{q}[c][c]{\scalebox{.9}{$p^\alpha_{j^*}$}}
 \,\rsdraw{.35}{.9}{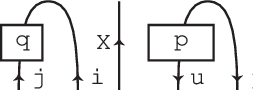} $\co Z(Z(X)^*) \to X^*$.
\end{center}

Being \kt linear, the Hopf monad $Z$ is centralizable and its centralizer $Z_Z \co \cc \to \cc$ is given by Formula~\eqref{eq-formulaforzofx} for $T=Z$, that is, $$Z_Z(X)=\bigoplus_{j \in I} Z(j)^* \otimes X \otimes j\simeq \bigoplus_{i,j \in I} i^* \otimes j^* \otimes i \otimes X \otimes j,$$ with universal dinatural transformation
$$\rho_{X,Y}=\sum_{i,j\in I, \, \alpha \in \Lambda_Y^j} \id_{i^*} \otimes (q_Y^\alpha)^* \otimes \psi_i^{-1} \otimes \id_{X} \otimes p_Y^\alpha\co (Z(Y))^\ast\otimes X\otimes Y\to Z_Z(X).$$
By Section~\ref{sect-centralizer-double}, $\zz(\cc)$ admits a coend $(C,\sigma)$, where
\begin{equation}\label{eq-formulaforc}
C=Z_Z(\un)=\bigoplus_{i,j \in I} i^* \otimes j^* \otimes i \otimes j,
\end{equation}
and $\sigma=\{\sigma_X\co C\otimes X\to X\otimes C\}_{X\in \Ob (\cc)}$ is given by~\eqref{eq-formulaforsigma}.
Using \eqref{def-omega} and the above description of the structural morphisms of $Z$,
we obtain that, for any $X\in \Ob (\cc)$,
$$ \psfrag{i}[Bc][Bc]{\scalebox{.9}{$i$}}
 \psfrag{j}[Bc][Bc]{\scalebox{.9}{$j$}}
 \psfrag{k}[Bc][Bc]{\scalebox{.9}{$k$}}
 \psfrag{l}[Bc][Bc]{\scalebox{.9}{$l$}}
 \psfrag{n}[Bc][Bc]{\scalebox{.9}{$z$}}
 \psfrag{X}[Bc][Bc]{\scalebox{.9}{$X$}}
 \psfrag{a}[cc][cc]{\scalebox{.9}{$p^\beta_{z^* \otimes j \otimes z}$}}
 \psfrag{s}[cc][cc]{\scalebox{.9}{$q^\beta_{z^* \otimes j \otimes z}$}}
 \psfrag{c}[cc][cc]{\scalebox{.9}{$q_{z \otimes k \otimes z^*}^\alpha$}}
 \psfrag{r}[cc][cc]{\scalebox{.9}{$p_{z \otimes k \otimes z^*}^\alpha$}}
  \psfrag{v}[cc][cc]{\scalebox{.9}{$p_X^\gamma$}}
   \psfrag{u}[cc][cc]{\scalebox{.9}{$q_X^\gamma$}}
\sigma_X= \!\!\!\sum_{\substack{i,j,k,l,z \in I \\ \alpha \in \Lambda_{z\otimes k \otimes z^*}^i\\
\beta \in \Lambda_{z^* \otimes j \otimes z}^l \\ \gamma \in \Lambda_X^z }} \; \rsdraw{.50}{.9}{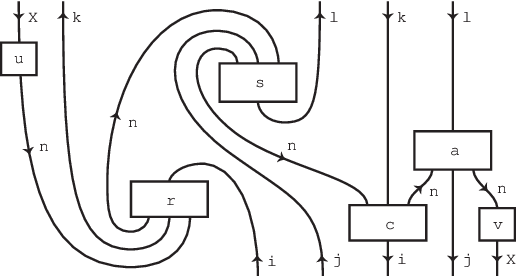}
\, .$$
Using the fact that for  any $X,Y \in \Ob(\cc)$, the family of pairs
\begin{center}
 \psfrag{k}[Bc][Bc]{\scalebox{.9}{$k$}}
 \psfrag{l}[Bc][Bc]{\scalebox{.9}{$l$}}
 \psfrag{Y}[Bc][Bc]{\scalebox{.9}{$Y$}}
 \psfrag{X}[Bc][Bc]{\scalebox{.9}{$X$}}
 \psfrag{a}[Bc][Bc]{\scalebox{.9}{$q_{X^* \otimes l \otimes Y^*}^{\alpha}$}}
 \psfrag{s}[Bc][Bc]{\scalebox{.9}{$p_{X^* \otimes l \otimes Y^*}^{\alpha}$}}
$ \displaystyle \left ( \begin{array}{c} \rsdraw{.45}{.9}{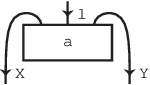}, \;\text{$ \dim(l)$}\, \rsdraw{.45}{.9}{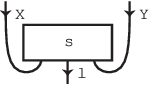}\end{array} \right )_{l \in I, \alpha \in \Lambda^\un_{X^* \otimes l \otimes Y^*}}$
\end{center}
is an $I$-partition of $X \otimes Y$, the half braiding $\sigma$ can be rewritten as
$$ \psfrag{i}[Bc][Bc]{\scalebox{.9}{$i$}}
 \psfrag{j}[Bc][Bc]{\scalebox{.9}{$j$}}
 \psfrag{k}[Bc][Bc]{\scalebox{.9}{$k$}}
 \psfrag{l}[Bc][Bc]{\scalebox{.9}{$l$}}
 \psfrag{n}[Bc][Bc]{\scalebox{.9}{$z$}}
 \psfrag{X}[Bc][Bc]{\scalebox{.9}{$X$}}
 \psfrag{a}[cc][cc]{\scalebox{.9}{$q^\beta_{z^* \otimes j^* \otimes z\otimes l}$}}
 \psfrag{s}[cc][cc]{\scalebox{.9}{$p^\beta_{z^* \otimes j^* \otimes z \otimes l}$}}
 \psfrag{c}[cc][cc]{\scalebox{.9}{$p_{z \otimes k^* \otimes z^*\otimes i}^\alpha$}}
 \psfrag{r}[cc][cc]{\scalebox{.9}{$q_{z \otimes k^* \otimes z^* \otimes i}^\alpha$}}
  \psfrag{v}[cc][cc]{\scalebox{.9}{$p_X^\gamma$}}
   \psfrag{u}[cc][cc]{\scalebox{.9}{$q_X^\gamma$}}
\sigma_X= \!\!\!\!\!\!\sum_{\substack{i,j,k,l,z \in I \\ \alpha \in \Lambda_{z\otimes k^* \otimes z^* \otimes i}^\un\\
\beta \in \Lambda_{z^* \otimes j^* \otimes z \otimes l}^\un \\ \gamma \in \Lambda_X^z }}\!\!\! \!\!\!\dim(i) \dim(l)\; \rsdraw{.50}{.9}{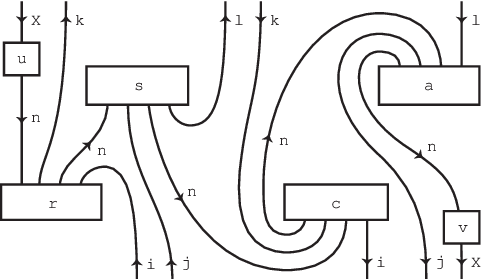}
\, .$$
The latter   formula simplifies  for $X\in I$. In this case  $(\id_X,\id_X)$ is an $I$-partition of~$X$, so the top-left and bottom-right boxes may be deleted from the picture and the summation over $z$ is unnecessary:  only $z=X$ may contribute a non-zero term.

\section{Proof of Lemma~\ref*{le-didim}}\label{proofoflemma}
Since $\kk$ is an algebraically closed field,     $\zz(\cc)$ is a fusion category. By Section~\ref{sect-coend}, the coend of $\zz(\cc)$ is $ \bigoplus_{i\in \mathcal{I}} i^* \otimes i$, where $\mathcal{I}$ is a  representative  set of simple objects of $\zz(\cc)$. By \cite[Chapter IV]{Tu1}, for a closed oriented surface $\Sigma$ of genus $g\geq 0$,
\begin{align*}
\tau_{\zz(\cc)}(\Sigma)&=\bigoplus_{i_1, \ldots,i_g \in \mathcal{I}}\Hom_{\zz(\cc)}(\un_{\zz(\cc)}, i_1 ^* \otimes i_1 \otimes \cdots \otimes i_g^* \otimes i_g)\\
&=\Hom_{\zz(\cc)}\bigl(\un_{\zz(\cc)},\bigoplus_{i_1, \dots,i_g \in \mathcal{I}} i_1 ^* \otimes i_1 \otimes \cdots \otimes i_g^* \otimes i_g\bigr)\\
&=\Hom_{\zz(\cc)}(\un_{\zz(\cc)},(\bigoplus_{i\in \mathcal{I}} i^* \otimes i)^{\otimes g}).
\end{align*}
Therefore Lemma~\ref*{le-didim} is a direct consequence of  the following lemma.
\begin{lem}\label{lem-A-generated} Let $\cc$ be a spherical fusion category over a commutative ring~$\kk$ such that $\dim(\cc)$ is invertible in~$\kk$.
Then for any closed connected oriented surface $\Sigma$ of genus $g\geq 0$,  the $\kk$-module $|\Sigma|_\cc$ is isomorphic to $\Hom_{\zz(\cc)}(\un_{\zz(\cc)},(C,\sigma)^{\otimes g})$, where $(C,\sigma)$ is the coend of $\zz(\cc)$.
\end{lem}
\begin{proof}
For $g=0$, the claim of the lemma follows from the computation of Section~\ref{Esp-S2}. Suppose that $g \geq 1$.
By the definition of the monoidal product of $\zz(\cc)$,  we have $(C,\sigma)^{\otimes g}=(C^{\otimes g},\sigma^g)$, where
$\sigma^g=\{\sigma_X^g \co C^{\otimes g} \otimes X\to X\otimes C^{\otimes g}\}_{X \in \Ob(\cc)}$ is the half braiding defined by
$$
\sigma^g_X=(\sigma_X \otimes \id_{C^{\otimes (g-1)}}) \cdots (\id_{C^{\otimes (g-2)}} \otimes \sigma_X \otimes \id_C)(\id_{C^{\otimes (g-1)}} \otimes \sigma_X).
$$
By Lemma~\ref{lem-proj-alg}, $\Hom_{\zz(\cc)}(\un_{\zz(\cc)},(C,\sigma)^{\otimes g})$ is the image of the involutive endomorphism
$\pi$ of $\Hom_\cc(\un_\cc,C^{\otimes g})$ carrying any  $f\in \Hom_\cc(\un_\cc,C^{\otimes g})$ to
$$
\psfrag{N}[Br][Br]{\scalebox{.8}{$C^{\otimes g}$}} \psfrag{C}[Bl][Bl]{\scalebox{.8}{$C^{\otimes g}$}} \psfrag{f}[Bc][Bc]{\scalebox{.9}{$f$}}\psfrag{i}[Bc][Bc]{\scalebox{.9}{$z$}} \psfrag{c}[Bc][Bc]{\scalebox{1}{$\sigma^g_z$}}
\pi(f)=(\dim(\cc))^{-1}\sum_{z\in I}\dim(z) \; \rsdraw{.45}{.9}{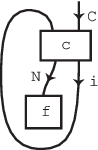},
$$
where $I$ is  a
representative set  of simple objects of $\cc$. Without loss of generality, we can assume that $\un_\cc \in I$. For $\mathbf{i}=(i_1,\ldots,i_g)\in I^g$ and $\mathbf{j}=(j_1,\ldots,j_g)\in I^g$, set $$\dim(\mathbf{i})=\prod_{k=1}^g\dim(i_k)$$ and
$$
V_{\mathbf{i},\mathbf{j}}=\Hom_\cc(\un_\cc,i_1^* \otimes j_1^* \otimes i_1 \otimes j_1 \otimes \cdots \otimes i_g^* \otimes j_g^* \otimes i_g \otimes j_g) .
$$
For an explicit description of the coend $(C, \sigma)$,  see Section~\ref{sect-fusioncoend}. In particular, Formula~\eqref{eq-formulaforc} implies that $$\Hom_\cc(\un_\cc,C^{\otimes g})=\bigoplus_{\mathbf{i},\mathbf{j}\in I^g}V_{\mathbf{i},\mathbf{j}}.$$
Hence $\pi =\sum_{\mathbf{i},\mathbf{j},\mathbf{k},\mathbf{l}\in I^g} \pi_{\mathbf{i},\mathbf{j}}^{\mathbf{k},\mathbf{l}}$, where $\pi_{\mathbf{i},\mathbf{j}}^{\mathbf{k},\mathbf{l}}$ is  a $\kk$-homomorphism
$V_{\mathbf{i},\mathbf{j}}\to V_{\mathbf{k},\mathbf{l}}$.
 Using the expression for  $\sigma$ given at the end of Section~\ref{sect-fusioncoend}, we obtain that for any $f \in V_{\mathbf{i},\mathbf{j}}$,
$$\pi_{\mathbf{i},\mathbf{j}}^{\mathbf{k},\mathbf{l}}(f)=(\dim(\cc))^{-1}\dim(\mathbf{i})\dim(\mathbf{l})\sum_{z\in I}\dim(z) \; \Pi_{\mathbf{i},\mathbf{j}}^{\mathbf{k},\mathbf{l}}(f),$$   where
\begin{equation}\label{coendg}
\psfrag{i}[Bc][Bc]{\scalebox{.9}{$i_1$}}
 \psfrag{j}[Bc][Bc]{\scalebox{.9}{$j_1$}}
 \psfrag{k}[Bc][Bc]{\scalebox{.9}{$k_1$}}
 \psfrag{l}[Bc][Bc]{\scalebox{.9}{$l_1$}}
 \psfrag{u}[Bc][Bc]{\scalebox{.9}{$i_g$}}
 \psfrag{v}[Bc][Bc]{\scalebox{.9}{$j_g$}}
 \psfrag{t}[Bc][Bc]{\scalebox{.9}{$k_g$}}
 \psfrag{d}[Bc][Bc]{\scalebox{.9}{$l_g$}}
 \psfrag{n}[Bc][Bc]{\scalebox{.9}{$z$}}
 \psfrag{a}[cc][cc]{\scalebox{.6}{$\beta_1$}}
 \psfrag{s}[cc][cc]{\scalebox{.6}{$\beta_1$}}
 \psfrag{c}[cc][cc]{\scalebox{.6}{$\alpha_1$}}
 \psfrag{r}[cc][cc]{\scalebox{.6}{$\alpha_1$}}
  \psfrag{o}[cc][cc]{\scalebox{.6}{$\beta_g$}}
 \psfrag{m}[cc][cc]{\scalebox{.6}{$\beta_g$}}
 \psfrag{x}[cc][cc]{\scalebox{.6}{$\alpha_g$}}
 \psfrag{e}[cc][cc]{\scalebox{.6}{$\alpha_g$}}
  \psfrag{g}[cc][cc]{\scalebox{.9}{$f$}}
\Pi_{\mathbf{i},\mathbf{j}}^{\mathbf{k},\mathbf{l}}(f)=\!\!\!\!\!\!\!\!\!\!\sum_{\substack{\alpha_1 \in \lambda_{z\otimes k_1^* \otimes z^* \otimes i_1}\\
\beta_1 \in \lambda_{z^* \otimes j_1^* \otimes z \otimes l_1} \\ \vdots \\ \alpha_g \in \lambda_{z\otimes k_g^* \otimes z^* \otimes i_g}\\
\beta_g \in \lambda_{z^* \otimes j_g^* \otimes z \otimes l_g} }}\;\rsdraw{.50}{.9}{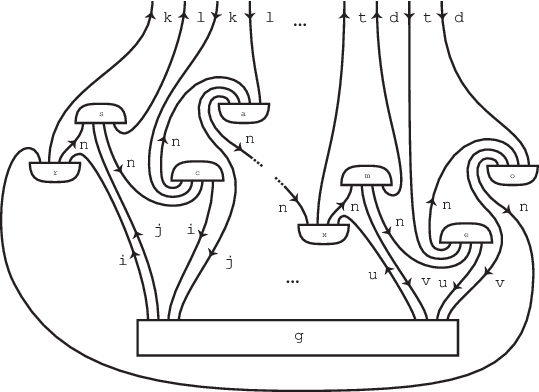}\,.
\end{equation}
Here, for any $X\in\Ob(\cc)$,  we pick an $I$-partition $(p_X^\alpha,q_X^\alpha)_{\alpha \in \Lambda_X}$   of $X$ and denote by $\lambda_X$   the set  of all $\alpha\in \Lambda_X$ such that the target of $p^\alpha_X$ is $\un_\cc$. For $\alpha \in \lambda_X$, the morphisms  $p_X^\alpha\co X\to \un_\cc$ and $q_X^\alpha\co \un_\cc \to X$ are  depicted respectively as
$$
 \psfrag{a}[cc][cc]{\scalebox{.9}{$\alpha$}}  \psfrag{Y}[Bc][Bc]{\scalebox{.9}{$X$}}  \rsdraw{.50}{.8}{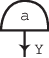} \quad \quad {\text {and}} \quad \quad  \rsdraw{.50}{.8}{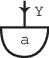} \, .
$$

Let $\Sigma$ be a closed connected oriented surface of genus $g\geq 1$. For $\mathbf{i}=(i_1,\ldots,i_g)\in I^g$ and $\mathbf{j}=(j_1,\ldots,j_g)\in I^g$, consider the following $I$-colored graph $ G_{\mathbf{i},\mathbf{j}}$ on   $\Sigma$:
\begin{equation*}
 \psfrag{i}[Bc][Bc]{\scalebox{.9}{$i_1$}}
 \psfrag{j}[Bc][Bc]{\scalebox{.9}{$j_1$}}
  \psfrag{c}[Bc][Bc]{\scalebox{.9}{$i_2$}}
 \psfrag{g}[Bc][Bc]{\scalebox{.9}{$j_2$}}
  \psfrag{e}[Bc][Bc]{\scalebox{.9}{$i_g$}}
 \psfrag{q}[Bc][Bc]{\scalebox{.9}{$j_g$}}
 G_{\mathbf{i},\mathbf{j}}=
 \rsdraw{.45}{.9}{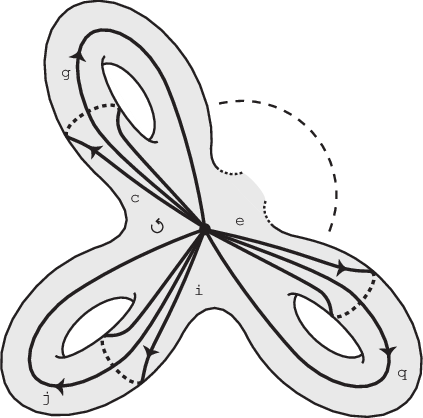}\,.
\end{equation*}
The underlying oriented graph of $G_{\mathbf{i},\mathbf{j}}$, denoted $G$, is a skeleton of $\Sigma$. It has one vertex $x$ of valence $4g$, and its complement in $\Sigma $ is  a disk.
The surface $\Sigma$ can be obtained
from a $4g$-sided polygon
by gluing  pairs of sides, and the graph $G_{\mathbf{i},\mathbf{j}}$ corresponds to the union of radii joining the center of the polygon to the centers of the sides:
\begin{equation*}
 \psfrag{i}[Bc][Bc]{\scalebox{.9}{$i_1$}}
 \psfrag{j}[Bc][Bc]{\scalebox{.9}{$j_1$}}
  \psfrag{k}[Bc][Bc]{\scalebox{.9}{$i_1$}}
 \psfrag{l}[Bc][Bc]{\scalebox{.9}{$j_1$}}
  \psfrag{m}[Bc][Bc]{\scalebox{.9}{$i_2$}}
 \psfrag{n}[Bc][Bc]{\scalebox{.9}{$j_2$}}
  \psfrag{o}[Bc][Bc]{\scalebox{.9}{$i_2$}}
 \psfrag{p}[Bc][Bc]{\scalebox{.9}{$j_2$}}
  \psfrag{q}[Bc][Bc]{\scalebox{.9}{$i_g$}}
 \psfrag{r}[Bc][Bc]{\scalebox{.9}{$j_g$}}
  \psfrag{s}[Bc][Bc]{\scalebox{.9}{$i_g$}}
 \psfrag{t}[Bc][Bc]{\scalebox{.9}{$j_g$}}
  \psfrag{a}[Bc][Bc]{\scalebox{.9}{$a_1$}}
 \psfrag{b}[Bc][Bc]{\scalebox{.9}{$b_1$}}
  \psfrag{c}[Bc][Bc]{\scalebox{.9}{$a_2$}}
 \psfrag{d}[Bc][Bc]{\scalebox{.9}{$b_2$}}
  \psfrag{e}[Bc][Bc]{\scalebox{.9}{$a_g$}}
 \psfrag{f}[Bc][Bc]{\scalebox{.9}{$b_g$}}
 \rsdraw{.45}{.9}{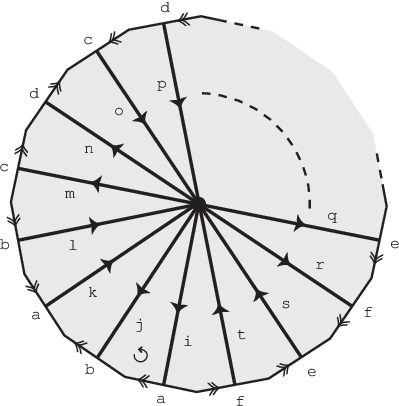}\,.
\end{equation*}
By definition,  the $\kk$-module  $|\Sigma|_\cc$ is isomorphic to the image of the homomorphism  $p(G,G)\co |G; \Sigma|^\circ \to |G; \Sigma|^\circ$, where $$|G; \Sigma|^\circ=\oplus_{\mathbf{i},\mathbf{j}\in I^g} H(G_{\mathbf{i},\mathbf{j}}) .$$  The cone isomorphisms  $\tau_{\mathbf{i},\mathbf{j}}  \co H(G_{\mathbf{i},\mathbf{j}}) =H_x(G_{\mathbf{i},\mathbf{j}})   \to V_{\mathbf{i},\mathbf{j}}$   induce  an isomorphism
$$\tau=\sum_{\mathbf{i},\mathbf{j}\in I^g} \tau_{\mathbf{i},\mathbf{j}} \co |G; \Sigma|^\circ \to \bigoplus_{\mathbf{i},\mathbf{j}\in I^g}V_{\mathbf{i},\mathbf{j}}=  \Hom_\cc(\un,C^{\otimes g}).$$
Consider the automorphism $\kappa=\sum_{\mathbf{i},\mathbf{j}\in I^g} \dim(\mathbf{j})\,  \id_{V_{\mathbf{i},\mathbf{j}}}$ of $\oplus_{\mathbf{i},\mathbf{j}\in I^g} V_{\mathbf{i},\mathbf{j}}$.
We claim that
\begin{equation}\label{proj-conjugate}
 \kappa\tau\, p(G,G)(\kappa\tau)^{-1}=\pi.
\end{equation}
Hence the images of  the projectors $p(G,G)$ and  $\pi$ are isomorphic. This will imply the claim of the lemma.

By definition,  $p(G,G)=\sum_{\mathbf{i},\mathbf{j}, \mathbf{k},\mathbf{l}\in I^g} p_{\mathbf{i},\mathbf{j}}^{\mathbf{k},\mathbf{l}}$, where
$$p_{\mathbf{i},\mathbf{j}}^{\mathbf{k},\mathbf{l}}= |\Sigma \times [0,1], \Sigma \times \{0\}, G_{\mathbf{i},\mathbf{j}}  \times \{0\}, \Sigma  \times \{1\}, G_{\mathbf{k},\mathbf{l}}  \times \{1\}|_\cc  \co H(G_{\mathbf{i},\mathbf{j}}) \to H(G_{\mathbf{k},\mathbf{l}}).$$   To compute  $p_{\mathbf{i},\mathbf{j}}^{\mathbf{k},\mathbf{l}}$, consider the $\cc$-colored graph
$$G_{\mathbf{i},\mathbf{j}}^{ \mathbf{k},\mathbf{l}} =(G_{\mathbf{i},\mathbf{j}}^\opp \times \{0\})  \cup (G_{\mathbf{k},\mathbf{l}} \times \{1\}) \subset \partial (\Sigma\times [0,1]).$$
By definition,
$$ p_{\mathbf{i},\mathbf{j}}^{\mathbf{k},\mathbf{l}}   = \dim(\cc)\left ( \dim(\mathbf{k})\dim(\mathbf{l})\right )^{-1}\Upsilon\bigl(|\Sigma \times [0,1], G_{\mathbf{i},\mathbf{j}}^{ \mathbf{k},\mathbf{l}}|_\cc\bigr).
$$
To compute the right-hand side, consider the 2-polyhedron $$P=(G\times [0,1])\cup (\Sigma \times \{\frac{1}{2}\}) \subset \Sigma \times [0,1].$$  We stratify $P$ by taking as its edges the arcs $x\times [0,\frac{1}{2}]$,  $x\times [\frac{1}{2},1]$, and  the edges of $G\times \{t\}$ for $t\in\{0, \frac{1}{2},1\}$. The polyhedron $P$ has 3 vertices $ x\times \{t\}$  with $t\in\{0, \frac{1}{2},1\}$. We orient the regions of $P$ as shown in the next picture.
\begin{equation*}
 \psfrag{i}[Bc][Bc]{\scalebox{.9}{$k_1$}}
 \psfrag{j}[Bc][Bc]{\scalebox{.9}{$l_1$}}
  \psfrag{k}[Bc][Bc]{\scalebox{.9}{$k_1$}}
 \psfrag{l}[Bc][Bc]{\scalebox{.9}{$l_1$}}
  \psfrag{m}[Bc][Bc]{\scalebox{.9}{$k_2$}}
 \psfrag{n}[Bc][Bc]{\scalebox{.9}{$l_2$}}
  \psfrag{o}[Bc][Bc]{\scalebox{.9}{$k_2$}}
 \psfrag{p}[Bc][Bc]{\scalebox{.9}{$l_2$}}
  \psfrag{q}[Bc][Bc]{\scalebox{.9}{$k_g$}}
 \psfrag{r}[Bc][Bc]{\scalebox{.9}{$l_g$}}
  \psfrag{s}[Bc][Bc]{\scalebox{.9}{$j_g$}}
 \psfrag{t}[Bc][Bc]{\scalebox{.9}{$i_g$}}
  \psfrag{a}[Bc][Bc]{\scalebox{.9}{$i_1$}}
 \psfrag{b}[Bc][Bc]{\scalebox{.9}{$j_1$}}
  \psfrag{c}[Bc][Bc]{\scalebox{.9}{$i_2$}}
 \psfrag{d}[Bc][Bc]{\scalebox{.9}{$j_2$}}
  \psfrag{x}[Bc][Bc]{\scalebox{.9}{$x$}}
 \psfrag{v}[Bc][Bc]{\scalebox{.9}{$v$}}
  \psfrag{u}[Bc][Bc]{\scalebox{.9}{$u$}}
   \psfrag{z}[Bc][Bc]{\scalebox{.9}{$z$}}
 P=\;\rsdraw{.45}{.9}{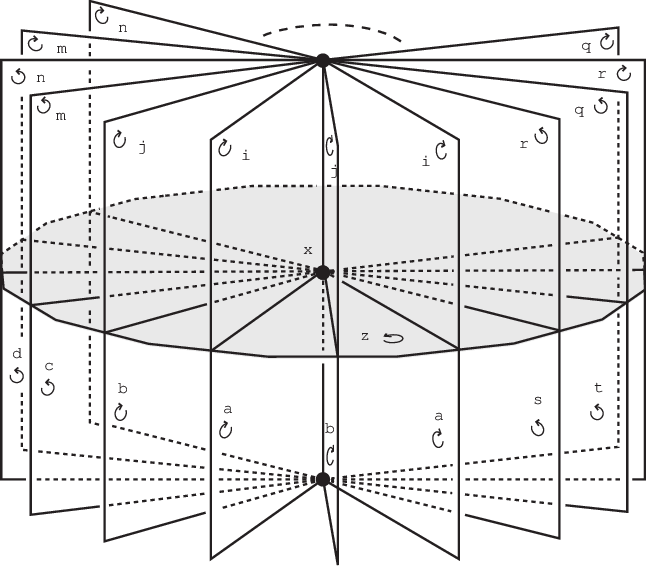}\,.
\end{equation*}
It is clear that $P$ is a skeleton of the pair $(\Sigma \times [0,1],
G_{\mathbf{i},\mathbf{j}}^{ \mathbf{k},\mathbf{l}}  )$.  The maps $  \Reg(P) \to I$ extending the coloring of the boundary are numerated by the color   $z\in I$ of the unique region of $P$ lying in $\Sigma \times \{\frac{1}{2}\}$. The link of the  vertex $ (x,\frac{1}{2})$ of $P$
determines a $\cc$-colored  graph $\Gamma^z$ in $S^2$:
\begin{equation*}
 \psfrag{i}[Br][Br]{\scalebox{.9}{$i_1$}}
 \psfrag{j}[Br][Br]{\scalebox{.9}{$j_1$}}
 \psfrag{k}[Br][Br]{\scalebox{.9}{$k_1$}}
 \psfrag{l}[Br][Br]{\scalebox{.9}{$l_1$}}
 \psfrag{u}[Bl][Bl]{\scalebox{.9}{$i_g$}}
 \psfrag{g}[Bl][Bl]{\scalebox{.9}{$j_g$}}
 \psfrag{c}[Bl][Bl]{\scalebox{.9}{$k_g$}}
 \psfrag{d}[Bl][Bl]{\scalebox{.9}{$l_g$}}
 \psfrag{z}[Bc][Bc]{\scalebox{.9}{$z$}}
 \Gamma^z  =\rsdraw{.45}{.9}{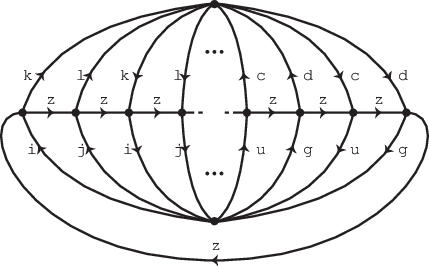}\,.
\end{equation*}
Let $u,v$ be the bottom and the top vertices of $\Gamma^z  $, respectively. Then
\begin{equation*}
|\Sigma \times [0,1], G_{\mathbf{i},\mathbf{j}}^{ \mathbf{k},\mathbf{l}} |_\cc= (\dim (\cc))^{-2}  \prod_{\mathbf{y}\in\{ \mathbf{i},\mathbf{j},\mathbf{k},\mathbf{l}\}}  \dim(\mathbf{y}) \,  \sum_{z \in I}  \dim(z)  \,
 \mu_{\mathbf{i},\mathbf{j}}^{\mathbf{k},\mathbf{l}} (z)
\end{equation*}
where
$$\mu_{\mathbf{i},\mathbf{j}}^{\mathbf{k},\mathbf{l}} (z)= {\ast}_P (\inv_\cc (\Gamma^z  )) \in H_u(\Gamma^z  )^* \otimes H_v(\Gamma^z  )^*.$$ In graphical notation,
\begin{equation*}
 \psfrag{i}[Br][Br]{\scalebox{.9}{$i_1$}}
 \psfrag{j}[Br][Br]{\scalebox{.9}{$j_1$}}
 \psfrag{k}[Br][Br]{\scalebox{.9}{$k_1$}}
 \psfrag{l}[Br][Br]{\scalebox{.9}{$l_1$}}
 \psfrag{u}[Bl][Bl]{\scalebox{.9}{$i_g$}}
 \psfrag{g}[Bl][Bl]{\scalebox{.9}{$j_g$}}
 \psfrag{c}[Bl][Bl]{\scalebox{.9}{$k_g$}}
 \psfrag{d}[Bl][Bl]{\scalebox{.9}{$l_g$}}
 \psfrag{z}[Bc][Bc]{\scalebox{.9}{$z$}}
  \psfrag{e}[Bc][Bc]{\scalebox{.9}{$u$}}
   \psfrag{v}[Bc][Bc]{\scalebox{.9}{$v$}}
  \mu_{\mathbf{i},\mathbf{j}}^{\mathbf{k},\mathbf{l}} (z)= \inv_\cc\left ( \rsdraw{.45}{.9}{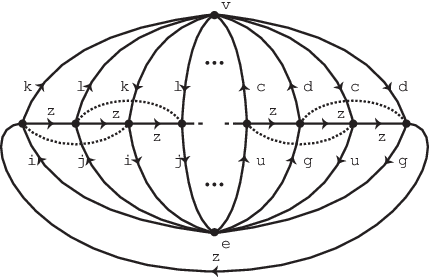} \right ) ,
\end{equation*}
where the dotted lines represent  the tensor contractions. Note that $H_u(\Gamma^z)=H(G_{\mathbf{i},\mathbf{j}})$ and $H_v(\Gamma^z)=H(G_{\mathbf{k},\mathbf{l}}^\opp)$.

Consider the sequence of signed objects of $\cc$
$$S=\bigl((k_1,-),(l_1,-),(k_1,+),(l_1,+), \ldots, (k_g,-),(l_g,-),(k_g,+),(l_g,+) \bigr).$$
Recall the object $X=X_S\in \Ob (\cc)$ defined by  \eqref{Evev-} and the isomorphism $\psi_{S^*}\co X_{S^*} \to X^*$ defined
 in the proof of Lemma~\ref{lem-pairing-non-degen}.  Consider the cone isomorphisms $$\tau_{\mathbf{k},\mathbf{l}}: H(G_{\mathbf{k},\mathbf{l}}) \to V_{\mathbf{k},\mathbf{l}}=\Hom_\cc (\un_\cc, X)
$$ and $$ \tau_{\bar{\mathbf{l}},\bar{\mathbf{k}}}\co  H(G_{\mathbf{k},\mathbf{l}}^\opp) \to V_{\bar{\mathbf{l}},\bar{\mathbf{k}}}=\Hom_\cc (\un_\cc, X_{S^*}),$$
 where $\bar{\mathbf{k}}=(k_g, \ldots,k_1)$ and  $\bar{\mathbf{l}}=(l_g, \ldots,l_1)$.  For $\alpha \in \lambda =\lambda_{X}$, set
$$
a_\alpha=\tau_{\mathbf{k},\mathbf{l}}^{-1}(q_{X}^\alpha)\in H(G_{\mathbf{k},\mathbf{l}}) \quad \text{and} \quad b_\alpha=\tau_{\bar{\mathbf{l}},\bar{\mathbf{k}}}^{-1} (\psi_{S^*}^{-1} \circ (p_{X}^\alpha)^*) \in H(G_{\mathbf{k},\mathbf{l}}^\opp).
$$
Then
$$\Omega=\sum_{\alpha \in \lambda} a_\alpha \otimes b_\alpha\in H(G_{\mathbf{k},\mathbf{l}}) \otimes H(G_{\mathbf{k},\mathbf{l}}^\opp)$$ is the inverse of the contraction pairing $   H(G_{\mathbf{k},\mathbf{l}}^\opp) \otimes H(G_{\mathbf{k},\mathbf{l}}) \to \kk$, cf.\  the proof of Lemma~\ref{lem-prefusion}(a). Therefore, for any $h \in H(G_{\mathbf{i},\mathbf{j}})$,
$$
p_{\mathbf{i},\mathbf{j}}^{\mathbf{k},\mathbf{l}}(h) = \frac{\dim(\mathbf{i})\dim(\mathbf{j})}{\dim(\cc)}\sum_{ {z \in I, \alpha \in \lambda}}  \dim(z)  \,
\mu_{\mathbf{i},\mathbf{j}}^{\mathbf{k},\mathbf{l}}(h\otimes b_\alpha) a_\alpha.
$$
For any $f\in V_{\mathbf{i},\mathbf{j}}$,
$$
\kappa\tau_{\mathbf{k},\mathbf{l}}\, p_{\mathbf{i},\mathbf{j}}^{\mathbf{k},\mathbf{l}} (\kappa\tau_{\mathbf{i},\mathbf{j}})^{-1}(f) = \frac{\dim(\mathbf{i})\dim(\mathbf{l})}{\dim(\cc)}\sum_{ {z \in I, \alpha \in \lambda}}  \dim(z)  \,
\mu_{\mathbf{i},\mathbf{j}}^{\mathbf{k},\mathbf{l}}\bigl(\tau_{\mathbf{i},\mathbf{j}}^{-1}(f) \otimes b_\alpha\bigr) \tau_{\mathbf{k},\mathbf{l}}(a_\alpha).
$$
We have
\begin{align*}
\sum_{\alpha \in \lambda} &\mu_{\mathbf{i},\mathbf{j}}^{\mathbf{k},\mathbf{l}}\bigl(\tau_{\mathbf{i},\mathbf{j}}^{-1}(f) \otimes b_\alpha\bigr) \tau_{\mathbf{k},\mathbf{l}}(a_\alpha)=\sum_{\alpha \in \lambda} \mu_{\mathbf{i},\mathbf{j}}^{\mathbf{k},\mathbf{l}}\bigl(\tau_{\mathbf{i},\mathbf{j}}^{-1}(f) \otimes \tau_{\bar{\mathbf{l}},\bar{\mathbf{k}}}^{-1}(\psi_{S^*}^{-1}\circ (p_{X}^\alpha)^*)\bigr) q_{X}^\alpha\\
&\psfrag{i}[Bc][Bc]{\scalebox{.9}{$i_1$}}
 \psfrag{j}[Bc][Bc]{\scalebox{.9}{$j_1$}}
 \psfrag{k}[Bc][Bc]{\scalebox{.9}{$k_1$}}
 \psfrag{l}[Bc][Bc]{\scalebox{.9}{$l_1$}}
 \psfrag{u}[Bc][Bc]{\scalebox{.9}{$i_g$}}
 \psfrag{v}[Bc][Bc]{\scalebox{.9}{$j_g$}}
 \psfrag{t}[Bc][Bc]{\scalebox{.9}{$k_g$}}
 \psfrag{d}[Bc][Bc]{\scalebox{.9}{$l_g$}}
 \psfrag{n}[Bc][Bc]{\scalebox{.9}{$z$}}
 \psfrag{a}[cc][cc]{\scalebox{.6}{$\beta_1$}}
 \psfrag{s}[cc][cc]{\scalebox{.6}{$\beta_1$}}
 \psfrag{c}[cc][cc]{\scalebox{.6}{$\alpha_1$}}
 \psfrag{r}[cc][cc]{\scalebox{.6}{$\alpha_1$}}
  \psfrag{o}[cc][cc]{\scalebox{.6}{$\beta_g$}}
 \psfrag{m}[cc][cc]{\scalebox{.6}{$\beta_g$}}
 \psfrag{x}[cc][cc]{\scalebox{.6}{$\alpha_g$}}
 \psfrag{e}[cc][cc]{\scalebox{.6}{$\alpha_g$}}
  \psfrag{g}[cc][cc]{\scalebox{.9}{$f$}}
 \psfrag{p}[cc][cc]{\scalebox{.9}{$p_{X}^\alpha$}}
  \psfrag{q}[cc][cc]{\scalebox{.9}{$q_{X}^\alpha$}}
=\!\!\!\!\!\!\!\!\!\!\sum_{\substack{ \alpha_1 \in \lambda_{z\otimes k_1^* \otimes z^* \otimes i_1} \\
\beta_1 \in \lambda_{z^* \otimes j_1^* \otimes z \otimes l_1} \\ \vdots \\ \alpha_g \in \lambda_{z\otimes k_g^* \otimes z^* \otimes i_g}\\
\beta_g \in \lambda_{z^* \otimes j_g^* \otimes z \otimes l_g} \\ \alpha \in \lambda}}\;\rsdraw{.50}{.9}{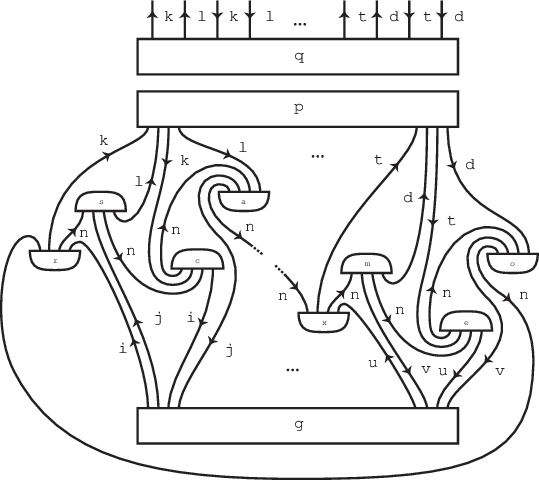}\, .
\end{align*}
Using \eqref{coendg} and the fact that $\sum_{\alpha \in \lambda } p_{X}^\alpha q_{X}^\alpha h=h$ for all   $h\in \Hom_\cc ( \un , X)$, we obtain that the latter expression is equal to $\Pi_{\mathbf{i},\mathbf{j}}^{\mathbf{k},\mathbf{l}}(f)$. Therefore
$$
\kappa\tau_{\mathbf{k},\mathbf{l}}\, p_{\mathbf{i},\mathbf{j}}^{\mathbf{k},\mathbf{l}} (\kappa\tau_{\mathbf{i},\mathbf{j}})^{-1}(f) = \frac{\dim(\mathbf{i})\dim(\mathbf{l})}{\dim(\cc)}\sum_{z \in I}  \dim(z)  \,
\Pi_{\mathbf{i},\mathbf{j}}^{\mathbf{k},\mathbf{l}}(f)=\pi_{\mathbf{i},\mathbf{j}}^{\mathbf{k},\mathbf{l}}(f).
$$
This proves \eqref{proj-conjugate} and  concludes the proof of the lemma.
\end{proof}

\section*{Appendix. The $6j$-symbols}   The tensors known
under the awkward name  $6j$-symbols (see Example~\ref{exa-6j}) play
a special role in this theory because only these tensors are needed
for the computation of the state
  sum  of Section~\ref{sec-computat}
   on   special skeletons of 3-manifolds. We outline
   the basic properties of the $6j$-symbols associated with a  spherical pre-fusion category~$\cc$.
  Pick a   6-tuple $T=( i,\varepsilon_i,
j,\varepsilon_j, k,\varepsilon_k, l,\varepsilon_l, m,\varepsilon_m,
n,\varepsilon_n)$ of signed objects of~$\cc$.
%
Let $\Gamma_T$  be the $\cc$-colored graph in $S^2$ obtained from
the graph $\Gamma\subset S^2$ in Example~\ref{exa-6j} by reversing
the orientation on all edges colored by $x\in\{i,j,k,l,m,n\}$ with
$\varepsilon_x=-$. Then
\begin{equation*}
\begin{vmatrix}
i\,\varepsilon_i& j\,\varepsilon_j& k\,\varepsilon_k\\
l\,\varepsilon_l& m\,\varepsilon_m & n\,\varepsilon_n
\end{vmatrix}=
\inv_\cc(\Gamma_T)\in H(\Gamma_T)^*,
\end{equation*}
is the \emph{6j-symbol} determined by   $T$. Here $H(\Gamma_T)$ is
the (non-ordered) tensor product of the modules $H(m,\varepsilon_m,
i,-\varepsilon_i,n, -\varepsilon_n)$, $H(j ,\varepsilon_j, i
,\varepsilon_i,k ,-\varepsilon_k)$, $H(n,\varepsilon_n,
j,-\varepsilon_j,l,-\varepsilon_l)$, and $H(l,\varepsilon_l,
k,\varepsilon_k,m,-\varepsilon_m)$. In particular,
$
\begin{vmatrix}
i+& j+& k+\\
l+& m+ & n+
\end{vmatrix} =\inv_\cc(\Gamma )$.

The  isotopy invariance of $\inv_\cc$ implies that the 6j-symbols
have the symmetries of an oriented tetrahedron. For example,
\begin{align*}
\begin{vmatrix}
i+& j+& k+\\
l+& m+ & n+
\end{vmatrix}=\begin{vmatrix}
j+& k-& i-\\
m+& n+ & l+
\end{vmatrix}=\begin{vmatrix}
k+& l+& m+\\
n-& i+ & j-
\end{vmatrix}.
\end{align*}

Given    $i,j,k \in \Ob(\cc)$ and signs $\varepsilon, \mu, \nu$, we
have a canonical  (non-degenerate)   pairing
$H({i\varepsilon,j\mu,k\nu}) \otimes H(k(-\nu),j(-\mu),i
(-\varepsilon)) \to \kk$, see
Sections~\ref{sect-non-degen-contraction} and~\ref{sect-graph}. This
form is denoted $\omega_{i\varepsilon,j\mu,k\nu}$ and the
corresponding tensor contraction is denoted
$\ast_{i\varepsilon,j\mu,k\nu}$.

\begin{thm}[The Biedenharn-Elliott identity]\label{thm-bied} Let $I$
be a representative set of simple objects in~$\cc$.  For any
$a,b,c,i,j,k,l,m,n \in \Ob (\cc)$,
\begin{gather*}
\sum_{{\letterf} \in I} \dim(z) \; \ast_{m \mp, k \pm, z \pm} \,\ast_{c \mp, j \mp, z \pm} \, \ast_{b \pm, i \mp, z \pm} \Bigl ( \begin{vmatrix}
i \pm & z \mp & b\pm\\
m \pm & n\pm & k \pm
\end{vmatrix}
 \otimes \begin{vmatrix}
z \pm & j\mp & c \pm\\
l \pm & m \pm & k\pm
\end{vmatrix} \otimes \\
\otimes \begin{vmatrix}
i\pm & j \mp & a \mp\\
c \mp & b\pm & z \mp
\end{vmatrix} \Bigr) = \ast_{a\mp, n\mp, l\pm}\, \Bigl ( \begin{vmatrix}
i \pm &  j\mp &  a\mp\\
l\pm & n \pm  & k\pm
\end{vmatrix} \otimes \begin{vmatrix}
n \pm & l \mp & a \mp\\
c \mp & b \pm & m \mp
\end{vmatrix} \Bigr) .
\end{gather*}
\end{thm}
\begin{proof} Note that the signs of all the objects in this formula may be chosen
independently from each other.  For the upper choice of all the
signs, the graphical proof is given in Figure~\ref{fig-Bied} (for
the opposite choice of some of the signs,   reverse the orientation
of the corresponding edges).   Figure~\ref{fig-Bied} presents
colored $\cc$-graphs in~$S^2$; the equality means the equality of
their $\inv_\cc$-invariants.  The first equality is obtained by
applying  Lemma~\ref{lem-calc-diag}(d) twice (along the dotted
lines) and then   Lemma~\ref{lem-calc-diag}(c). The second equality
follows from the invariance of $\inv_\cc$ under isotopies of the
graph in~$S^2$. The last equality follows from
Lemma~\ref{lem-calc-diag}(d).
\end{proof}
\begin{figure}[h, t]
\begin{center}
\psfrag{i}[Bc][Bc]{\scalebox{.7}{$i$}} \psfrag{x}[Bc][Bc]{\scalebox{.7}{${\letterf}$}}
\psfrag{a}[Bc][Bc]{\scalebox{.7}{$a$}} \psfrag{b}[Bc][Bc]{\scalebox{.7}{$b$}}
\psfrag{k}[Bc][Bc]{\scalebox{.7}{$k$}} \psfrag{c}[Bc][Bc]{\scalebox{.7}{$c$}}
\psfrag{l}[Bc][Bc]{\scalebox{.7}{$l$}} \psfrag{m}[Bc][Bc]{\scalebox{.7}{$m$}}
\psfrag{j}[Bc][Bc]{\scalebox{.7}{$j$}} \psfrag{n}[Bc][Bc]{\scalebox{.7}{$n$}}
$\displaystyle \sum_{{\letterf} \in I} \dim({\letterf})$ \,\rsdraw{.45}{.9}{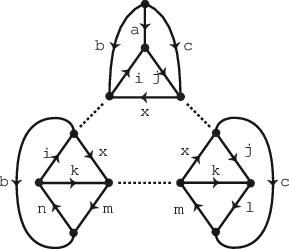} \\ \phantom{XXXXXXXXXXXX} = \rsdraw{.45}{.9}{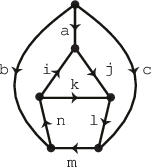}
= \, \rsdraw{.45}{.9}{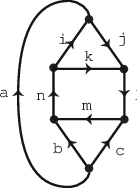}\; = \rsdraw{.45}{.9}{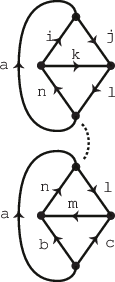}
\end{center}
\caption{Proof of the Biedenharn-Elliot identity}
\label{fig-Bied}
\end{figure}

\begin{thm}[The orthonormality relation]\label{thm-ortho}
For any  objects  $i,k,l,m,m',n$  of~$\cc$ such that $m,m'$ are
simple and for any~$I$ as above,
\begin{gather*}
\dim(m)\sum_{{\letterf} \in I} \dim({\letterf}) \ast_{z \pm,i\pm,k\mp} \ast_{z\pm,n\mp,l\pm}  \,\Bigl ( \begin{vmatrix}
i\pm & z\pm & k \pm \\
l \pm & m \pm & n\pm
\end{vmatrix} \otimes \begin{vmatrix}
z\mp & i\mp & k\mp \\
m'\pm & l \pm & n\pm
\end{vmatrix} \Bigr)\\
= \delta_{m,m'} \; \omega_{m\mp,n\pm,i\pm} \otimes \omega_{l\mp,m\pm,k\mp}.
\end{gather*}
\end{thm}

The proof is given in~Figure~\ref{fig-ortho}.

\begin{figure}[h, t]
\begin{center}
\psfrag{i}[Bc][Bc]{\scalebox{.7}{$i$}} \psfrag{x}[Bc][Bc]{\scalebox{.7}{${\letterf}$}}
\psfrag{w}[Bc][Bc]{\scalebox{.7}{$m'$}}
\psfrag{k}[Bc][Bc]{\scalebox{.7}{$k$}}
\psfrag{l}[Bc][Bc]{\scalebox{.7}{$l$}} \psfrag{m}[Bc][Bc]{\scalebox{.7}{$m$}}
 \psfrag{n}[Bc][Bc]{\scalebox{.7}{$n$}}
$\displaystyle \sum_{{\letterf} \in I} \dim({\letterf})$ \,\rsdraw{.45}{.9}{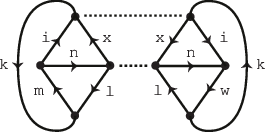}  \\ = \, \rsdraw{.45}{.9}{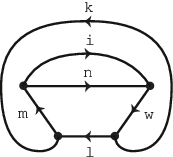}\,
$= \, \delta_{m,m'}\, \dim(m)^{-1}$ \rsdraw{.45}{.9}{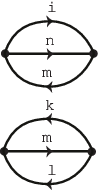}
\end{center}
\caption{Proof of the orthonormality relation}
\label{fig-ortho}
\end{figure}


\end{document}